\documentclass[12pt,centertags]{amsart}
\usepackage{amsmath,amstext,amsthm,a4,amssymb,amscd}
\usepackage[mathscr]{eucal}
\usepackage{mathrsfs}
\usepackage{epsf}
\numberwithin{equation}{section}

\def\mE{\mathcal{E}}
\def\mF{\mathcal{F}}

\def\mM{\mathcal{M}}

\newtheorem{thm}{Theorem}[section]
\newtheorem{lemma}[thm]{Lemma}
\newtheorem{prop}[thm]{Proposition}
\newtheorem{cor}[thm]{Corollary}
\theoremstyle{definition}
\newtheorem{rem}[thm]{Remark}
\theoremstyle{definition}
\newtheorem{defn}[thm]{Definition}
\newcommand{\be}{\begin{eqnarray}}
\newcommand{\ee}{\end{eqnarray}}

\newcommand{\comment}[1]{}

\begin{document}

\title{Dirac operators on foliations:  the Lichnerowicz inequality
}
\author{Weiping Zhang}

\address{Chern Institute of Mathematics \& LPMC, Nankai
University, Tianjin 300071, P.R. China}
\email{weiping@nankai.edu.cn}

\begin{abstract}    
We construct Dirac operators on foliations by applying the Bismut-Lebeau analytic localization technique to the Connes fibration over a foliation. The Laplacian of the resulting  Dirac operators has better lower bound than that obtained by using the usual adiabatic limit arguments on the original foliation. As a consequence, we prove an extension of the Lichnerowicz-Hitchin vanishing  theorem to the case of foliations.
\end{abstract}

\maketitle
\tableofcontents

\setcounter{section}{-1}

\section{Introduction} \label{s0}
  
Let $D$ be the canonical  Dirac operator on a closed spin Riemannian  manifold. Then the standard   Lichnerowicz  formula  \cite{L63} states that 
$D^2=-\Delta+\frac{k}{4}$,
where $\Delta$ is the associated Bochner Laplacian and $k $ is the scalar curvature of the Riemannian manifold. Moreover,   $-\Delta $ is nonnegative and one has the classical inequality
\begin{align}\label{002}
D^2\geq \frac{k}{4}.
\end{align}
The purpose of this  paper is  to   generalize of this {\it Lichnerowicz inequality} to the
case of foliations.

\comment{

Let
$\widehat{\mathcal A}(M)$ be defined by that if $\dim M=8k+4i$
with $i=0$ or $1$, then $\widehat{\mathcal
A}(M)=\frac{3-(-1)^{i+1}}{4}\widehat{A}(M)$, where $\widehat{A}(M)$ is  the Hirzebruch
$\widehat{A}$-genus\footnote{Cf. \cite[pp. 13]{Z01} for a
definition.}; if $\dim M=8k+i$ with
$i=1$ or $2$, then $\widehat{\mathcal A}(M)\in {\bf Z}_2$ is the
Atiyah-Milnor-Singer $\alpha$ invariant\footnote{Cf. \cite[Section
2.7]{LaMi89} for a definition.}; while in other dimensions one
takes  $\widehat{\mathcal A}(M)=0$.

A classical  theorem of Lichnerowicz \cite{L63} states that if a
closed spin manifold of dimension $4k$  admits a Riemannian metric
of positive scalar curvature, then $\widehat{A}(M)=0$.   Hitchin \cite{Hi74}
extended it to the case of  $\widehat{\mathcal A}(M)$ in all dimensions. 

}

To be more precise, let $M$ be a smooth manifold, let $F $ be an
integrable subbundle of the tangent vector bundle $TM$ of $M$. Let
$g^F$ be a Euclidean metric on $F$. Then $g^F$ determines a
leafwise scalar curvature $k^F\in C^\infty(M)$ as follows: for any $x\in M$,   the integrable subbundle
$F$ determines a leaf ${\mathcal F}_x$ passing through $x$ such
that $F|_{{\mathcal F}_x}=T{\mathcal F}_x$. Thus, $g^F$ determines
a Riemannian metric on ${\mathcal F}_x$. Let $k^{{\mathcal F}_x}$
denote the scalar curvature of this Riemannian metric. We
define
\begin{align}\label{001}
k^F(x)=k^{{\mathcal F}_x}(x).
\end{align}

On the other hand, let $F^\perp\simeq TM/F$ be a subbundle of $TM$ which is transversal to $F$.\footnote{In what follows, we   identify $F^\perp$ with $TM/F$.} Let $g^{F^\perp}$ be a Euclidean metric on $F^\perp$. Then we get a Riemannian metric $g^{TM}$ on $TM$ so that we have an orthogonal splitting
\begin{align}\label{14.0.1}
TM=F\oplus F^\perp,\ \ \ \ \ \ \ \ g^{TM}=g^F\oplus g^{F^\perp}.
\end{align}

Let $\nabla^B$ be the Bott connection \cite{Bo70} on $TM/F\simeq F^\perp$. Set\footnote{Equivalently,  for any $X\in\Gamma( TM)$,   $U,\, V\in \Gamma(F^\perp)$, one has
$
\langle\omega(X)U,V\rangle=X\langle U,V\rangle-\left\langle \nabla^B_XU,V\right\rangle-\left\langle U,\nabla^B_XV\right\rangle.
$
}
\begin{align}\label{14.0.1b}
\omega=\left(g^{F^\perp}\right)^{-1}\nabla^Bg^{F^\perp}.
\end{align}

Now we assume that $M$ is spin.    Let $f_1,\,\cdots,\,f_q$ (resp. $h_1,\,\cdots,\,h_{q_1}$) be an orthonormal basis of $(F,g^F)$ (resp. $(F^\perp,g^{F^\perp})$). 

The main result of this paper can be stated as follows.

\begin{thm}\label{t0.2} Let $F $ be an  integrable subbundle  of the
tangent  bundle   of   a closed spin manifold $M$ as above. Then for any $c>0$, there is a formally self-adjoint Dirac type operator $D_c$   on $M$, which can be constructed canonically,\footnote{See (\ref{175n}) for a more precise form.}  such that the following inequality holds,
  \begin{align}\label{14.0.3}
 D^2_c+c\geq \frac{1}{4}\left( {k^F}-\frac{1}{4}\sum_{i=1}^q\sum_{s=1}^{q_1}|\omega(f_i)h_s|^2\right).
\end{align}
\end{thm}


\begin{cor}\label{t14.1}
Let $F $ be an  integrable subbundle  of the
tangent  bundle   of   a closed spin manifold $M$. Then if there is a metric $g^{TM}$ of form (\ref{14.0.1}) such that
\begin{align}\label{14.0.2}
    {k^F}-\frac{1}{4}\sum_{i=1}^q\sum_{s=1}^{q_1}|\omega(f_i)h_s|^2>0
\end{align}
over $M$, 
    one has $\widehat{\mathcal A}(M)=0$, where $\widehat{\mathcal A}(M)\in KO^{\dim M}({\rm pt.})$ is the canonical $KO$-characteristic number of $M$.\footnote{Cf. \cite[Section II.7] {LaMi89} for a definition of $\widehat{\mathcal A}(M)$.}
\end{cor}

When taking $F=TM$, Corollary \ref{t14.1} recovers  the classical vanishing theorems of Lichnerowicz \cite{L63} and Hitchin \cite{Hi74}) (cf. \cite[Theorem II.8.12]{LaMi89}). 

By (\ref{002}), a natural possible way to  prove Theorem \ref{t0.2} is to compute the scalar curvature $k^{TM,\varepsilon}$ of the metric $g^{TM}_\varepsilon=g^F\oplus \frac{1}{\varepsilon^2}g^{F^\perp}$, when $\varepsilon>0$ tends to zero.  An explicit formula for $k^{TM,\varepsilon}$ under the adiabatic limit $\varepsilon\rightarrow 0$ is included in Appendix \ref{sA}, from which one sees that the condition in (\ref{14.0.2}) is   cleaner than what one would expect from $k^{TM,\varepsilon}$ (cf. (\ref{a1.9})). Indeed, even in the codimension one case, the bound $-\frac{1}{4}$ in  (\ref{14.0.2}) is   better than what one would expect from $k^{TM,\varepsilon}$, which is $-\frac{3}{4} $ (cf. (\ref{a1.12a})){.

\begin{rem}\label{t002a}
 Corollary  \ref{t14.1}   maybe thought of   as a non-existence result. For example, take any $8k+1$ dimensional closed spin manifold $M$ such that $\widehat{\mathcal A}(M)\neq 0$. Then by a result of Thurston \cite{T76}, there always exists a codimension one foliation on $M$. However, by our result, there is no metric on   $TM$ verifying  (\ref{14.0.2}). 
\end{rem}

Our original motivation, dating back to  \cite{LZ01},  is to look for  a purely geometric understanding of the following 
 celebrated vanishing theorem of Connes, where instead of
assuming    $TM$   being  spin, one assumes that $F$ is spin .

 \begin{thm}\label{t0.1} {\bf (Connes \cite{Co86})}   Let  $F$ be  a spin integrable subbundle of  the tangent   bundle of a  compact
oriented manifold $M$. If there is a metric  $g^F$ on $F$
such that $k^F>0$ over $M$, then  $\widehat{A}(M)=0$.
\end{thm}

Clearly, if one assumes that $\dim M=4k$ and that $TM$ is also spin, then Theorem \ref{t0.1} is stronger than Corollary \ref{t14.1} in this case.

Recall that the proof outlined    in \cite{Co86} for Theorem
\ref{t0.1} uses in an essential way the noncommutative geometry. It is based on the
Connes-Skandalis longitudinal index theorem for foliations
\cite{CoS84} as well as the techniques of cyclic cohomology. Thus
it relies  on the spin structure on $F$.  
Also, it does not cover the $\dim M= 8k+i $ ($i=1,\,2$) cases. 

Our main result, as stated in Theorem \ref{t0.2}, concerns   concrete Dirac type operators on $M$. It gives more information (like eigenvalue estimates) than just the index.

The construction of the Dirac type  operator in Theorem \ref{t0.2}  makes use  of the important  geometric trick in
\cite{Co86}, which is the construction of a
fibration\footnote{Which will be called a Connes fibration in what
follows.} over an arbitrary foliation. The key advantage of this fibration is that the lifted (from the original) foliation is
almost isometric, i.e., very close to the Riemannian foliation (which corresponds to the $\omega=0$ case). On the other hand, this fibration is
noncompact, which makes the proof of   Theorem \ref{t0.2}  highly nontrivial.

Roughly speaking, the Connes fibration over a foliation 
$(M,F)$ is a fibration $\pi:\mM\rightarrow M$ where for any $x\in
M$, the fiber $\pi^{-1}(x)$ is the space of Euclidean metrics on
the quotient space $T_xM/F_x$. The integrable subbundle $F$ of
$TM$ lifts to an integrable subbundle $\mF$ of $T\mM$, and
$(\mM,\mF)$ carries  an almost isometric structure  in the sense
of  \cite[Section 4]{Co86}.  Take any metric on the transverse
bundle $TM/F$, which by definition determines an embedded section
$s:M\hookrightarrow \mM$. 

\comment{

The induced fibration
$s\circ\pi:\mM\rightarrow s(M)$ looks like a vector bundle, and
Connes obtained his theorem by examining the corresponding
Riemann-Roch property in   noncommutative frameworks.

}

Our proof of Theorem \ref{t0.2}   applies 
the   analytic localization techniques,  as developed by
Bismut-Lebeau \cite[Sections 8 and 9]{BL91}, to the embedding $s:M\hookrightarrow \mM$, and can be thought of
as a kind of transgression.  

To be more precise, let $T^V\mM$ be the vertical tangent bundle of
the Connes fibration $\pi:\mM\rightarrow M$.  
 Taking a splitting
$
 T {\mM}= {\mF}\oplus T^V {\mM}\oplus
  {\mF}^\perp,
$
then $T^V\mM$   (resp. $ {\mF}^\perp\simeq
\pi^*(TM/F)$)  carries a natural metric $g^{T^V\mM}$ 
(resp. $g^{ {\mF}^\perp}$).
 If
one lifts $g^F$ to a metric $g^{ {\mF}}$ on
$ {\mF}$, then for any $\beta>0$, $\varepsilon>0$, one can
consider the rescaled metric
$g_{\beta,\varepsilon}^{T {\mM}}=\beta^2g^{ {\mF}}\oplus
 g^{T^V {\mM}}\oplus
\frac{g^{ {\mF}^\perp}}{\varepsilon^2}$.

Since $TM$ is assumed to be spin,
  $ {\mF}\oplus {\mF}^\perp\simeq
 {\pi}^*(TM)$ is also spin. Thus one can construct a Dirac
type operator\footnote{Called a sub-Dirac operator in
\cite{LZ01}.} $D_{\beta,\varepsilon}^{ {\mM}}$ acting on
$\Gamma(S( {\mF}\oplus {\mF}^\perp)\otimes
\Lambda^*(T^V {\mM}))$, where $S(\cdot)$ (resp.
$\Lambda^*(\cdot)$) is the notation for spinor bundle (resp.
exterior algebra bundle).

Now take a  sufficiently small open neighborhood $U$ of
$ {s}(M)$ in $  {\mM}$. Inspired by \cite{BL91}, for
any $\beta,\ \varepsilon,\ T>0$, we construct an isometric
embedding (see Section 2 for more details)
\begin{align}\label{003}
J_{T,\beta,\varepsilon}:\Gamma\left(\left. S
\left( {\mF}\oplus {\mF}^\perp\right)
\right|_{ {s}(M)}\right)\rightarrow
\Gamma\left(S\left( {\mF}\oplus {\mF}^\perp\right)\otimes
\Lambda^*\left(T^V {\mM}\right)\right)
\end{align}
 such
that for any
$\sigma\in\Gamma(S( {\mF}\oplus {\mF}^\perp)|_{ {s}(M)}
)$, $J_{T,\beta,\varepsilon}\sigma$ has compact support in $U$.
Let $E_{T,\beta,\varepsilon}$ be the $L^2$-completion of the image
space of $J_{T,\beta,\varepsilon}$. Let
$p_{T,\beta,\varepsilon}:L^2(S( {\mF}\oplus {\mF}^\perp)\otimes
\Lambda^*(T^V {\mM}))\rightarrow E_{T,\beta,\varepsilon}$
be the orthogonal projection. Then one finds that the operator
\begin{align}\label{004}
J_{T,\beta,\varepsilon}^{-1}p_{T,\beta,\varepsilon}
D_{\beta,\varepsilon}^{ {\mM}}
J_{T,\beta,\varepsilon}:\Gamma\left(\left. S
\left( {\mF}\oplus {\mF}^\perp\right)
\right|_{ {s}(M)}\right)\rightarrow \Gamma\left(\left. S
\left( {\mF}\oplus {\mF}^\perp\right)
\right|_{ {s}(M)}\right)
\end{align}
 is elliptic, formally self-adjoint and homotopic to the Dirac
operator on $ {s}(M)\simeq M$. Thus Theorem \ref{t0.2} will
follow if one can show that for certain values of $\beta$,
$\varepsilon$ and $T$, this operator verifies the estimate in (\ref{14.0.3}). Indeed, this
is exactly what we will establish  in this paper.

We would  like to mention that the idea of constructing sub-Dirac
operators has also been used in \cite{LMZ01} to prove a
generalization of the Atiyah-Hirzebruch vanishing theorem for circle
actions \cite{AH} to the case of foliations.

This paper is organized as follows. In Section 1, we discuss the
case of almost isometric foliations and carry out the local
computation. We also introduce the sub-Dirac operator in this section. In Section 2, we work on   noncompact Connes
  fibrations and carry out the proof of Theorem \ref{t0.2}. There is also an Appendix \ref{sA} where we include a caculation of the adiabatic limit behaviour of the scalar curvature on a foliation.

$\ $

\noindent{\bf Acknowledgements} The author  is indebted  to Kefeng
LIU for sharing his ideas in the joint work \cite{LZ01} and for
many related discussions.  The author is also grateful to Huitao
FENG,  Xiaonan MA and Yong WANG for  many helpful suggestions. We would also like to thank the referees of this paper for many helpful suggestions. This work was
partially supported by MOEC and NNSFC.

\section{Adiabatic limit and almost isometric  foliations}\label{s1}

In this section, we discuss the geometry of almost isometric
foliations in the sense of Connes \cite{Co86}. We introduce for
this kind of foliations a rescaled  metric and show that 
the leafwise scalar curvature shows up
from the limit behavior of the rescaled  scalar curvature.
We also introduce in this setting the sub-Dirac operators inspired
by the original construction given in \cite{LZ01}.
Finally, by combining the above two procedures, we prove a vanishing
result when the almost isometric foliation under discussion is
compact.

This section is organized as follows. In Section \ref{s1.1}, we
recall the definition of the almost isometric foliation in the
sense of Connes. In Section \ref{s1.2} we introduce a rescaling of
the given metric on the almost isometric foliation  and study the
corresponding limit behavior of the scalar curvature. In Section
\ref{s1.3}, we study  Bott type connections on certain bundles
transverse  to the integrable subbundle. In Section \ref{s1.4}, we
introduce the so called sub-Dirac operator  and compute the
corresponding Lichnerowicz type formula. In Section \ref{s1.5} we
prove a vanishing result when the almost isometric foliation is
compact and verifies the conditions in Theorem \ref{t0.1}.

\subsection{Almost isometric foliations
}\label{s1.1}

Let $(M,F)$ be a foliated manifold, where $F$ is an integrable
subbundle of $TM$, i.e., for any smooth sections  $X,\ Y\in\Gamma(F)$, one has
\begin{align}\label{1.0}
[X,Y]\in \Gamma(F).
\end{align}

 Let $G$ be the holonomy groupoid of
$(M,F)$ (cf.
  \cite{W83}).

Let $TM/F$ be the transverse  bundle. We make the assumption that
there is a proper subbundle $E$ of $TM/F$ and choose a splitting
\begin{align}\label{1.1} TM/F=E\oplus (TM/F)/E.
\end{align}
Let $q_1$, $q_2$ denote the ranks of $E$ and $(TM/F)/E$
respectively.

\begin{defn}\label{t1.1} {\bf (Connes \cite[Section 4]{Co86})} If there exists a metric $g^{TM/F}$ on
$TM/F$ with its restrictions to $E$ and $(TM/F)/E$ such that the
action of $G$ on $TM/F$ takes the form \begin{align}\label{1.2}
\left(
\begin{array}{cc}
O(q_1) & 0 \\
A & O(q_2)%
\end{array}%
\right),
\end{align}
where $O(q_1)$, $O(q_2)$ are orthogonal matrices of ranks $q_1$,
$q_2$ respectively, and $A$ is a $q_2\times q_1$ matrix, then we say
that $(M,F)$ carries an almost isometric structure.
\end{defn}

Clearly, the existence of the almost isometric structure does not
depend on the splitting (\ref{1.1}).
We assume from now on that $(M,F)$ carries an almost isometric structure as
above.

Now choose a splitting 
$TM =F\oplus F^\perp.$
We can and we will identify $TM/F$ with $F^\perp$. Thus $E$ and
$(TM/F)/E$ are identified with subbundles $F^\perp_1$, $F^\perp_2$
of $F^\perp$  respectively.

Let $g^F$ be a metric on $F$. Let $g^{F^\perp}$ be the metric on
$F^\perp$ corresponding to the metric $g^{TM/F}$ and let
$g^{F^\perp_1}$, $g^{F^\perp_2}$ be the restrictions of
$g^{F^\perp}$ to $F^\perp_1$, $F^\perp_2$.

Let $g^{TM}$ be a metric on $TM$ so that we have the orthogonal
splitting
\begin{align}\label{1.4}
       TM =  F\oplus F^\perp_1\oplus F^\perp_2,\ \ \ \ \ \ \
g^{TM} = g^{F}\oplus g^{F^\perp_1}\oplus
g^{F^\perp_2}.
\end{align}
Let $\nabla^{TM}$ be the Levi-Civita connection associated to
$g^{TM}$.

From the almost isometric condition (\ref{1.2}), one deduces that
for any $X\in \Gamma(F)$, $U_i,\, V_i\in\Gamma(F^\perp_i)$, $i=1,\
2$, the following identities, which may be thought of as infinitesimal versions of (\ref{1.2}),
 hold (cf. \cite[(A.5)]{LZ01}):
\begin{align}\label{1.5} \begin{split}
       \langle [X,U_i],V_i\rangle+\langle U_i,[X,V_i]\rangle =X\langle U_i,V_i\rangle,\\
\langle [X,U_2],U_1\rangle =0.\end{split}
\end{align}
Equivalently,
\begin{align}\label{1.6} \begin{split}
       \left\langle X,\nabla^{TM}_{U_i}V_i + \nabla^{TM}_{ V_i}U_i\right\rangle =0,\\
\left\langle \nabla^{TM}_XU_2 ,U_1\right\rangle+\left\langle
X,\nabla^{TM}_{U_2}U_1\right\rangle =0.\end{split}
\end{align}

In this paper, for simplicity, we also make the following assumption.
This assumption holds by the Connes   fibration to be dealt with in the next section.

\begin{defn}\label{t1.1c}  We call an almost isometric foliation as above verifies Condition ({C}) if  $F_2^\perp$ is also integrable. That is, for any $U_2,\, V_2\in\Gamma(F_2^\perp)$, one has
\begin{align}\label{1.6a}
\left[U_2,V_2\right]\in \Gamma\left(F_2^\perp\right).
\end{align}
\end{defn}

\subsection{Adiabatic limit and the scalar curvature
}\label{s1.2} It has been shown in \cite[Proposition A.2]{LZ01} that
an almost isometric foliation in the sense of Definition \ref{t1.1} is an almost Riemannian foliation in
the sense of \cite[Definition 2.1]{LZ01}. Thus many computations in
what follows are contained implicitly in  \cite{LZ01} (see also
\cite{LW09}).


For convenience, we   recall the standard formula for the
Levi-Civita connection that for any $X,\ Y,\ Z\in\Gamma(TM)$,
\begin{multline}\label{1.7}
2\left\langle \nabla^{TM}_XY,Z\right\rangle =X\langle Y,Z\rangle
+Y\langle X,Z\rangle-Z\langle X,Y\rangle\\
+\langle [X,Y],Z\rangle -\langle [X,Z],Y\rangle-\langle
[Y,Z],X\rangle.
\end{multline}

For any $\beta,\ \varepsilon>0$, let
$g_{\beta,\varepsilon}^{TM}$ be the rescaled  Riemannian metric on
$TM$ defined by
\begin{align}\label{1.8}
g^{TM}_{\beta,\varepsilon }= \beta^2 g^{F}\oplus \frac{1}{
\varepsilon^{2}} g^{F^\perp_1}\oplus    g^{F^\perp_2}.
\end{align}
We will always assume that $0<\beta,\ \varepsilon\leq 1$.
We will use the subscripts and/or superscripts ``$\beta$, $\varepsilon$'' to
decorate the geometric data associated to $g^{TM}_{\beta,\varepsilon }$.
For example, $\nabla^{TM,\beta,\varepsilon}$ will denote the Levi-Civita
connection associated to $g^{TM}_{\beta,\varepsilon} $. When the
corresponding notation does not involve ``$\beta,\ \varepsilon$'', we will
mean that it corresponds to the case of $\beta=\varepsilon=1$.

Let $p$, $p_1^\perp$, $p_2^\perp$ be the orthogonal projections from
$TM$ to $F$, $F_1^\perp$, $F^\perp_2$ with respect to the orthogonal
splitting (\ref{1.4}).
Let $\nabla^{F,\beta,\varepsilon}$, $\nabla^{F_1^\perp,\beta,\varepsilon}$,
$\nabla^{F_2^\perp,\beta,\varepsilon}$ be the Euclidean connections on
$F$, $F_1^\perp$, $F^\perp_2$ defined by
\begin{align}\label{1.9}
\nabla^{F,\beta,\varepsilon}=p \nabla^{TM,\beta,\varepsilon}p,\ \
\nabla^{F_1^\perp,\beta,\varepsilon}=p_1^\perp
\nabla^{TM,\beta,\varepsilon}p_1^\perp,\ \
\nabla^{F_2^\perp,\beta,\varepsilon}=p_2^\perp\nabla^{TM,\beta,\varepsilon}p_2^\perp.
\end{align}
In particular, one has
\begin{align}\label{1.9a}
\nabla^{F }=p \nabla^{TM }p,\ \
\nabla^{F_1^\perp }=p_1^\perp
\nabla^{TM }p_1^\perp,\ \
\nabla^{F_2^\perp }=p_2^\perp\nabla^{TM }p_2^\perp.
\end{align}

By (\ref{1.7})-(\ref{1.9a}) and the integrability of $F$, the
following identities hold for $X\in\Gamma(F)$:
\begin{align}\label{1.10}
\nabla^{F,\beta,\varepsilon}=\nabla^F ,\ \ p
\nabla^{TM,\beta,\varepsilon}_Xp_i^\perp=p \nabla^{TM }_Xp_i^\perp,\ \
i=1,\ 2,
\end{align}
$$p_1^\perp \nabla^{TM,\beta,\varepsilon}_X p=\beta^{2 }\varepsilon^{2} p_1^\perp \nabla^{TM }_X p,\ \ \
 p_2^\perp \nabla^{TM,\beta,\varepsilon}_X p=\beta^2   p_2^\perp \nabla^{TM }_X p. $$

 From (\ref{1.5})-(\ref{1.8}), we deduce that for
 $X\in\Gamma(F)$, $U_i,\, V_i\in\Gamma(F^\perp_i)$, $i=1,\, 2$,
\begin{align}\label{1.11}
 \left\langle \nabla^{TM,\beta,\varepsilon}_{U_1}V_1,X\right\rangle = \left\langle \nabla^{TM
 }_{U_1}V_1,X\right\rangle= \frac{1}{2 }  \left\langle \left[{U_1},V_1\right],X\right\rangle,
\end{align}
while
\begin{align}\label{1.11a}
 \left\langle \nabla^{TM,\beta,\varepsilon}_{U_2}V_2,X\right\rangle = \left\langle \nabla^{TM
 }_{U_2}V_2,X\right\rangle= \frac{1}{2 }  \left\langle \left[{U_2},V_2\right],X\right\rangle=0.
\end{align}
Equivalently, for any $U_i\in \Gamma(F^\perp_i)$, $i=1,\, 2$,
\begin{align}\label{1.12}
 p_1^\perp\nabla^{TM,\beta,\varepsilon}_{U_1}p= {\beta^{2 }\varepsilon^{2}}  p_1^\perp\nabla^{TM
 }_{U_1}p,\ \ \ p_2^\perp\nabla^{TM,\beta,\varepsilon}_{U_2}p=0.
\end{align}
Similarly, one verifies that
\begin{align}\label{1.13}
  \left\langle \nabla^{TM,\beta,\varepsilon}_{U_1}X,U_2\right\rangle
  =\frac{1}{2} \left\langle[U_1,X],U_2\right\rangle
  -\frac{\beta^2 }{2}\left\langle[U_1,U_2],X\right\rangle,
\end{align}
$$\left\langle \nabla^{TM,\beta,\varepsilon}_{U_2}X,U_1\right\rangle
  =\frac{\varepsilon^{2 }}{2 } \left\langle[U_1,X],U_2\right\rangle
  +\frac{\beta^{2 }\varepsilon^{2}}{2}\left\langle[U_1,U_2],X\right\rangle .$$

 For   convenience  of the  later computations, we collect the
 asymptotic behavior of various
covariant derivatives in the following lemma. These formulas can
be derived by applying (\ref{1.5})-(\ref{1.8}). The inner products
appear in the   lemma correspond to $\beta=\varepsilon=1$.

\begin{lemma}\label{tf} The following formulas hold for $X,\,
Y,\,Z\in \Gamma(F)$, $U_i,\,V_i,\, W_i\in \Gamma(F_i^\perp)$ with
$i=1,\, 2$, when $\beta>0$, $\varepsilon>0$ are small,
\begin{align}\label{f1}
 \left\langle\nabla^{TM,\beta,\varepsilon}_XY,Z\right\rangle  =O(1)
 , \ \
 \left\langle\nabla^{TM,\beta,\varepsilon}_XY,U_1\right\rangle=O\left(\beta^{2 }\varepsilon^{2}\right),\ \
\left\langle\nabla^{TM,\beta,\varepsilon}_XY,U_2\right\rangle=O\left(\beta^2\right),
\end{align}
\begin{align}\label{f2}
 \left\langle\nabla^{TM,\beta,\varepsilon}_XU_1,Y\right\rangle  =O\left(1\right)
 , \ \
 \left\langle\nabla^{TM,\beta,\varepsilon}_XU_1,V_1\right\rangle =O\left(1 \right),\ \
\left\langle\nabla^{TM,\beta,\varepsilon}_XU_1,U_2\right\rangle
=O\left(1\right),
\end{align}
\begin{align}\label{f3}
 \left\langle\nabla^{TM,\beta,\varepsilon}_XU_2,Y\right\rangle  =O\left(1\right)
 , \ \
 \left\langle\nabla^{TM,\beta,\varepsilon}_XU_2,U_1\right\rangle =O\left(  {\varepsilon^{2} } \right),\ \
\left\langle\nabla^{TM,\beta,\varepsilon}_XU_2,V_2\right\rangle =O\left(
{1} \right),
\end{align}
\begin{align}\label{f4}
 \left\langle\nabla^{TM,\beta,\varepsilon}_{U_1}X,Y\right\rangle  =O\left(1\right)
 , \ \
 \left\langle\nabla^{TM,\beta,\varepsilon}_{U_1}X,V_1\right\rangle =O\left( {\beta^{ 2}\varepsilon^{2}} \right),\ \
\left\langle\nabla^{TM,\beta,\varepsilon}_{U_1}X,U_2\right\rangle =O\left(
{1} \right),
\end{align}
\begin{align}\label{f5}
 \left\langle\nabla^{TM,\beta,\varepsilon}_{U_1}V_1,X\right\rangle  =O\left( {1} \right)
 , \ \
 \left\langle\nabla^{TM,\beta,\varepsilon}_{U_1}V_1,W_1\right\rangle =O\left( {1} \right),\ \
\left\langle\nabla^{TM,\beta,\varepsilon}_{U_1}V_1,U_2\right\rangle
=O\left(\frac{1}{\varepsilon^{2}}\right),
\end{align}
\begin{align}\label{f6}
 \left\langle\nabla^{TM,\beta,\varepsilon}_{U_1}U_2,X\right\rangle  =O\left(\frac{1}{\beta^2}\right)
 , \ \
 \left\langle\nabla^{TM,\beta,\varepsilon}_{U_1}U_2,V_1\right\rangle =O\left( {1}  \right),\ \
\left\langle\nabla^{TM,\beta,\varepsilon}_{U_1}U_2,V_2\right\rangle
=O\left(1\right),
\end{align}
\begin{align}\label{f7}
 \left\langle\nabla^{TM,\beta,\varepsilon}_{U_2}X,Y\right\rangle =O\left( {1} \right)
 , \ \
 \left\langle\nabla^{TM,\beta,\varepsilon}_{U_2}X,U_1\right\rangle =O\left( { {\varepsilon^{2}} }
\right),\ \
\left\langle\nabla^{TM,\beta,\varepsilon}_{U_2}X,V_2\right\rangle
=0,
\end{align}
\begin{align}\label{f8}
 \left\langle\nabla^{TM,\beta,\varepsilon}_{U_2}U_1,X\right\rangle  =O\left( \frac{1}{\beta^2} \right)
 , \ \
 \left\langle\nabla^{TM,\beta,\varepsilon}_{U_2}U_1,V_1\right\rangle =O\left( {1}  \right),\ \
\left\langle\nabla^{TM,\beta,\varepsilon}_{U_2}U_1,V_2\right\rangle
=O\left(1\right),
\end{align}
\begin{align}\label{f9}
 \left\langle\nabla^{TM,\beta,\varepsilon}_{U_2}V_2,X\right\rangle  =0
 , \ \
 \left\langle\nabla^{TM,\beta,\varepsilon}_{U_2}V_2,U_1\right\rangle =O\left( {\varepsilon^2}  \right),\ \
\left\langle\nabla^{TM,\beta,\varepsilon}_{U_2}V_2,W_2\right\rangle
=O\left( {1} \right).
\end{align}
\end{lemma}

In what follows, when we compute the asymptotics of various covariant derivatives,  
we will simply use the above asymptotic formulas freely without further notice.

  Let $R^{TM,\beta,\varepsilon}=(\nabla^{TM,\beta,\varepsilon})^2$ be the curvature of
  $\nabla^{TM,\beta,\varepsilon}$. Then for any $X,\ Y\in\Gamma(TM)$, one has
  the following standard formula,
\begin{align}\label{1.14}
   R^{TM,\beta,\varepsilon}(X,Y)=\nabla^{TM,\beta,\varepsilon}_X\nabla^{TM,\beta,\varepsilon}_Y-\nabla^{TM,\beta,\varepsilon}_Y\nabla^{TM,\beta,\varepsilon}_X-\nabla^{TM,\beta,\varepsilon}_{[X,Y]}.
\end{align}
Let $R^F=(\nabla^F)^2$ be the curvature of $\nabla^F$.
Let $k^{TM,\beta,\varepsilon}$, $k^F$ denote the scalar curvature
of $g^{TM,\beta,\varepsilon}$, $g^F$ respectively. Recall that
$k^F$ is defined in (\ref{001}). The following formula for $k^F$
is obvious,
\begin{align}\label{1.23aa}
 k^F=- \sum_{i,\, j=1}^{{\rm rk}(F)}\left\langle R^F\left(f_i,f_j\right)f_i,f_j\right\rangle,
\end{align}
where $f_i$, $i=1,\,\cdots,\, {\rm rk}(F)$, is an orthonormal basis of $F$.  Clearly, when $F=TM$, it reduces to the usual definition of the
scalar curvature $k^{TM}$ of $g^{TM}$.

\begin{prop}\label{t12.1}  If Condition (C) holds, then   when $\beta>0$, $\varepsilon>0$ are small, the following formula holds uniformly on any compact subset of $M$,
\begin{align}\label{1.24}
k^{TM,\beta,\varepsilon}=\frac{k^F }{\beta^2}+O\left(1+\frac{\varepsilon^2}{\beta^2}\right).
\end{align}
\end{prop}
\begin{proof}
 By (\ref{1.0}),  (\ref{1.10}),
(\ref{1.14}) and Lemma \ref{tf}, one
deduces that when $\beta>0$, $\varepsilon>0$ are very small, for any $X,\
Y\in\Gamma(F)$, one has
\begin{multline}\label{1.15}
\left\langle
R^{TM,\beta,\varepsilon}(X,Y)X,Y\right\rangle=\left\langle\nabla^{TM,\beta,\varepsilon}_X
\left(p+p_1^\perp+p_2^\perp\right)\nabla^{TM,\beta,\varepsilon}_YX,Y\right\rangle
\\
-\left\langle\nabla^{TM,\beta,\varepsilon}_Y\left(p+p_1^\perp+p_2^\perp\right)
\nabla^{TM,\beta,\varepsilon}_XX,Y\right\rangle-\left\langle
\nabla^{TM,\beta,\varepsilon}_{[X,Y]}X,Y\right\rangle\\
=\left\langle R^F(X,Y)X,Y\right\rangle
-\beta^{2 }\varepsilon^{2}\left\langle p_1^\perp
\nabla^{TM}_YX,\nabla^{TM}_XY\right\rangle
-\beta^2\left\langle p_2^\perp
\nabla^{TM}_YX,\nabla^{TM}_XY\right\rangle\\
+\beta^{2 }\varepsilon^{2}\left\langle p_1^\perp
\nabla^{TM}_XX,\nabla^{TM}_YY\right\rangle+\beta^2\left\langle
p_2^\perp \nabla^{TM}_XX,\nabla^{TM}_YY\right\rangle\\
=\left\langle R^F(X,Y)X,Y\right\rangle
+O\left( \beta^2\right).
\end{multline}

For $X\in\Gamma(F),\ U\in\Gamma(F_1^\perp)$, by
(\ref{1.5})-(\ref{1.14}), one finds that when $\beta , \ \varepsilon>0$ are
  small,
\begin{multline}\label{1.16}
\left\langle
R^{TM,\beta,\varepsilon}(X,U)X,U\right\rangle=\left\langle\nabla^{TM,\beta,\varepsilon}_X
\left(p+p_1^\perp+p_2^\perp\right)\nabla^{TM,\beta,\varepsilon}_UX,U\right\rangle
\\
-\left\langle\nabla^{TM,\beta,\varepsilon}_U
\left(p+p_1^\perp+p_2^\perp\right)\nabla^{TM,\beta,\varepsilon}_XX,U\right\rangle-\left\langle
\nabla^{TM,\beta,\varepsilon}_{\left(p+p_1^\perp+p_2^\perp\right)[X,U]}X,U\right\rangle\end{multline}
$$
=\beta^{2 }\varepsilon^{2}\left\langle\nabla^{TM}_Xp\nabla^{TM}_UX,U\right\rangle+\beta^{2 }\varepsilon^{2}
\left\langle\nabla^{TM,\beta,\varepsilon}_Xp_1^\perp
\nabla^{TM}_UX,U\right\rangle - {\varepsilon^{2}} \left\langle
p_2^\perp\nabla^{TM,\beta,\varepsilon}_UX,
\nabla^{TM,\beta,\varepsilon}_XU\right\rangle
$$
$$
-\beta^{2 }\varepsilon^{2}\left\langle\nabla^{TM}_Up\nabla^{TM}_XX,U\right\rangle-
\beta^{2 }\varepsilon^{2}
\left\langle\nabla^{TM,\beta,\varepsilon}_Up_1^\perp\nabla^{TM}_XX,U\right\rangle
+ {\varepsilon^{2}} \left\langle
p_2^\perp\nabla^{TM,\beta,\varepsilon}_XX,\nabla^{TM,\beta,\varepsilon}_UU\right\rangle
$$
$$
-\beta^{2 }\varepsilon^{2}\left\langle\nabla^{TM}_{\left(p+p_1^\perp\right)[X,U]}X,U\right\rangle - \left\langle\nabla^{TM,\beta,\varepsilon}_{ p_2^\perp [X,U]}X,U\right\rangle
= O\left(\beta^2+  \varepsilon^{2} \right).$$

Similarly,  for $X\in\Gamma(F)$,
$U\in\Gamma(F_2^\perp)$, one has that when $\beta>0$, $\varepsilon>0$ are small,
\begin{multline}\label{1.18}
\left\langle
R^{TM,\beta,\varepsilon}(X,U)X,U\right\rangle=\left\langle\nabla^{TM,\beta,\varepsilon}_X
\left(p+p_1^\perp+p_2^\perp\right)\nabla^{TM,\beta,\varepsilon}_UX,U\right\rangle
\\
-\left\langle\nabla^{TM,\beta,\varepsilon}_U
\left(p+p_1^\perp+p_2^\perp\right)\nabla^{TM,\beta,\varepsilon}_XX,U\right\rangle-\left\langle
\nabla^{TM,\beta,\varepsilon}_{\left(p+p_1^\perp+p_2^\perp\right)[X,U]}X,U\right\rangle\\
=\beta^2 \left\langle\nabla^{TM}_Xp\nabla^{TM}_UX,U\right\rangle-\frac{1}{\varepsilon^{2}}\left\langle
p_1^\perp \nabla^{TM,\beta,\varepsilon}_UX,
\nabla^{TM,\beta,\varepsilon}_XU\right\rangle+\beta^2\left\langle
\nabla^{TM,\beta,\varepsilon}_X
p_2^\perp\nabla^{TM}_UX,U\right\rangle
\end{multline}
$$
-\beta^2\left\langle\nabla^{TM}_Up\nabla^{TM}_XX,U\right\rangle- {\beta^{2 }}{\varepsilon^2}
\left\langle\nabla^{TM,\beta,\varepsilon}_Up_1^\perp\nabla^{TM}_XX,U\right\rangle
-\beta^2
\left\langle\nabla^{TM,\beta,\varepsilon}_Up_2^\perp\nabla^{TM}_XX,U\right\rangle$$
$$-\beta^2\left\langle\nabla^{TM}_{p[X,U]}X,U\right\rangle
-\beta^2\left\langle\nabla^{TM
}_{p_2^\perp[X,U]}X,U\right\rangle
=
O\left(  {\beta^2+\varepsilon^2} \right).$$

 For $U,\ V\in\Gamma(F_1^\perp)$, one verifies   that
\begin{multline}\label{1.19}
\left\langle
R^{TM,\beta,\varepsilon}(U,V)U,V\right\rangle=\left\langle\nabla^{TM,\beta,\varepsilon}_U
\left(p+p_1^\perp+p_2^\perp\right)\nabla^{TM,\beta,\varepsilon}_VU,V\right\rangle
\\
-\left\langle\nabla^{TM,\beta,\varepsilon}_V\left(p+p_1^\perp+p_2^\perp\right)
\nabla^{TM,\beta,\varepsilon}_UU,V\right\rangle-\left\langle
\nabla^{TM,\beta,\varepsilon}_{\left(p+p_1^\perp+p_2^\perp\right)[U,V]}U,V\right\rangle
\end{multline}
$$
=\beta^2\varepsilon^{2}\left\langle\nabla^{TM
}_Up\nabla^{TM,\beta,\varepsilon}_VU,V\right\rangle+\left\langle\nabla^{TM
}_Up_1^\perp\nabla^{TM
}_VU,V\right\rangle-\varepsilon^{2}\left\langle
p_2^\perp\nabla^{TM,\beta,\varepsilon}_VU,\nabla^{TM,\beta,\varepsilon}_UV\right\rangle$$
$$
-\beta^2\varepsilon^{2}\left\langle\nabla^{TM
}_Vp\nabla^{TM,\beta,\varepsilon}_UU,V\right\rangle
-\left\langle\nabla^{TM }_Vp_1^\perp\nabla^{TM }_UU,V\right\rangle
+\varepsilon^{2}\left\langle
p_2^\perp\nabla^{TM,\beta,\varepsilon}_UU,\nabla^{TM,\beta,\varepsilon}_VV\right\rangle$$
$$
-\left\langle\nabla^{TM,\beta,\varepsilon}_{p[U,V]}U,V\right\rangle-\left\langle\nabla^{TM
}_{p_1^\perp[U,V]}U,V\right\rangle-
\left\langle\nabla^{TM,\beta,\varepsilon}_{p_2^\perp[U,V]}U,V\right\rangle$$
$$
=-\varepsilon^{2}\left\langle
p_2^\perp\nabla^{TM,\beta,\varepsilon}_VU,\nabla^{TM,\beta,\varepsilon}_UV\right\rangle
+\varepsilon^{2}\left\langle
p_2^\perp\nabla^{TM,\beta,\varepsilon}_UU,\nabla^{TM,\beta,\varepsilon}_VV\right\rangle
 +O\left(1\right) =O\left(\frac{1}{\varepsilon^{2}}\right),$$
from which one gets that when $\beta>0$, $\varepsilon>0$ are
small,
\begin{align}\label{1.20}
\varepsilon^{2}\left\langle
R^{TM,\beta,\varepsilon}(U,V)U,V\right\rangle=O\left(1\right).
\end{align}

For $U,\ V\in \Gamma(F_2^\perp)$, one verifies directly that
\begin{multline}\label{1.21}
\left\langle
R^{TM,\beta,\varepsilon}(U,V)U,V\right\rangle=\left\langle\nabla^{TM,\beta,\varepsilon}_U\left(p+p_1^\perp+p_2^\perp\right)\nabla^{TM,\beta,\varepsilon}_VU,V\right\rangle
\\
-\left\langle\nabla^{TM,\beta,\varepsilon}_V\left(p+p_1^\perp+p_2^\perp\right)
\nabla^{TM,\beta,\varepsilon}_UU,V\right\rangle-\left\langle
\nabla^{TM,\beta,\varepsilon}_{ [U,V]}U,V\right\rangle\\
=\beta^2\left\langle\nabla^{TM
}_Up\nabla^{TM,\beta, \varepsilon}_VU,V\right\rangle-\frac{1}{\varepsilon^{2}}\left\langle
p_1^\perp\nabla^{TM,\beta,\varepsilon }_VU,\nabla^{TM,\beta,\varepsilon
}_UV\right\rangle+ \left\langle
\nabla^{TM }_Up_2^\perp\nabla^{TM }_VU,V\right\rangle\\
-\beta^2\left\langle\nabla^{TM
}_Vp\nabla^{TM,\beta,\varepsilon}_UU,V\right\rangle +
\frac{1}{\varepsilon^{2}}\left\langle
p_1^\perp\nabla^{TM,\beta,\varepsilon }_UU,\nabla^{TM,\beta,\varepsilon
}_VV\right\rangle -\left\langle
\nabla^{TM }_Vp_2^\perp\nabla^{TM }_UU,V\right\rangle\\
 -\left\langle
\nabla^{TM }_{ [U,V]}U,V\right\rangle=O(1).
\end{multline}

For $U\in \Gamma(F_1^\perp),\ V\in \Gamma(F_2^\perp)$, one verifies
directly that,
\begin{multline}\label{1.23}
\left\langle
R^{TM,\beta,\varepsilon}(U,V)U,V\right\rangle=\left\langle\nabla^{TM,\beta,\varepsilon}_U
\left(p+p_1^\perp+p_2^\perp\right)\nabla^{TM,\beta,\varepsilon}_VU,V\right\rangle
\\
-\left\langle\nabla^{TM,\beta,\varepsilon}_V\left(p+p_1^\perp+p_2^\perp\right)
\nabla^{TM,\beta,\varepsilon}_UU,V\right\rangle-\left\langle
\nabla^{TM,\beta,\varepsilon}_{ [U,V]}U,V\right\rangle\\
= -\beta^2\left\langle
p\nabla^{TM,\beta,\varepsilon}_VU,\nabla^{TM,\beta,\varepsilon}_UV\right\rangle-\frac{1}{\varepsilon^{2}}\left\langle
p_1^\perp\nabla^{TM,\beta,\varepsilon}_VU,\nabla^{TM,\beta,\varepsilon}_UV\right\rangle
+\left\langle\nabla^{TM,\beta,\varepsilon}_U p_2^\perp
\nabla^{TM,\beta,\varepsilon}_VU,V\right\rangle\\
+\beta^2\left\langle
p\nabla^{TM,\beta,\varepsilon}_UU,\nabla^{TM,\beta,\varepsilon}_VV\right\rangle+\frac{1}{\varepsilon^{2}}\left\langle
p_1^\perp\nabla^{TM,\beta,\varepsilon}_UU,\nabla^{TM,\beta,\varepsilon}_VV\right\rangle
-\left\langle\nabla^{TM }_V p_2^\perp
\nabla^{TM,\beta,\varepsilon}_UU,V\right\rangle\\
+\frac{1}{\varepsilon^{2}}\left\langle
U,\nabla^{TM,\beta,\varepsilon}_{[U,V]}V\right\rangle =O\left(\frac{1}{\varepsilon^2}+\frac{1}{\beta^2}\right) ,
\end{multline}
 from which  one gets that when $\beta>0$, $\varepsilon>0$ are small,
\begin{align}\label{1.23a}
\varepsilon^{2}\left\langle
R^{TM,\beta,\varepsilon}(U,V)U,V\right\rangle= \left\langle
R^{TM,\beta,\varepsilon}(V,U)V,U\right\rangle= O\left(1+\frac{\varepsilon^2}{\beta^2}\right).
\end{align}

From  (\ref{1.23aa}), (\ref{1.15})-(\ref{1.18}), (\ref{1.20}),  (\ref{1.21}) and
(\ref{1.23a}), one gets  (\ref{1.24}).  \end{proof}

\subsection{Bott connections on $F_1^\perp$ and $F_2^\perp$
}\label{s1.3}

From (\ref{1.5}) and  (\ref{1.6a})-(\ref{1.9}), one verifies directly
that for $X\in\Gamma(F)$, $U_i,\, V_i\in\Gamma(F_i^\perp)$, $i=1,\,
2$, one has
\begin{align}\label{1.25} \left\langle\nabla^{F_1^\perp,\beta,
\varepsilon}_XU_1,V_1
\right\rangle=\left\langle\left[X,U_1\right],V_1\right\rangle-
\frac{\beta^2\varepsilon^{2}}{2}\left\langle\left[U_1,V_1\right],X\right\rangle
,
\end{align}
$$\left\langle\nabla^{F_2^\perp,\beta,\varepsilon}_XU_2,V_2
\right\rangle=\left\langle\left[X,U_2\right],V_2\right\rangle.
$$

By (\ref{1.25}), one has that for $X\in\Gamma(F)$, $U_i
\in\Gamma(F_i^\perp)$, $i=1,\, 2$,
\begin{align}\label{1.26}\lim_{\varepsilon\rightarrow 0^+} \nabla^{F_i^\perp,\beta,
\varepsilon}_XU_i
 =
 \widetilde{\nabla}^{F_i^\perp }_XU_i:=p_i^\perp \left[X,U_i\right]
.
\end{align}

Let $\widetilde{\nabla}^{F_i^\perp }$ be the connection on
$F_i^\perp$ defined  by the second equality in (\ref{1.26}) and by
$\widetilde{\nabla}^{F_i^\perp }_UU_i= {\nabla}^{F_i^\perp }_UU_i$
for $U\in \Gamma(F^\perp)=\Gamma(F_1^\perp\oplus F_2^\perp)$. In
view of (\ref{1.26}) and \cite{Bo70}, we   call
$\widetilde{\nabla}^{F_i^\perp }$ a Bott connection on $F_i^\perp$
for   $i=1$ or $ 2$. Let $\widetilde{R}^{F_i^\perp }$ denote the
curvature of $\widetilde{\nabla}^{F_i^\perp }$ for $i=1,\, 2$.

The following result holds without Condition (C). 

\begin{lemma}\label{t1.2} For $X,\, Y\in\Gamma(F)$ and $i=1,\, 2$, the following
identity holds,
\begin{align}\label{1.27}
 \widetilde{R}^{F_i^\perp }(X,Y)=0
.
\end{align}
\end{lemma}
\begin{proof}We proceed as in \cite[Proof of Lemma 1.14]{Z01}. By
(\ref{1.26}) and the standard formula for the curvature (cf.
\cite[(1.3)]{Z01}), for any
$U\in\Gamma(F_i^\perp)$, $i=1,\, 2$,  one has,
\begin{multline}\label{1.28}
\widetilde{R}^{F_i^\perp }(X,Y)U=\widetilde{\nabla}^{F_i^\perp
}_X\widetilde{\nabla}^{F_i^\perp }_YU-\widetilde{\nabla}^{F_i^\perp
}_Y\widetilde{\nabla}^{F_i^\perp }_XU-\widetilde{\nabla}^{F_i^\perp
}_{[X,Y]}U\\
=p_i^\perp\big([X,[Y,U]]+[Y,[U,X]]+[U,[X,Y]]\big) -p_i^\perp
\left[X,\left({\rm Id}-p_i^\perp\right)[Y,U]\right]\\ -p_i^\perp
\left[Y,\left({\rm Id}-p_i^\perp\right)[U,X]\right]\\
=-p_i^\perp
\left[X,\left(p_1^\perp+p_2^\perp-p_i^\perp\right)[Y,U]\right]-p_i^\perp
\left[Y,\left(p_1^\perp+p_2^\perp-p_i^\perp\right)[U,X]\right],
\end{multline}
where the last equality follows from the Jacobi identity and the
integrability of $F$.

Now if $i=1$, then by (\ref{1.5}), one has $U\in\Gamma(F_1^\perp)$ and
\begin{align}\label{1.29}
p_1^\perp \left[X,  p_2^\perp  [Y,U]\right]=p_1^\perp \left[Y,
p_2^\perp  [U,X]\right]=0.
\end{align}
While if $i=2$, still by (\ref{1.5}), one has $U\in\Gamma(F_2^\perp)$ and
\begin{align}\label{1.30}
p_1^\perp[Y,U]=p_1^\perp[U,X]=0
 .
\end{align}

From (\ref{1.28})-(\ref{1.30}), one gets (\ref{1.27}). The proof of
Lemma  \ref{t1.2} is completed.
\end{proof}

\begin{rem}\label{t1.3} For $i=1,\, 2$,  let $ {R}^{F_i^\perp,\beta,\varepsilon }$ denote the
curvature of $ {\nabla}^{F_i^\perp,\beta,\varepsilon }$. From (\ref{1.25})-(\ref{1.27}), one finds that for any $X,\, Y\in\Gamma(F)$, when
$\beta>0$, $\varepsilon>0$ are small, the following identity holds:
\begin{align}\label{1.31}
 {R}^{F_i^\perp,\beta,\varepsilon }(X,Y)=O\left(\beta^2\varepsilon^2\right)
 .
\end{align}
\end{rem}

On the other hand, for $i=1,\, 2$, and $U_i,\, V_i,\, W_i,\,
Z_i\in\Gamma(F_i^\perp)$, by using (\ref{1.5}), (\ref{1.6a}),
(\ref{1.7}), (\ref{1.9}) and (\ref{1.14}), one verifies directly
that when $\beta>0$, $\varepsilon>0$ are small, the following identites, which will be used later, hold,
\begin{align}\label{1.101}
\beta^{-1}\varepsilon \left\langle
{R}^{F_1^\perp,\beta,\varepsilon }\left(X,U_1\right)V_1,W_1
\right\rangle  =O\left(\beta^{-1}\varepsilon\right)
 ,
\end{align}
\begin{align}\label{1.101a}
\beta^{-1}\left\langle {R}^{F_2^\perp,\beta,\varepsilon
}\left(X,U_2\right)V_2,W_2 \right\rangle
=O\left(\beta^{-1}\right)
 ,
\end{align}
\begin{align}\label{1.102}
\beta^{-1}\left\langle {R}^{F_1^\perp,\beta,\varepsilon
}\left(X,U_2\right)V_1,W_1 \right\rangle  =O\left(\beta^{-1}
\right)
 ,
\end{align}
\begin{align}\label{1.103}
 \varepsilon^{2}\left\langle {R}^{F_1^\perp,\beta,\varepsilon
}\left(U_1,V_1\right)W_1,Z_1 \right\rangle
=O\left(\varepsilon^{2} \right),
\end{align}
\begin{align}\label{1.103a}
 \left\langle {R}^{F_2^\perp,\beta,\varepsilon
}\left(U_2,V_2\right)W_2,Z_2 \right\rangle
=O\left(1\right),
\end{align}
\begin{align}\label{1.104}
\varepsilon \left\langle
{R}^{F_1^\perp,\beta,\varepsilon }\left(U_1,U_2\right)V_1,W_1
\right\rangle = O\left(\varepsilon \right)
 ,\end{align}
\begin{align}\label{1.105}
 \left\langle {R}^{F_1^\perp,\beta,\varepsilon
}\left(U_2,V_2\right)V_1,W_1 \right\rangle = O\left(1 \right)
 ,\end{align}
\begin{align}\label{1.106}
\beta^{-1}\varepsilon \left\langle
{R}^{F_2^\perp,\beta,\varepsilon
}\left(X,U_1\right)V_2,W_2\right\rangle
 =O\left(\beta^{-1}\varepsilon
\right)
 ,
\end{align}
\begin{align}\label{1.107}
\varepsilon \left\langle
{R}^{F_2^\perp,\beta,\varepsilon }\left(U_1,U_2\right)V_2,W_2
\right\rangle = O\left(\varepsilon \right),
 \end{align}
and
\begin{align}\label{1.108}
\varepsilon^{2}\left\langle {R}^{F_2^\perp,\beta,\varepsilon
}\left(U_1,V_1\right)V_2,W_2 \right\rangle =
O\left(\varepsilon^{2} \right).
 \end{align}

\subsection{Sub-Dirac operators associated to spin integrable  subbundles }\label{s1.4} Following \cite[\S2b]{LZ01}, we assume
now that $TM$, $ F$, $ F_i^\perp$, $i=1,\, 2$, are all oriented and
of even rank, with the orientation of $TM$ being  compatible with
the orientations on $F$, $F_1^\perp$ and $F_2^\perp$ through
(\ref{1.4}).
We further assume that $F$ is spin and carries a fixed spin
structure.

Let $S(F)=S_+(F)\oplus S_-(F)$ be the Hermitian bundle of spinors
associated to $(F,g^F)$. For any $X\in\Gamma(F)$, the Clifford
action $c(X)$ exchanges $S_\pm(F)$.

Let $i=1$ or $ 2$.
Let $\Lambda^*(F_i^{\perp })$ denote  the exterior algebra bundle
of $F_i^{\perp, *}$. Then $\Lambda^*(F_i^{\perp })$ carries a
canonically induced metric $g^{\Lambda^*(F_i^{\perp })}$ from
$g^{F_i^\perp}$. For any $U\in F_i^\perp$, let $U^*\in
F^{\perp,*}_i$ correspond  to $U$ via $g^{F_i^\perp}$.
For any $U\in\Gamma(F_i^\perp)$, set
\begin{align}\label{1.32}
c(U)=U^*\wedge-i_U,\ \ \ \widehat{c}(U)=U^*\wedge+i_U
 ,
\end{align}
where $U^*\wedge$ and $i_U$ are the exterior and interior
multiplications by $U^*$ and $U$ on $\Lambda^*(F_i^{\perp })$.

Denote  $q={\rm rk}( F)$,  $q_i={\rm rk}(F_i^\perp)$.

Let $h_1,\, \cdots,\, h_{q_i}$ be an oriented orthonormal basis of
$F_i^\perp$. Set
\begin{align}\label{1.33}
\tau\left(F_i^\perp,g^{F_i^\perp}\right)=\left(\frac{1}{\sqrt{-1}}\right)^{\frac{q_i(q_i+1)}{2}}c\left(h_1\right)\cdots
c\left(h_{q_i}\right).
\end{align}
Then
\begin{align}\label{1.34}
\tau\left(F_i^\perp,g^{F_i^\perp}\right)^2={\rm
Id}_{\Lambda^*\left(F_i^{\perp }\right)}.
\end{align}
Set
\begin{align}\label{1.35}
 {\Lambda^*_\pm\left(F_i^{\perp }\right)}=
 \left\{ h\in \Lambda^*\left(F_i^{\perp }\right):\tau\left(F_i^{\perp },g^{F_i^\perp}\right)h=\pm h\right\} .
\end{align}

Since $q_i$ is even, for any $h\in  F_i^\perp$, $c(h)$
anti-commutes with $\tau (F_i^\perp,g^{F_i^\perp} )$, while
$\widehat{c}(h)$ commutes with $\tau (F_i^\perp,g^{F_i^\perp} )$.
In particular, $c(h)$ exchanges $ {\Lambda^*_\pm (F_i^{\perp }
)}$.

Let $\widetilde{\tau}(F_i^\perp  )$ denote the ${\bf Z}_2$-grading
of $ {\Lambda ^* (F_i^{\perp } )}$ defined by
\begin{align}\label{1.351}
 \left. \widetilde{\tau}\left(F_i^\perp
 \right)\right|_{\Lambda^{\frac{\rm even}{\rm odd}} \left(F_i^{\perp }
 \right)}=\pm{\rm Id}|_{\Lambda^{\frac{\rm even}{\rm odd}} \left(F_i^{\perp }
 \right)}.
\end{align}

Now we have  the following  ${\bf Z}_2$-graded vector bundles over
$M$:
\begin{align}\label{1.36}
S(F)=S_+(F)\oplus S_-(F),
\end{align}
\begin{align}\label{1.37}
\Lambda ^*\left(F_i^{\perp }\right)=\Lambda^*_+\left(F_i^{\perp
}\right)\oplus\Lambda^*_-\left(F_i^{\perp }\right),\ \ \ i=1,\ 2,
\end{align}
and
\begin{align}\label{1.37a}
\Lambda^* \left(F_i^{\perp }\right)=\Lambda^{\rm
even}\left(F_i^{\perp }\right)\oplus\Lambda^{\rm
odd}\left(F_i^{\perp }\right),\ \ \ i=1,\ 2.
\end{align}

 We form the following  ${\bf Z}_2$-graded  tensor product:
\begin{align}\label{1.38}
 W\left(F,F^\perp_1,F^\perp_2\right)=S(F)\widehat{\otimes}
 \Lambda^*
\left(F_1^\perp\right)\widehat{\otimes}\Lambda^*
\left(F_2^\perp\right) ,
\end{align}
with the ${\bf Z}_2$-grading operator given by
\begin{align}\label{1.39}
\tau_W=\tau_{S(F)}\cdot
\tau\left(F_1^\perp,g^{F_1^\perp}\right)\cdot
\widetilde{\tau}\left(F_2^\perp \right),
\end{align}
where $\tau_{S(F)}$ is the ${\bf Z}_2$-grading operator defining the
splitting in (\ref{1.36}). We denote by
\begin{align}\label{1.40}
 W\left(F,F^\perp_1,F^\perp_2\right)=W_+\left(F,F^\perp_1,F^\perp_2\right)\oplus W_-\left(F,F^\perp_1,F^\perp_2\right)
\end{align}
the ${\bf Z}_2$-graded decomposition with respect to $\tau_W$.

Recall that the connections $\nabla^F$, $\nabla^{F_1^\perp}$ and
$\nabla^{F_2^\perp}$ have been defined in (\ref{1.9a}).  They lift canonically to Hermitian
connections $\nabla^{S(F)}$, $\nabla^{ \Lambda^*
\left(F_1^\perp\right)}$, $\nabla^{ \Lambda^*
\left(F_2^\perp\right)}$ on $S(F)$, $\Lambda^*
\left(F_1^\perp\right)$, $\Lambda^* \left(F_2^\perp\right)$
respectively, preserving the corresponding ${\bf Z}_2$-gradings.
Let $\nabla^{W(F,F_1^\perp,F_2^\perp)}$ be the canonically induced
connection on $W(F,F_1^\perp,F_2^\perp)$ which preserves the
canonically induced Hermitian metric on
$W(F,F_1^\perp,F_2^\perp)$, and also the ${\bf Z}_2$-grading of
$W(F,F_1^\perp,F_2^\perp)$.

For any vector bundle $E$ over $M$, by an integral polynomial of $E$
we will mean a bundle $\phi(E)$ which is a polynomial in the
exterior and symmetric powers of $E$ with integral coefficients.

For $i=1,\, 2$, let $\phi_i(F_i^\perp)$ be an integral polynomial of
$F_i^\perp$. We   denote the complexification of $\phi_i(F_i^\perp)$
by the same notation. Then $\phi_i(F_i^\perp)$ carries a naturally
induced Hermitian metric from $g^{F_i^\perp}$ and also a naturally
induced Hermitian connection $\nabla^{\phi_i(F_i^\perp)}$ from
$\nabla^{F_i^\perp}$.

Let $W(F,F_1^\perp,F_2^\perp)\otimes \phi_1(F_1^\perp)\otimes
\phi_2(F_2^\perp)$ be the ${\bf Z}_2$-graded vector bundle over $M$,
\begin{multline}\label{1.41}
 W\left(F,F_1^\perp,F_2^\perp\right)\otimes \phi_1\left(F_1^\perp\right)\otimes
\phi_2\left(F_2^\perp\right)=W_+\left(F,F_1^\perp,F_2^\perp\right)\otimes
\phi_1\left(F_1^\perp\right)\otimes \phi_2\left(F_2^\perp\right)\\
\oplus W_-\left(F,F_1^\perp,F_2^\perp\right)\otimes
\phi_1\left(F_1^\perp\right)\otimes \phi_2\left(F_2^\perp\right).
\end{multline}
Let $\nabla^{W\otimes\phi_1\otimes\phi_2}$ denote the naturally
induced Hermitian connection on $W(F,F_1^\perp,F_2^\perp)\otimes
\phi_1(F_1^\perp)\otimes \phi_2(F_2^\perp)$ with respect to the
naturally induced Hermitian metric on it. Clearly,
$\nabla^{W\otimes\phi_1\otimes\phi_2}$ preserves the ${\bf
Z}_2$-graded decomposition  in (\ref{1.41}).

Let $S$ be the ${\rm End}(TM)$-valued one form on $M$ defined by
\begin{align}\label{1.42}
 \nabla^{TM}=\nabla^F+\nabla^{F_1^\perp}+\nabla^{F_2^\perp}+S.
\end{align}
Let $e_1,\,\cdots,\,e_{\dim M}$ be an orthonormal basis of $TM$. Let
$\nabla ^{F,\phi_1(F^\perp_1)\otimes \phi_2(F^\perp_2)}$ be the
Hermitian connection on $W (F,F_1^\perp,F_2^\perp )\otimes
\phi_1(F_1^\perp)\otimes \phi_2(F_2^\perp)$ defined by that for any
$X\in\Gamma(TM)$,
\begin{align}\label{1.42a}
 \nabla ^{F,\phi_1(F^\perp_1)\otimes \phi_2(F^\perp_2)}_X= \nabla_{X}^{W\otimes\phi_1\otimes\phi_2}+
 \frac{1}{4}\sum_{i,\,j=1}^{\dim M}\left\langle S(X)e_i,e_j\right\rangle  c\left(e_i\right)c\left(e_j\right).
\end{align}

Let the linear operator $D^{F,\phi_1(F^\perp_1)\otimes
\phi_2(F^\perp_2)}:\Gamma(W (F,F_1^\perp,F_2^\perp )\otimes \phi_1
(F_1^\perp )\otimes \phi_2 (F_2^\perp ))\rightarrow \Gamma(W
(F,F_1^\perp,F_2^\perp )\otimes \phi_1 (F_1^\perp )\otimes \phi_2
(F_2^\perp ))$ be    defined by (compare with \cite[Definition
2.2]{LZ01})
\begin{align}\label{1.43}
D^{F,\phi_1(F^\perp_1)\otimes \phi_2(F^\perp_2)}=\sum_{i=1}^{\dim M}
c\left(e_i\right)\nabla_{e_i}^{F, \phi_1(F_1^\perp)\otimes
\phi_2(F_2^\perp)}.
\end{align}
We call $D^{F,\phi_1(F^\perp_1)\otimes \phi_2(F^\perp_2)}$ a
sub-Dirac operator with respect to the spin vector bundle $F$.

One verifies  that $D^{F,\phi_1(F^\perp_1)\otimes
\phi_2(F^\perp_2)}$ is a first order formally self-adjoint elliptic
differential operator. Moreover, it exchanges $\Gamma(W_\pm
(F,F_1^\perp,F_2^\perp )\otimes \phi_1 (F_1^\perp )\otimes \phi_2
(F_2^\perp ))$. We denote by $D_\pm^{F,\phi_1(F^\perp_1)\otimes
\phi_2(F^\perp_2)}$ the restrictions of
$D^{F,\phi_1(F^\perp_1)\otimes \phi_2(F^\perp_2)}$ to $\Gamma(W_\pm
(F,F_1^\perp,F_2^\perp )\otimes \phi_1 (F_1^\perp )\otimes \phi_2
(F_2^\perp ))$. 
Then one has
\begin{align}\label{1.44}
 \left(D_+^{F,\phi_1(F^\perp_1)\otimes \phi_2(F^\perp_2)}\right)^*=D_-^{F,\phi_1(F^\perp_1)\otimes \phi_2(F^\perp_2)} .
\end{align}

\begin{rem}\label{t1.4} As in \cite[(2.21)]{LZ01}, when
$F_1^\perp$, $F_2^\perp$ are also spin and carry fixed spin
structures, then    $TM=F\oplus F^\perp_1\oplus F^\perp_2$ is spin
and carries an induced spin structure from the spin structures on
$F$, $F_1^\perp$ and $F_2^\perp$. Moreover, one has the following
identifications  of ${\bf Z}_2$-graded vector bundles (cf.
\cite{LaMi89}) for $i=1,\, 2$,
\begin{align}\label{1.44a}
  \Lambda^*_+\left(F_i^\perp\right)\oplus \Lambda^*_-\left(F_i^\perp\right)=S_+ \left(F_i^\perp\right)\otimes
  S
  \left(F_i^\perp\right)^*\oplus S_- \left(F_i^\perp\right)\otimes
  S
  \left(F_i^\perp\right)^*,
\end{align}
\begin{multline}\label{1.44b}
  \Lambda^{\rm even}\left(F_i^\perp\right)\oplus \Lambda^{\rm odd}
  \left(F_i^\perp\right)=\left(S_+ \left(F_i^\perp\right)\otimes
  S_+
  \left(F_i^\perp\right)^*\oplus S_- \left(F_i^\perp\right)\otimes
  S_-
  \left(F_i^\perp\right)^*\right)
\\ \oplus   \left(S_+ \left(F_i^\perp\right)\otimes
  S_-
  \left(F_i^\perp\right)^*\oplus S_- \left(F_i^\perp\right)\otimes
  S_+
  \left(F_i^\perp\right)^*\right). \end{multline}
By (\ref{1.33})-(\ref{1.43}), (\ref{1.44a}) and (\ref{1.44b}),
$D^{F,\phi_1(F^\perp_1)\otimes \phi_2(F^\perp_2)}$ is simply the
twisted   Dirac operator
\begin{multline}\label{1.45}
D^{F,\phi_1(F^\perp_1)\otimes \phi_2(F^\perp_2)}:
\Gamma\left(S(TM)\widehat{\otimes} S\left(F_2^\perp\right)^*\otimes
S\left(F_1^\perp\right)^*\otimes\phi_1\left(F^\perp_1\right)\otimes
\phi_2\left(F^\perp_2\right)
 \right)\\
 \longrightarrow \Gamma\left(S(TM)\widehat{\otimes }
S\left(F_2^\perp\right)^*\otimes
S\left(F_1^\perp\right)^*\otimes\phi_1\left(F^\perp_1\right)\otimes
\phi_2\left(F^\perp_2\right)
 \right),
\end{multline}
where for $i=1,\, 2$, the Hermitian (dual) bundle of spinors
$S(F_i^\perp)^*$  associated to $(F_i^\perp,g^{F_i^\perp})$ carries
the Hermitian connection induced from $\nabla^{F_i^\perp}$.

The point of (\ref{1.43}) is that it only requires $F$ being spin.
While on the other hand, (\ref{1.45}) allows us to take the
advantage of applying the calculations already done for usual
(twisted) Dirac operators when doing local computations.
\end{rem}

\begin{rem}\label{t1.4a} It is clear that the definition in
(\ref{1.43}) does not require that $F$ being an integrable subbundle
of $TM$. 
\end{rem}

Let $\Delta^{F,\phi_1(F^\perp_1)\otimes \phi_2(F^\perp_2)}$ denote
the Bochner Laplacian defined by
\begin{align}\label{1.46}
\Delta^{F,\phi_1(F^\perp_1)\otimes
\phi_2(F^\perp_2)}=\sum_{i=1}^{\dim M}
\left(\nabla_{e_i}^{F,\phi_1(F^\perp_1)\otimes
\phi_2(F^\perp_2)}\right)^2-\nabla^{F,\phi_1(F^\perp_1)\otimes
\phi_2(F^\perp_2)}_{\sum_{i=1}^{\dim M}\nabla^{TM}_{e_i}e_i}.
  \end{align}

 Let
$k^{TM}$ be the scalar curvature of $g^{TM}$, $R^{F_i^\perp}$
($i=1,\, 2$) be the curvature of $\nabla^{F_i^\perp}$. Let
$R^{\phi_1(F_1^\perp)\otimes \phi_2(F_2^\perp)}$  be the curvature of the tensor product connection on $\phi_1(F_1^\perp)\otimes \phi_2(F_2^\perp)$ induced from
$\nabla^{\phi_1(F_1^\perp)}$ and $\nabla^{\phi_2(F_2^\perp)}$.

 In view of Remark \ref{t1.4}, the
following Lichnerowicz type formula holds:
\begin{multline}\label{1.47}
\left(D^{F,\phi_1(F^\perp_1)\otimes
\phi_2(F^\perp_2)}\right)^2=-\Delta^{F,\phi_1(F^\perp_1)\otimes
\phi_2(F^\perp_2)}+\frac{k^{TM}}{4}
+
\frac{1}{2}\sum_{i,\,j=1}^{\dim M}c\left(e_i\right)c\left(e_j\right)R^{ \phi_1(F^\perp_1)\otimes
\phi_2(F^\perp_2)}\left(e_i,e_j\right)
\\
+\frac{1}{8}\sum_{i,\, j=1}^{\dim M}\sum_{s,t}  \left\langle
R^{F_1^\perp }\left(e_i,e_j\right)h_t,h_s\right\rangle 
c\left(e_i\right)c\left(e_j\right)\widehat{c}\left(h_s\right)\widehat{c}\left(h_t\right)
\\
+\frac{1}{8}\sum_{i,\, j=1}^{\dim M}\sum_{s,t}  \left\langle
R^{F_2^\perp }\left(e_i,e_j\right)h_t',h_s'\right\rangle 
c\left(e_i\right)c\left(e_j\right)\widehat{c}\left(h_s'\right)\widehat{c}\left(h_t'\right)  ,
\end{multline}
where $h_s,\ h_t$ (resp. $h_s',\ h_t'$) run through an orthonormal basis of $F_1^\perp $ (resp. $F_2^\perp$). 

When $M$ is compact,   by   the Atiyah-Singer index theorem
\cite{ASI} (cf. \cite{LaMi89}), one has
\begin{multline}\label{1.48}
{\rm ind}\left( D_+^{F,\phi_1(F^\perp_1)\otimes
\phi_2(F^\perp_2)}\right)\\ = 2^{\frac{q_1}{2} }\left\langle
\widehat{A}(F)\widehat{L}\left(F_1^\perp\right)
e\left(F_2^\perp\right){\rm
ch}\left(\phi_1\left(F_1^\perp\right)\right){\rm
ch}\left(\phi_2\left(F_2^\perp\right)\right),[M]\right\rangle,
\end{multline}
where $\widehat{L}(F_1^\perp)$   is the Hirzebruch
$\widehat{L}$-class (cf. \cite[(11.18') of Chap. III]{LaMi89}) of
$F_1^\perp$, $e(F_2^\perp)$ is the Euler class (cf. \cite[\S
3.4]{Z01}) of $F_2^\perp$, and ``${\rm ch}$'' is the notation for
the Chern character (cf. \cite[\S1.6.4]{Z01}).

\subsection{A vanishing theorem for   almost isometric
foliations}\label{s1.5} In this subsection, we assume $M$ is compact
and prove a vanishing theorem.  Some of the computations  in
this subsection will be used in the next section where we will deal
with the case where $M$ is non-compact.

Let $f_1,\, \cdots,\, f_q$ be an oriented orthonormal basis of $F$.
Let $h_1,\,\cdots,\,h_{q_1}$ (resp. $e_1,\,\cdots,\,e_{q_2}$) be an
oriented orthonormal basis of $F_1^\perp$ (resp. $F_2^\perp$).

Let $\beta>0$, $\varepsilon>0$ and   consider the construction in Section
\ref{s1.4} with respect to the metric $g^{TM}_{\beta,\varepsilon}$ defined
in (\ref{1.8}). We still use the superscripts ``$\beta$,
$\varepsilon$'' to decorate the geometric data associated to
$g^{TM}_{\beta,\varepsilon} $. For example, $D^{F,\phi_1(F^\perp_1)\otimes
\phi_2(F^\perp_2),\beta,\varepsilon }$ now denotes the sub-Dirac
operator constructed in (\ref{1.43}) associated  to
$g^{TM}_{\beta,\varepsilon} $. Moreover, it can be written as
\begin{multline}\label{1.431}
D^{F,\phi_1(F^\perp_1)\otimes
\phi_2(F^\perp_2),\beta,\varepsilon}=\beta^{-1}\sum_{i=1}^{q}
c_{\beta,\varepsilon}\left(\beta^{-1}f_i\right)\nabla_{f_i}^{F, \phi_1(F_1^\perp)\otimes
\phi_2(F_2^\perp),\beta,\varepsilon}
\\
+\varepsilon \sum_{j=1}^{q_1}
c_{\beta,\varepsilon}\left(\varepsilon h_j\right)\nabla_{h_j}^{F, \phi_1(F_1^\perp)\otimes
\phi_2(F_2^\perp),\beta,\varepsilon}
 + \sum_{s=1}^{q_2}
c_{\beta,\varepsilon}\left(e_s\right)\nabla_{e_s}^{F, \phi_1(F_1^\perp)\otimes
\phi_2(F_2^\perp),\beta,\varepsilon}.
\end{multline}

By (\ref{1.431}), the Lichnerowicz type formula (\ref{1.47}) for
$(D^{F,\phi_1(F^\perp_1)\otimes \phi_2(F^\perp_2),\beta,\varepsilon })^2$
takes the following form (compare with
\cite[Theorem 2.3]{LZ01}),
\begin{multline}\label{1.49}
\left(D^{F,\phi_1(F^\perp_1)\otimes
\phi_2(F^\perp_2),\beta,\varepsilon}\right)^2=-\Delta^{F,\phi_1(F^\perp_1)\otimes
\phi_2(F^\perp_2),\beta,\varepsilon}+\frac{k^{TM,\beta,\varepsilon}}{4}
\\
+
\frac{1}{2\beta^2}\sum_{i,\,j=1}^qc_{\beta,\varepsilon}\left(\beta^{-1}f_i\right)c_{\beta,\varepsilon}\left(\beta^{-1}f_j\right)
R^{\phi_1(F^\perp_1)\otimes \phi_2(F^\perp_2),\beta,\varepsilon}\left(f_i,f_j\right)
\end{multline}
$$
+\frac{\varepsilon^2 }{2}\sum_{i,\,j=1}^{q_1}c_{\beta,\varepsilon}\left(\varepsilon h_i\right)
c_{\beta,\varepsilon}\left(\varepsilon h_j\right)R^{\phi_1(F^\perp_1)\otimes \phi_2(F^\perp_2),\beta,\varepsilon}\left(h_i,h_j\right) 
$$
$$
+\frac{1}{2}\sum_{i,\,j=1}^{q_2}c_{\beta,\varepsilon}\left(e_i\right)c_{\beta,\varepsilon}\left(e_j\right)R^{\phi_1(F^\perp_1)\otimes \phi_2(F^\perp_2),\beta,\varepsilon}\left(e_i,e_j\right)
$$
$$
+\frac{\varepsilon }{\beta}\sum_{i =1}^q\sum_{j
=1}^{q_1}c_{\beta,\varepsilon}\left(\beta^{-1}f_i\right)c_{\beta,\varepsilon}\left(\varepsilon h_j\right)R^{\phi_1(F^\perp_1)\otimes \phi_2(F^\perp_2),\beta,\varepsilon}\left(f_i,h_j\right)$$
$$
 +\frac{1}{\beta}\sum_{i =1}^q\sum_{j
=1}^{q_2}c_{\beta,\varepsilon}\left(\beta^{-1}f_i\right)c_{\beta,\varepsilon}\left(e_j\right)R^{\phi_1(F^\perp_1)\otimes \phi_2(F^\perp_2),\beta,\varepsilon}\left(f_i,e_j\right)
$$
$$
+\varepsilon \sum_{i =1}^{q_1}\sum_{j
=1}^{q_2}c_{\beta,\varepsilon}\left(\varepsilon h_i\right)c_{\beta,\varepsilon}\left(e_j\right)R^{\phi_1(F^\perp_1)\otimes \phi_2(F^\perp_2),\beta,\varepsilon}\left(h_i,e_j\right)
 $$
 $$
+\frac{1}{8\beta^2}\sum_{i,\, j=1}^q\sum_{s,\, t=1}^{q_1}\left\langle
R^{F_1^\perp,\beta,\varepsilon}\left(f_i,f_j\right)h_t,h_s\right\rangle
c_{\beta,\varepsilon}\left(\beta^{-1}f_i\right)c_{\beta,\varepsilon}\left(\beta^{-1}f_j\right)\widehat{c}_{\beta,\varepsilon}\left(\varepsilon h_s\right)\widehat{c}_{\beta,\varepsilon}\left(\varepsilon h_t\right)
$$
$$ +\frac{\varepsilon^2}{8}\sum_{i,\, j=1}^{q_1}\sum_{s,\,
t=1}^{q_1}\left\langle
R^{F_1^\perp,\beta,\varepsilon}\left(h_i,h_j\right)h_t,h_s\right\rangle
c_{\beta,\varepsilon}\left(\varepsilon h_i\right)c_{\beta,\varepsilon}\left(\varepsilon h_j\right)\widehat{c}_{\beta,\varepsilon}\left(\varepsilon h_s\right)\widehat{c}_{\beta,\varepsilon}\left(\varepsilon h_t\right)$$
$$
+\frac{1}{8}\sum_{i,\, j=1}^{q_2}\sum_{s,\,
t=1}^{q_1}\left\langle
R^{F_1^\perp,\beta,\varepsilon}\left(e_i,e_j\right)h_t,h_s\right\rangle
c_{\beta,\varepsilon}\left(e_i\right)c_{\beta,\varepsilon}\left(e_j\right)\widehat{c}_{\beta,\varepsilon}\left(\varepsilon h_s\right)\widehat{c}_{\beta,\varepsilon}\left(\varepsilon h_t\right)$$
$$
+\frac{\varepsilon }{4\beta}\sum_{i =1}^q\sum_{j=1
}^{q_1}\sum_{s,\, t=1}^{q_1}\left\langle
R^{F_1^\perp,\beta,\varepsilon}\left(f_i,h_j\right)h_t,h_s\right\rangle
c_{\beta,\varepsilon}\left(\beta^{-1}f_i\right)c_{\beta,\varepsilon}\left(\varepsilon h_j\right)\widehat{c}_{\beta,\varepsilon}\left(\varepsilon h_s\right)\widehat{c}_{\beta,\varepsilon}\left(\varepsilon h_t\right)$$
$$
+\frac{1}{4\beta}\sum_{i =1}^q\sum_{j=1 }^{q_2}\sum_{s,\,
t=1}^{q_1}\left\langle
R^{F_1^\perp,\beta,\varepsilon}\left(f_i,e_j\right)h_t,h_s\right\rangle
c_{\beta,\varepsilon}\left(\beta^{-1}f_i\right)c_{\beta,\varepsilon}\left(e_j\right)\widehat{c}_{\beta,\varepsilon}\left(\varepsilon h_s\right)\widehat{c}_{\beta,\varepsilon}\left(\varepsilon h_t\right)$$
$$
+\frac{\varepsilon }{4}\sum_{i=1
}^{q_1}\sum_{j=1 }^{q_2}\sum_{s,\, t=1}^{q_1}\left\langle
R^{F_1^\perp,\beta,\varepsilon}\left(h_i,e_j\right)h_t,h_s\right\rangle
c_{\beta,\varepsilon}\left(\varepsilon h_i\right)c_{\beta,\varepsilon}\left(e_j\right)\widehat{c}_{\beta,\varepsilon}\left(\varepsilon h_s\right)\widehat{c}_{\beta,\varepsilon}\left(\varepsilon h_t\right)
$$
$$
+\frac{1}{8\beta^2}\sum_{i,\, j=1}^q\sum_{s,\, t=1}^{q_2}\left\langle
R^{F_2^\perp,\beta,\varepsilon}\left(f_i,f_j\right)e_t,e_s\right\rangle
c_{\beta,\varepsilon}\left(\beta^{-1}f_i\right)c_{\beta,\varepsilon}\left(\beta^{-1}f_j\right)\widehat{c}_{\beta,\varepsilon}\left(e_s\right)\widehat{c}_{\beta,\varepsilon}\left(e_t\right)
$$
$$ +\frac{\varepsilon^{2}}{8}\sum_{i,\, j=1}^{q_1}\sum_{s,\,
t=1}^{q_2}\left\langle
R^{F_2^\perp,\beta,\varepsilon}\left(h_i,h_j\right)e_t,e_s\right\rangle
c_{\beta,\varepsilon}\left(\varepsilon h_i\right)c_{\beta,\varepsilon}\left(\varepsilon h_j\right)\widehat{c}_{\beta,\varepsilon}\left(e_s\right)\widehat{c}_{\beta,\varepsilon}\left(e_t\right)$$
$$
+\frac{1 }{8}\sum_{i,\, j=1}^{q_2}\sum_{s,\,
t=1}^{q_2}\left\langle
R^{F_2^\perp,\beta,\varepsilon}\left(e_i,e_j\right)e_t,e_s\right\rangle
c_{\beta,\varepsilon}\left(e_i\right)c_{\beta,\varepsilon}\left(e_j\right)\widehat{c}_{\beta,\varepsilon}\left(e_s\right)\widehat{c}_{\beta,\varepsilon}\left(e_t\right)$$
$$
+\frac{\varepsilon }{4\beta}\sum_{i=1 }^q\sum_{j=1
}^{q_1}\sum_{s,\, t=1}^{q_2}\left\langle
R^{F_2^\perp,\beta,\varepsilon}\left(f_i,h_j\right)e_t,e_s\right\rangle
c_{\beta,\varepsilon}\left(\beta^{-1}f_i\right)c_{\beta,\varepsilon}\left(\varepsilon h_j\right)\widehat{c}_{\beta,\varepsilon}\left(e_s\right)\widehat{c}_{\beta,\varepsilon}\left(e_t\right)$$
$$
+\frac{1}{4\beta}\sum_{i =1}^q\sum_{j=1 }^{q_2}\sum_{s,\,
t=1}^{q_2}\left\langle
R^{F_2^\perp,\beta,\varepsilon}\left(f_i,e_j\right)e_t,e_s\right\rangle
c_{\beta,\varepsilon}\left(\beta^{-1}f_i\right)c_{\beta,\varepsilon}\left(e_j\right)\widehat{c}_{\beta,\varepsilon}\left(e_s\right)\widehat{c}_{\beta,\varepsilon}\left(e_t\right)$$
$$
+\frac{\varepsilon }{4}\sum_{i
=1}^{q_1}\sum_{j=1 }^{q_2}\sum_{s,\, t=1}^{q_2}\left\langle
R^{F_2^\perp,\beta,\varepsilon}\left(h_i,e_j\right)e_t,e_s\right\rangle
c_{\beta,\varepsilon}\left(\varepsilon h_i\right)c_{\beta,\varepsilon}\left(e_j\right)\widehat{c}_{\beta,\varepsilon}\left(e_s\right)\widehat{c}_{\beta,\varepsilon}\left(e_t\right).
$$

 By (\ref{1.24}), (\ref{1.31})-(\ref{1.108}) and (\ref{1.49}), we get that when $\beta>0$,
 $\varepsilon>0$ are small,
 \begin{align}\label{1.50}
\left(D^{F,\phi_1(F^\perp_1)\otimes
\phi_2(F^\perp_2),\beta,\varepsilon}\right)^2=-\Delta^{F,\phi_1(F^\perp_1)\otimes
\phi_2(F^\perp_2),\beta,\varepsilon}+\frac{k^{F}}{4\beta^2}
+O\left( \frac{1 }{\beta}+\frac{\varepsilon^2 }{\beta^2}\right).
 \end{align}

 \begin{prop}\label{t1.6} If $k^F>0$ over $M$, then for any Pontrjagin classes
 $p(F_1^\perp)$,
 $p'(F_2^\perp)$  of $F_1^\perp$,
 $F_2^\perp$ respectively, the following identity holds,
 \begin{align}\label{1.51}
\left\langle\widehat{A}(F)p\left(F_1^\perp\right)e\left(F_2^\perp\right)p'\left(F_2^\perp\right),[M]\right\rangle
=0.
 \end{align}
 \end{prop}
 \begin{proof}Since $k^F>0$ over $M$, one can take $\beta>0$, $\varepsilon>0$
 small enough so that the corresponding terms in the right hand side of
 (\ref{1.50}) verifies that
 \begin{align}\label{1.52}
 \frac{k^{F}}{4\beta^2} +O\left( \frac{1 }{\beta}+\frac{\varepsilon^2 }{\beta^2}\right)>0
 \end{align}
 over $M$.
Since  $-\Delta^{F,\phi_1(F^\perp_1)\otimes
\phi_2(F^\perp_2),\beta,\varepsilon}$ is   nonnegative, by
(\ref{1.44}), (\ref{1.50}) and (\ref{1.52}), one gets
\begin{align}\label{1.53}
{\rm ind}\left( D_+^{F,\phi_1(F^\perp_1)\otimes
\phi_2(F^\perp_2),\beta,\varepsilon}\right)= 0.
\end{align}

From (\ref{1.48}) and (\ref{1.53}), we get
\begin{align}\label{1.54}\left\langle
\widehat{A}(F)\widehat{L}\left(F_1^\perp\right){\rm
ch}\left(\phi_1\left(F_1^\perp\right)\right)
e\left(F_2^\perp\right){\rm
ch}\left(\phi_2\left(F_2^\perp\right)\right),[M]\right\rangle=0.
\end{align}

Now as it is standard that any Pontrjagin class of $F_1^\perp$
(resp. $F_2^\perp$)  can be expressed as a rational linear
combination of   classes of the form $\widehat{L} (F_1^\perp
){\rm ch} (\phi_1 (F_1^\perp ) )$ (resp. $ {\rm ch} (\phi_2 (F_2^\perp ) )$), one gets (\ref{1.51}) from
(\ref{1.54}).
 \end{proof}

 \begin{rem}\label{t1.7} Recall that $F^\perp=F^\perp_1\oplus F^\perp_2$. It is proved in \cite[Theorem 2.6]{LZ01}
 that if the conditions in Proposition \ref{t1.6} hold, then
 $
 \langle\widehat{A}(F)p (F^\perp   ),[M] \rangle
=0.
 $ Here if one changes the ${\bf Z}_2$-grading in the definition of the sub-Dirac operator by replacing
 $\widetilde{\tau}(F_2^\perp)$ in (\ref{1.39}) by $\tau(F_2^\perp,
 g^{F_2^\perp})$, then one can prove that under the same condition
 as in Proposition \ref{t1.6}, 
\begin{align}\label{1.54a}\left\langle
\widehat{A}(F)
p\left(F_1^\perp\right)p'\left(F_2^\perp\right),[M]\right\rangle=0
\end{align}
for any Pontrjagin classes $p(F_1^\perp)$, $p'(F_2^\perp)$ of
$F_1^\perp$, $F_2^\perp$.
 \end{rem}

\begin{rem}\label{t1.7x}  Formulas (\ref{1.51}) and (\ref{1.54a}) hold indeed without   Condition (C) in Definition \ref{t1.1c}. This can be checked if we set $\varepsilon=\sqrt{\beta}$. 
 \end{rem}

\section{Connes  fibration and the Dirac operator on foliations}\label{s2}

In this Section we prove Theorem \ref{t0.2}. We will make use of the Connes fibration
which has indeed played an essential role in Connes' original
proof of Theorem \ref{t0.1} given in \cite{Co86}.

This Section is organized as follows. In Section \ref{s2.1}, we recall the construction of the  Connes  fibration
over a foliation. 
 In Section
\ref{s2.2a}, we introduce a coordinate system near the embedded 
submanifold from the original foliation into the Connes  foliation.  In Section \ref{s2.2}, we give an adiabatic limit
estimate of the sub-Dirac operator on the Connes    fibration.  In
Section \ref{s2.3}, we  embed the smooth sections over the
embedded submanifold to the space of smooth sections, having
compact support near the embedded submanifold,  on the Connes  
fibration. In Section \ref{s2.4},  we state a key estimate result which will be proved in   Sections
\ref{s2.5}-\ref{s2.8}.   In Sections
\ref{s2.9}, we complete  the proof of Theorem
\ref{t0.2}.

\subsection{The Connes   fibration
}\label{s2.1}

We  start by recalling the original construction in \cite{Co86}. 

Let $(M,F)$ be a compact foliation, where $F$ is an
integrable subbundle of the tangent vector bundle $TM$ of a closed
manifold $M$. For simplicity, we  make the assumption that $TM, \, F$ are oriented,
then $TM/F$ is also oriented. We further assume that $TM$ is spin and
carries a fixed spin structure.

For any oriented vector space $E$ of rank $n$, let $\mE$ be the
set of all Euclidean metrics on $E$. It is well known that $\mE$
is the homogeneous space $GL(n,{\bf R})^+/SO(n)$ (with $\dim
\mE=\frac{n(n+1) }{2}$), which carries a
natural Riemannian metric of nonpositive sectional curvature (cf.
\cite{He}). In particular, any two points of $\mE$ can be joined
by a unique geodesic.

Following  \cite[Section 5]{Co86}, let $\pi:\mM\rightarrow M$ be the
fibration over $M$ such that for any $x\in M$, $\mM_x=\pi^{-1}(x)$
is the space of Euclidean metrics on the linear space $T_xM/F_x$.
Clearly, $\mM$ is noncompact.

Let  $T^V\mM$ denote the vertical tangent bundle of the fibration
$\pi:\mM\rightarrow M$. Then it carries a natural metric
$g^{T^V\mM}$   such that   any two points $p,\,
q\in\mM_x$, with $x\in M$,  can be joined by a unique geodesic in
$\mM_x$.

By using the Bott connection
\cite{Bo70} on $TM/F$, one can lift $F$ to an integrable subbundle
$\mF$ of $T\mM$.\footnote{Indeed,
the Bott connection on $TM/F$ determines an integrable lift
$\widetilde{\mF}$ of $F$ in $T\widetilde{\mM}$, where
$\widetilde{\mM}=GL(TM/F)^+$ is the $GL(q_1,{\bf R})^+$ (with
$q_1={\rm rk}(TM/F)$) principal bundle of oriented frames  over
$M$. Now as $\widetilde{\mM}$ is a principal $SO(q_1)$ bundle over
$\mM$, $\widetilde{\mF}$ determines an integrable subbundle $\mF$
of $T\mM$.}

For any $v\in\mM$, $T_v\mM/(\mF_v\oplus T^V_v\mM)$ identifies with
$T_{\pi(v)}M/F_{\pi(v)}$ under the projection $\pi:\mM\rightarrow
M$. By definition, $v$ determines a metric on
$T_{\pi(v)}M/F_{\pi(v)}$, thus it also determines a metric on
$T_v\mM/(\mF_v\oplus T^V_v\mM)$. In this way, $T\mM/(\mF \oplus T^V
\mM)$ carries a canonically induced metric.

Let $\mF_1^\perp$ be a subbundle of $T\mM$, which is  transversal  to $\mF \oplus T^V \mM$, such that we have a
splitting $T\mM=(\mF \oplus T^V \mM)\oplus\mF_1^\perp$. Then
$\mF_1^\perp$ can be identified with $T\mM/(\mF \oplus T^V \mM)$
and carries a canonically induced metric $g^{\mF_1^\perp}$.
We also  denote $T^V\mM$ by $\mF_2^\perp$.

Let $g^F$ be a Euclidean metric on $F$, then it lifts to a Euclidean metric $g^{\mF}$ on $\mF$. 
Let $g^{T\mM}$ be the Riemannian metric on $T\mM$ defined by the
following orthogonal splitting,
\begin{align}\label{2.1}\begin{split}
       T\mM =   \mF\oplus \mF^\perp_1\oplus \mF^\perp_2,\ \ \ \ \ \
g^{T\mM }=   g^{\mF}\oplus g^{\mF^\perp_1}\oplus
g^{\mF^\perp_2}.\end{split}
\end{align}

 By   \cite[Lemma 5.2]{Co86}, $(\mM,\mF)$
admits  an almost isometric structure in the sense of Definition
\ref{t1.1}, with the metrics given in (\ref{1.4}) and/or
(\ref{2.1}).\footnote{We will use notations  similar to those in Section \ref{s1}, with the only difference  that
when  dealing with the Connes   fibration, we use caligraphic letters.
}  
 In particular, (\ref{1.5}) holds.\footnote{In fact, for
any  $X\in \Gamma(F) $, let  ${\mathcal X} \in\Gamma(\mF)$ denote
the lift of ${  X} $. Let $\varphi_t$ (with $t$ close to zero) be
the one parameter family of diffeomorphisms on $\mM$ generated  by
${\mathcal X} $. Then each $\varphi_t$  acts on the complete
transversal to $\mF$ in $\mM$. The differential  of $\varphi_t$,
when acting on the complete transversal,  maps each
$(\mF_1^\perp+\mF_2^\perp)_x$ ($x\in \mM$) to
$(\mF_1^\perp+\mF_2^\perp)_{\varphi_t(x)}$ and verifies \cite[Lemma
5.2]{Co86}. By taking derivative at $t=0$, one gets (\ref{1.5}).}

\comment{

\begin{defn}\label{t12.2} 
By a Connes   fibration over $(M,F)$ we mean a fibration $\pi:\mM\rightarrow M$ such that (i) there exists a splitting 
$T\mM=T^V\mM\oplus T^H\mM,$
where $T^V\mM$ is the vertical tangent bundle of the fibration, such that $F$ lifts to an integrable  subbundle $\mF\subset T^H\mM$;
 (ii) if we denote $T^V\mM=\mF_2^\perp$, then there exists a splitting 
$T^H\mM=\mF\oplus \mF_1^\perp$
as well as  Euclidean metrics $g^{\mF_1^\perp}$,  $g^{\mF_2^\perp}$ on $\mF_1^\perp$, $\mF_2^\perp$ such that  the foliation $(\mM,\mF)$ carries an associated  almost isometric foliation structure in the sense of Section \ref{s1.1};
(iii) there exists a smooth (embedded)  section $s:M\hookrightarrow \mM$. 
\end{defn}

}

One of the specific features of the Connes  fibration is that since
$\mF_2^\perp=T^V\mM$ is the vertical tangent bundle of a fibration,
  the following identity holds:
\begin{align}\label{2.2} [U,V]\in \Gamma\left(\mF_2^\perp\right)\ \
\ {\rm for}\ \ \ U,  \ V\in\Gamma\left(\mF_2^\perp\right).
\end{align}
That is,  Condition (C) in Definition \ref{t1.1c} holds for
$(\mM,\mF)$. Combining with (\ref{1.0}) and the second identity in
(\ref{1.5}), one sees that $\mF\oplus \mF_2^\perp$ is also an integrable
subbundle of $T\mM$.

For any $\beta>0$, $\varepsilon>0$, let $g_{\beta,
\varepsilon}^{T\mM}$ be the Riemannian metric on $T\mM$ defined as
in (\ref{1.8}). By (\ref{1.7}), (\ref{1.8}) and (\ref{2.2}), the
following identity holds   for the Connes   fibration,
\begin{align}\label{2.3}  \nabla^{\mF_2^\perp,\beta,\varepsilon}= {\nabla}^{\mF_2^\perp
}.
\end{align}
Equivalently, for any $X\in T\mM$ and $U,\
V\in\Gamma(\mF_2^\perp)$, one has
$ \langle\nabla_X^{\mF_2^\perp,\beta,\varepsilon}U,V \rangle
= \langle\nabla_X^{\mF_2^\perp }U,V \rangle.
$


Take a metric on ${TM/F}$. This is
equivalent to taking an embedded section $s:M\hookrightarrow \mM$
of the Connes fibration $\pi:\mM\rightarrow M$.

\subsection{A coordinate system near $s(M)$ }\label{s2.2a}

Let $s(M)\subset \mM$ be the image of the embedded  section $s:M\hookrightarrow \mM$. 
Consider the induced  fibration $s\circ  {\pi}:\mM\rightarrow s(M)$. In
what follows, for any $x\in s(M)$, we will denote the fiber
$\mM_{\pi(x)}$ simply by $\mM_x$.

 For any $x\in s(M)$, $Z\in T_x\mM_x=\mF_2^\perp|_x$ with $|Z|$ sufficiently small, let
 $\exp^{\mM_x}(tZ)$ be the geodesic in
 $\mM_x$ such that $\exp^{\mM_x}(0)=x$,
 $\frac{d \exp^{\mM_x}(tZ)}{dt}|_{t=0}=Z$.


 For any  $\alpha>0$,
 let $\psi:U_\alpha(\mF_2^\perp)=\{ (x,Z):x\in s(M),\, Z\in \mF_{2}^\perp|_x,\, |Z|<\alpha
 \}\rightarrow \mM$ be defined such that for any $x\in s(M)$, $Z\in T_x\mM_x$ with $|Z|<\alpha$, 
\begin{align}\label{1311}
 \psi (x,Z)\mapsto \exp^{\mM_x}(Z).
 \end{align}
Clearly, $\psi$ is a diffeomorphism from $U_\alpha(\mF_2^\perp)$ to its image, when $\alpha $ is sufficiently small, which we fix it now. In case of no confusion, we will also use the notation $(x,Z) $ to denote  its image $\psi(x,Z)$. 
In particular, $ (x,0)=x$.
We also denote the geodesic $\exp^{\mM_x}(tZ)$   
 by $tZ$.

 On $\psi(U_{\alpha}(\mF_2^\perp))\simeq U_{\alpha}(\mF_2^\perp)$, the volume form $dv_{\mM}$ can be
 written as
\begin{align}\label{111}
 dv_{\mM}(x,Z)=k(x,Z)dv_{\mF^\perp_{2,x}}(Z)dv_{s(M)}(x),
 \end{align}
 where $dv_{\mF^\perp_{2,x}}$ is the volume form on $\mF^\perp_{2,x}=\mF^\perp_2|_x$ which in turn determines
 the corresponding volume form on $\mM_x\cap \psi(U_{\alpha}(\mF_2^\perp))$, $dv_{s(M)}$ is
 the volume form on $s(M)$ with respect to the restricted metric,
 and $k(x,Z)>0$ is the  function determined by (\ref{1311}) and
 (\ref{111}).\footnote{As $\mF_2^\perp|_{s(M)}$ need
 not be orthogonal to $Ts(M)$, $k(x,0) $ need not be constant on $s(M)$   (compare with \cite[(8.22)]{BL91}).  }

 In what follows, we will also denote $dv_{\mF^\perp_{2,x}}$ by
 $dv_{\mM_x}$.

\subsection{Adiabatic limit near     $s(M)$
}\label{s2.2} 

Recall that for $\beta>0$ and $\varepsilon>0$, $g_{\beta,\varepsilon}^{T\mM}$ is the
Riemannian  metric on $T\mM$ defined by
\begin{align}\label{2.7}
g^{T\mM}_{\beta,\varepsilon }= \beta^2 g^{\mF}\oplus \frac{1}{
\varepsilon^{2}} g^{\mF^\perp_1}\oplus g^{\mF^\perp_2}. 
\end{align}

Since we assume $TM$ is spin, $\mF\oplus \mF_1^\perp=\pi^*(TM)$ is spin, and we take 
  $D^{\mF, \beta,\varepsilon}$ to be the sub-Dirac operator
constructed in (\ref{1.43}) with respect to $g_{\beta,\varepsilon}^{T\mM}$, but with $S(\mF)\widehat\otimes \Lambda^*(\mF^\perp_1)$ being  replaced by $S(\mF\oplus\mF^\perp_1)$.\footnote{In this section, for simplicity, we will not consider the twisted bundles  $\phi_1(\mF_1^\perp)$ and  $\phi_2(\mF_2^\perp)$.}

By (\ref{2.7}) one has
\begin{align}\label{m56}
dv_{(T\mM,g_{\beta,\varepsilon}^{T\mM})}=\frac{\beta^q}{\varepsilon^{q_1}}dv_{(T\mM,g^{T\mM})}.
\end{align}

For simplicity, from now on, by   $L^2$-norms we will mean the
$L^2$-norms with respect to the volume form
$dv_{(T\mM,g^{T\mM})}$, i.e., for any $s\in\Gamma(W
(\mF,\mF_1^\perp,\mF_2^\perp )   )$ with compact support, one has
 \begin{align}\label{k1}
 \|s\|_0^2:=\int_{\mM}\langle s, s\rangle_{\beta,\varepsilon}
 dv_{(T\mM,g^{T\mM})},
 \end{align}
 where the subscripts ``$\beta,\ \varepsilon$'' indicate that the pointwise
 inner product  is induced from $g_{\beta,\varepsilon}^{T\mM}$.

 From (\ref{m56}) and (\ref{k1}), one sees that the operators
 which are formally self-adjoint with respect to the usual
 $L^2$-norm, which is associated with the volume form
 $dv_{(T\mM,g_{\beta,\varepsilon}^{T\mM})}$, is still formally
 self-adjoint with respect to the $L^2$-norm defined in
 (\ref{k1}).

By (\ref{1.50}), one knows that when $\beta$, $\varepsilon>0$ are
sufficiently small,   the following identity holds on $U_{\alpha}(\mF_2^\perp)$:
 \begin{align}\label{m2}
\left(D^{\mF, \beta,\varepsilon}\right)^2=-\Delta^{\mF, \beta,\varepsilon}+\frac{k^{\mF}}{4\beta^2}
+O\left(\frac{1}{\beta}+\frac{\varepsilon^2}{\beta^2}  \right).
 \end{align}

 Let $h_1,\, \cdots,\, h_{\dim \mM}$ be an oriented orthonormal basis of
$(T\mM,g_{\beta,\varepsilon}^{T\mM})$.  Then
for any $s\in\Gamma(W (\mF,\mF_1^\perp,\mF_2^\perp )  )$ having compact support, the following identity holds,\footnote{From now on,   $\nabla^{\mF, \beta,\varepsilon}$ will denote the canonical connection   on $ W (\mF,\mF_1^\perp,\mF_2^\perp  )$. This should not be confused with the connection on $\mF$ as in (\ref{1.9a}), which will not appear in the rest of this section.}
 \begin{align}\label{3.1}
 \left\langle -\Delta^{\mF, \beta,\varepsilon} s,s\right\rangle
=\sum_{i=1}^{\dim \mM}\left\|
\nabla^{\mF, \beta,\varepsilon}_{h_i}s \right\|_0^2.
 \end{align}

On the other hand, for any $\sigma\in\Gamma((S(\mF\oplus\mF^\perp_1 ) )|_{s(M)})$,
 similarly as in (\ref{k1}), we define its $L^2$-norm by
 \begin{align}\label{k2}
 \|\sigma\|_0^2:=\int_{s(M)}\langle \sigma, \sigma\rangle_{\beta,\varepsilon}
 dv_{ s(M)},
 \end{align}
where, as in (\ref{111}), $dv_{ s(M)}$ is the volume form on $s(M)$ associated to the restricted metric from $g^{T\mM}|_{s(M)}$. 

In what follows, we will also denote $dv_{(T\mM,g^{T\mM})}$ by $dv_{\mM}$ as in  (\ref{111}). 

\subsection{An embedding from sections on $s(M)$ to sections on
$\mM$ }\label{s2.3}

Recall that $\Lambda ^*(\mF^\perp_2)=\oplus_{i=0}^{{\rm
rk}(\mF_2^\perp)}\Lambda^i(\mF^\perp_2)$, with
$\Lambda^0(\mF^\perp_2)={\bf C}$ (or ${\bf R}$ in the case where  we consider real operators). Let
 \begin{align}\label{3.2}
 Q:\Lambda^* (\mF^\perp_2) \rightarrow
\Lambda^0(\mF^\perp_2)={\bf C}
 \end{align}
 denote the corresponding orthogonal projection. Let
 \begin{align}\label{3.3}
 i_Q:\Lambda^0(\mF^\perp_2) \hookrightarrow
\Lambda ^*(\mF^\perp_2)
 \end{align}
 denote the canonical inclusion. In view of (\ref{1.38}) and (\ref{1.41}), the projection $Q$ and the
 embedding $i_Q$ induce the following canonical orthogonal projection and
 embedding, which we will denote by the same notations,
 \begin{align}\label{3.4}
 Q:W \left(\mF,\mF_1^\perp,\mF_2^\perp \right) 
\rightarrow S\left(\mF\oplus\mF^\perp_1\right) ,
 \end{align}
   \begin{align}\label{3.5}
 i_Q:S\left(\mF\oplus\mF_1^\perp\right) \hookrightarrow
W\left(\mF,\mF_1^\perp,\mF_2^\perp \right) .
 \end{align}


 Let $^Q
\nabla^{\mF, \beta,\varepsilon}$ be the induced  connection on
$S(\mF\oplus\mF^\perp_1) $ defined by
 \begin{align}\label{3.6}
^Q \nabla^{\mF, \beta,\varepsilon}=Q\nabla^{\mF, \beta,\varepsilon}i_Q.
 \end{align}
Clearly, $^Q
\nabla^{\mF, \beta,\varepsilon}$ is a Euclidean connection.

Let $\sigma\in \Gamma((S(\mF\oplus\mF^\perp_1) )|_{s(M)})$. For
 any $(x,Z)\in U_\alpha(\mF_2^\perp)$, let $\tau\sigma (x,Z)\in (S(\mF\oplus\mF^\perp_1)  )|_{\psi(x,Z)}$ be the parallel transport of
 $\sigma(x)$
 along the geodesic $(x,tZ),\ 0\leq t\leq 1$, with respect to the connection $^Q \nabla^{\mF, \beta,\varepsilon}$.

 Let $\gamma $ be a smooth function on ${\bf R}$ such that
$\gamma(b)=1 $ if $b\leq \frac{\alpha}{3}$, while  $\gamma(b)=0 $
if $b\geq \frac{2\alpha}{3}$.

 For $T>0$,  $x\in s(M)$, set
\begin{align}\label{m21}
\alpha_{T }(x)=\int_{\mM_x}\exp\left(- {T|
 Z|^2} \right) \gamma^2\left( {|Z|} \right)dv_{\mM_x}(Z).
\end{align}
Clearly, $\alpha_T (x)$ is constant on $s(M)$, which
we will denote by $\alpha_{T }$.

Inspired by \cite[Definition 9.4]{BL91}, for $T>0$,
  let
$$J_{T,\beta,\varepsilon }:\Gamma\left((S(\mF\oplus\mF^\perp_1) )|_{s(M)}\right) \longrightarrow \Gamma\left(W
(\mF,\mF_1^\perp,\mF_2^\perp ) \right)$$ be the embedding defined by
\begin{align}\label{m22}
 J_{T,\beta,\varepsilon }:\sigma \mapsto \left.\left(J_{T,\beta,\varepsilon}\sigma\right)\right|_{\psi(x,  Z)}=
\left( k(x,  Z) \alpha_{T }
\right)^{-\frac{1}{2}}\gamma\left( {|Z|} \right)\exp\left(- \frac{T|Z|^2}{2 }\right)i_Q
({\tau\sigma}(x,  Z)) .
\end{align}

By the definition of $\gamma$, one sees that $J_{T ,\beta,\varepsilon}$
is well-defined. Moreover, in view of (\ref{111}), (\ref{k1}),
(\ref{k2}), (\ref{m21}) and (\ref{m22}), one sees that   $ J_{T ,\beta,\varepsilon}$ is an isometric embedding.

Clearly, any $J_{T,\beta,\varepsilon } \sigma$ has compact support in
$\mM_{2 \alpha/3} $.  Let $E_{T,\beta,\varepsilon }'$
denote the image of $\Gamma ((S(\mF\oplus\mF^\perp_1)  )|_{s(M)})$ under
 $J_{T,\beta,\varepsilon }$. Let $p_{T ,\beta,\varepsilon}$ denote the orthogonal
 projection from the $L^2$-completion of $\Gamma(W (\mF,\mF_1^\perp,\mF_2^\perp  )  )$
to the $L^2$-completion of $E_{T,\beta,\varepsilon }'$, which we denote by   $ {E}_{T,\beta,\varepsilon }$.

\subsection{An estimate  for    $\|p_{T,\beta,\varepsilon }D^{\mF, \beta,\varepsilon} p_{T,\beta,\varepsilon }\|^2_0$}\label{s2.4}

Let $f_1,\,\cdots,\, f_{q+q_1}$ be an   orthonormal basis
of $(\mF\oplus\mF^\perp_1)|_{s(M)}$ with respect to $(g^\mF\oplus
g^{\mF_1^\perp})|_{s(M)}$, where $f_1,\,\cdots,\, f_q$ is an  
orthonormal basis of $\mF|_{s(M)}$ and thus $f_{q+1},\,\cdots,\,
f_{q+q_1}$ is an   orthonormal basis of
$\mF_1^\perp|_{s(M)}$. Let $e_1,\,\cdots,\, e_{q_2}$ be an
  orthonormal basis of $\mF^\perp_2|_{s(M)}$ with respect
to $g^{\mF_2^\perp}|_{s(M)}$.

For any $f\in (\mF \oplus \mF_1^\perp)|_{s(M)}$ (resp.
$e\in\mF^\perp_2|_{s(M)}$), let $\tau f\in\Gamma(\mF\oplus
\mF_1^\perp)$ (resp. $\tau e\in\Gamma(\mF^\perp_2)$) be  such that for any $(x,Z)\in U_{\alpha}(\mF_2^\perp)$, $\tau f
|_{\psi(x,Z)}$ (resp.   $\tau e|_{\psi(x,Z)}$) is the parallel transport of $f_x$
(resp.
 $e_x$) along the geodesic $ {(x,tZ)},\, 0\leq t\leq 1,$ with respect to the Euclidean 
connection $(p+p^\perp_1)\nabla^{T\mM,\beta,\varepsilon}(p+p^\perp_1)$
(resp. $\nabla^{\mF^\perp_2,\beta,\varepsilon}=\nabla^{\mF^\perp_2 }$).

Clearly, $\beta^{-1}\tau f_i$ ($1\leq i\leq q$),
$\varepsilon \tau f_j$ ($q+1\leq j\leq q+q_1$)
and $ \tau e_k$ ($1\leq k\leq q_2$) form an orthonormal
basis of $(T\mM,g_{\beta,\varepsilon}^{T\mM })$.

Let $\tau Z\in\Gamma(\psi(U_\alpha(\mF_2^\perp)))$ be the tautological section defined by
\begin{align}\label{12.6mb}
(\tau Z)|_{\psi(x,Z)}=\sum_{k=1}^{q_2}z_k\,\tau e_k,
\end{align}
with $Z=\sum_{k=1}^{q_2}z_k\,e_k\in\mF_2^\perp|_x$. 
In case of no confusion, we also denote $\tau Z$ by $Z$.  

Let $c_{\beta,\varepsilon}(\cdot)$ be the Clifford action associated to $g^{T\mM}_{\beta,\varepsilon}$. For any $X,\ Y\in T\mM$, one has
\begin{align}\label{3.8d}
c_{\beta,\varepsilon}(X)c_{\beta,\varepsilon}(Y)+c_{\beta,\varepsilon}(Y)c_{\beta,\varepsilon}(X)=-2\langle X,Y\rangle_{g^{T\mM}_{\beta,\varepsilon}}.
\end{align}

By (\ref{1.43}), one has
\begin{multline}\label{3.8}
D^{\mF, \beta,\varepsilon}=\beta^{-1}\sum_{i=1}^{q}  c_{\beta,\varepsilon}\left(\beta^{-1}\tau
f_i\right)\nabla_{\tau f_i}^{\mF,  \beta,\varepsilon}  +
\varepsilon \sum_{k=q+1}^{q+q_1} c_{\beta,\varepsilon}\left(\varepsilon\tau
f_k\right)\nabla_{\tau f_k}^{\mF,  \beta,\varepsilon} 
\\
+  \sum_{s=1}^{q_2} c_{\beta,\varepsilon}\left(\tau e_s\right)\nabla_{\tau
e_s}^{\mF,  \beta,\varepsilon}.
\end{multline}

We   state a key asymptotic estimate result
 for $\|p_{T,\beta,\varepsilon }D^{\mF, \beta,\varepsilon} p_{T,\beta,\varepsilon }\|^2_0$, when $T\rightarrow +\infty$ and $\beta$, $\varepsilon>0$ being small, as follows.

\begin{prop}\label{t12.4}
There exist $C'>0$,   $0<\delta,\ \beta_0,\ \varepsilon_0<1$ and  $T_0>0$ such that for any $0<\beta\leq\beta_0,\ 
0<\varepsilon\leq \varepsilon_0$, there exists $C_{\beta,\varepsilon}>0$ for which  the following inequality holds for any   $T\geq T_0$ and
$\sigma\in \Gamma ((S(\mF\oplus\mF^\perp_1)  )|_{s(M)})$:
\begin{multline}\label{12.6}
 \left\|p_{T,\beta,\varepsilon}D^{\mF, \beta,\varepsilon}
J_{T,\beta,\varepsilon}\sigma\right\|_0^2
\geq
 \int_{s(M)} \left(\frac{k^\mF }{4\beta^2}-
\frac{1}{4\beta^2}\sum_{i=1}^q \sum_{t=q+1}^{q+q_1} \left|p_1^\perp
\nabla^{ T\mM,\beta,
\varepsilon}_{f_t}\left(\nabla^{\mF_2^\perp}_{f_i}  Z\right) \right|^2\right)|\sigma|^2dv_{s(M)}
\\
- C'\left(\frac{1}{\beta}+\frac{\varepsilon^\delta}{\beta^4} \right)  \int_{s(M)}|\sigma|^2dv_{s(M)}
+  \frac{\varepsilon^{\delta} }{8\beta^2} \sum_{k=1}^q\int_{s(M)}  \left|
\,^Q\nabla^{\mF, \beta,\varepsilon}_{ f_k}(\tau\sigma)\right|^2 dv_{s(M)}
\end{multline} 
$$
+ \frac{\varepsilon^{2+\delta} }{16 } \sum_{k=q+1}^{q+q_1}\int_{s(M)}  \left|
\,^Q\nabla^{\mF, \beta,\varepsilon}_{ f_k}(\tau\sigma)\right|^2 dv_{s(M)}
-\frac{C_{\beta,\varepsilon}}{\sqrt{T}} \int_{s(M)} \left(
|\sigma|^2+\sum_{k=1}^{q+q_1}\left|
\,^Q\nabla^{\mF, \beta,\varepsilon}_{
f_k}(\tau\sigma)\right|^2\right) dv_{s(M)}.
$$
\end{prop}

\begin{rem}\label{t1211}
In the right hand side of (\ref{12.6}), since $Z|_{s(M)}\equiv 0$ and $(\tau f_j)|_{s(M)}=f_j$ for any $1\leq j\leq q+q_1$, one verifies by (\ref{12.6mb}) that for any $1\leq i\leq q$, $q+1\leq t\leq q+q_1$, the following identity holds on $s(M)$,
\begin{align}\label{12.6m}
p_1^\perp
\nabla^{ T\mM,\beta,
\varepsilon}_{f_t}\left(\nabla^{\mF_2^\perp}_{f_i}  Z\right) :=\left.\left(p_1^\perp
\nabla^{ T\mM,\beta,
\varepsilon}_{\tau f_t}\left(\nabla^{\mF_2^\perp}_{\tau f_i}  Z\right) \right)\right|_{s(M)}=\sum_{k=1}^{q_2}f_i(z_k)\,p_1^\perp\nabla^{ T\mM,\beta,
\varepsilon}_{f_t}  \tau e_k ,
\end{align}
where   $f_i(z_k)$ is the restriction on $s(M)$ of   $\tau f_i(z_k) \in C^\infty(\psi(U_\alpha(\mF_2^\perp)))$.  Also, for any $1\leq j\leq q+q_1$, one denotes on $s(M)$ that 
\begin{align}\label{12.6ma}
^Q\nabla^{\mF, \beta,\varepsilon}_{
f_j}(\tau\sigma)=\left.\left(^Q\nabla^{\mF, \beta,\varepsilon}_{\tau
f_j}(\tau\sigma)\right)\right|_{s(M)} .
\end{align}
\end{rem}

\comment{

\begin{rem}\label{t12.5}
Proposition \ref{t12.4} holds on any Connes type fibration over $(M,F)$. However, the second term in the right hand 
side of (\ref{12.6}) becomes an obstruction to get the positivity of the term in the left hand side of (\ref{12.6}), which is our ultimate goal. 
Indeed, this obstruction comes from the fact that $\mF|_{s(M)}$ need not be included in $Ts(M)$. Thus 
if one writes $Z=\sum_{i=1}^{q_2}z_i\tau e_i$ near $s(M)$, then one need not have $f_i(z_j)=0$ on $s(M)$.\footnote{Indeed, one has on $s(M)$ that $p_1^\perp \nabla^{ T\mM,\beta, \varepsilon}_{f_t}(\nabla^{\mF_2^\perp}_{f_i}Z )= \sum_{j=1}^{q_2}f_i(z_j)p_1^\perp\nabla^{T\mM,\beta,\varepsilon}_{f_t}\tau e_j$.} This is the  key difference with respect to the situation considered in \cite[Chapters 8 and 9]{BL91}.\footnote{One may also want to consider the splitting $T\mM|_{s(M)}=Ts(M)\oplus N$ on $s(M)$, similarly as in \cite{BL91}, where $N$ is the  normal bundle orthogonal to $Ts(M)$ in  $T\mM|_{s(M)}$. However, then this splitting would depend on $\beta$ and $\varepsilon$ which would cause  other kinds of troubles.}
\end{rem}

}

The basic idea of the proof of Proposition \ref{t12.4} is very natural. 
Indeed, since $p_{T,\beta,\varepsilon }:L^2(W (\mF,\mF_1^\perp,\mF_2^\perp
)   )\rightarrow  {E}_{T,\beta,\varepsilon }$ is an orthogonal
projection, for any $\sigma\in\Gamma ((S(\mF\oplus\mF^\perp_1) )|_{s(M)})$, one has
\begin{align}\label{m55}
 \left\|p_{T,\beta,\varepsilon}D^{\mF, \beta,\varepsilon}
J_{T,\beta,\varepsilon}\sigma\right\|_0^2=\left\|
D^{\mF, \beta,\varepsilon}
 J_{T,\beta,\varepsilon}\sigma\right\|_0^2
 -\left\|\left(1-p_{T,\beta,\varepsilon}\right)D^{\mF, \beta,\varepsilon}
J_{T,\beta,\varepsilon}\sigma\right\|_0^2.
\end{align}


In view of (\ref{m56}) and (\ref{k1}), the operator
$D^{\mF, \beta,\varepsilon}$ is formally
self-adjoint with respect to the $L^2$-norm in (\ref{k1}).  Thus,
the first term in the right hand side of (\ref{m55})  can be
estimated by using (\ref{m2}) and (\ref{3.1}). So we need to estimate
the second term in the right hand side of (\ref{m55}), to make it as small as possible.

Set 
\begin{align}\label{i1}
I_1=\sum_{i\neq j,\ 1\leq
i,\, j\leq q}\left\langle \left(1-p_{T,\beta,\varepsilon}\right)c_{\beta,\varepsilon}\left(\beta^{-1}\tau
f_i\right) {\nabla}^{\mF,\beta,\varepsilon}_{\beta^{-1}\tau f_i}
J_{T,\beta,\varepsilon}\sigma, c_{\beta,\varepsilon}\left(\beta^{-1}\tau
f_j\right) {\nabla}^{\mF,\beta,\varepsilon}_{\beta^{-1}\tau f_j}
J_{T,\beta,\varepsilon}\sigma\right\rangle,
\end{align}
\begin{align}\label{i2}
I_2=\sum_{i\neq j,\ q+1\leq i,\, j\leq
q+q_1}\left\langle \left(1-p_{T,\beta,\varepsilon}\right)c_{\beta,\varepsilon}(\varepsilon\tau
f_i) {\nabla}^{\mF,\beta,\varepsilon}_{\varepsilon\tau f_i}
J_{T,\beta,\varepsilon}\sigma, c_{\beta,\varepsilon}(\varepsilon\tau
f_j) {\nabla}^{\mF,\beta,\varepsilon}_{\varepsilon\tau f_j}
J_{T,\beta,\varepsilon}\sigma\right\rangle,
\end{align}
\begin{align}\label{i3}
I_3=\sum_{i\neq j,\ 1\leq i,\, j\leq q_2}\left\langle
\left(1-p_{T,\beta,\varepsilon}\right)c_{\beta,\varepsilon}(\tau
e_i) {\nabla}^{\mF,\beta,\varepsilon}_{\tau e_i}
J_{T,\beta,\varepsilon}\sigma, c_{\beta,\varepsilon}(\tau
e_j) {\nabla}^{\mF,\beta,\varepsilon}_{\tau e_j}
J_{T,\beta,\varepsilon}\sigma\right\rangle,
\end{align}
\begin{align}\label{i4}
I_4=2 \sum_{i=1}^q\sum_{j=q+1}^{q+q_1} {\rm Re}\left( \left\langle \left(1-p_{T,\beta,\varepsilon}\right)c_{\beta,\varepsilon}\left(\beta^{-1}\tau
f_i\right) {\nabla}^{\mF,\beta,\varepsilon}_{\beta^{-1}\tau f_i}
J_{T,\beta,\varepsilon}\sigma, c_{\beta,\varepsilon}(\varepsilon\tau
f_j) {\nabla}^{\mF,\beta,\varepsilon}_{\varepsilon\tau f_j}
J_{T,\beta,\varepsilon}\sigma\right\rangle\right),
\end{align}
\begin{align}\label{i5}
I_5=2
\sum_{i=1}^q\sum_{j=1}^{q_2} {\rm Re}\left( \left\langle
\left(1-p_{T,\beta,\varepsilon}\right)c_{\beta,\varepsilon}\left(\beta^{-1}\tau
f_i\right) {\nabla}^{\mF,\beta,\varepsilon}_{\beta^{-1}\tau f_i}
J_{T,\beta,\varepsilon}\sigma, c_{\beta,\varepsilon}(\tau
e_j) {\nabla}^{\mF,\beta,\varepsilon}_{\tau e_j}
J_{T,\beta,\varepsilon}\sigma\right\rangle\right),
\end{align}
\begin{align}\label{i6}
I_6=2
\sum_{i=q+1}^{q+q_1}\sum_{j=1}^{q_2} {\rm Re}\left( \left\langle
\left(1-p_{T,\beta,\varepsilon}\right)c_{\beta,\varepsilon}(\varepsilon\tau
f_i) {\nabla}^{\mF,\beta,\varepsilon}_{\varepsilon\tau
f_i} J_{T,\beta,\varepsilon}\sigma, c_{\beta,\varepsilon}(\tau
e_j) {\nabla}^{\mF,\beta,\varepsilon}_{\tau e_j}
J_{T,\beta,\varepsilon}\sigma\right\rangle\right) .
\end{align}

By (\ref{3.8}) and (\ref{i1})-(\ref{i6}), one has
\begin{multline}\label{3.8a}
\left\|\left(1-p_{T,\beta,\varepsilon}\right)D^{\mF , \beta,\varepsilon}
J_{T, \beta,\varepsilon}
 \sigma\right\|_0^2
=\sum_{k=1}^6 I_k
+\sum_{i=1}^q
\left\|\left(1-p_{T,\beta,\varepsilon}\right)c_{\beta,\varepsilon}\left(\beta^{-1}\tau
f_i\right) {\nabla}^{\mF, \beta,\varepsilon}_{\beta^{-1}\tau f_i}
J_{T,\beta,\varepsilon}\sigma\right\|_0^2
\\
+\sum_{i=q+1}^{q+q_1}
\left\|\left(1-p_{T,\beta,\varepsilon}\right)c_{\beta,\varepsilon}(\varepsilon\tau
f_i) {\nabla}^{\mF, \beta,\varepsilon}_{\varepsilon\tau f_i}
J_{T,\beta,\varepsilon}\sigma\right\|_0^2
+ \sum_{i=1}^{q_2}
\left\|\left(1-p_{T,\beta,\varepsilon}\right)c_{\beta,\varepsilon}(\tau
e_i) {\nabla}^{\mF, \beta,\varepsilon}_{\tau e_i}
J_{T,\beta,\varepsilon}\sigma\right\|_0^2
.
\end{multline}

Naturally, we need to study the behaviour when $T\rightarrow +\infty$ of each term in the right hand side of (\ref{3.8a}).  
Due to the Gaussian factor $\exp(-T|Z|^2/2)$  in (\ref{m22}), one sees as in \cite[Chapters 8 and 9]{BL91} that when $T\rightarrow +\infty$, all terms in (\ref{3.8a}) localize onto $s(M)$.  All one need is to choose   the rescaling factors $\beta$, $\varepsilon$ conveniently such that   the estimate goes as desired.  For this   the geometric nature of the Connes  fibration plays an essential role.

  The fact that the right hand side of (\ref{3.8a}) has nine terms, with each term  further  splits   into   four or even more terms in the process of estimation,  partly explains the length of the computations, which are purely routine and elementary.  

\comment{

Proposition \ref{t12.4} will be applied to a specific Connes type fibration in Section \ref{s2.8a} so that one can go further to get a positivity result for the left hand side of (\ref{12.6}). 
The reader who wishes to see directly that positivity can skip temporarily the next three subsections where a detailed proof of Proposition \ref{t12.4} is given.

}

\subsection{Estimates of  the terms $I_k$, $1\leq k\leq 6$, Part I}\label{s2.5}

Before going on, we set a notational  convention: in what follows,
by $O(|Z|^2)$ and $O(\frac{1}{\sqrt{T}})$, we will mean
$O_{\beta,\varepsilon} (|Z|^2)$ and $O_{\beta,\varepsilon}(\frac{1}{\sqrt{T}})$,
i.e., the associated estimating constants may depend on $\beta>0$ and
$\varepsilon>0$. While for other $O(\cdots)$ terms, the
corresponding estimating constants will not depend on $\beta>0$ and
$\varepsilon>0$, unless there appear the  subscripts ``$\beta$'' and/or ``$\varepsilon$''
which will indicate that the corresponding estimating coefficient   will depend  on $\beta$ and/or $\varepsilon$.

For brevity, let $f_{T }$ be the smooth function on
$\mM$ defined by that on any $(x,Z)\simeq \psi(x,Z)$, one has, 
  \begin{align}\label{3.9}
 f_{T}(x,  Z)=\left( k(x,  Z)
\alpha_{T }
\right)^{-\frac{1}{2}}\gamma\left( {|Z|} \right)\exp\left(-\frac{T|Z|^2}{2 }\right).
 \end{align}
 Then one can rewrite $J_{T,\beta,\varepsilon}\sigma$ in (\ref{m22}) as
\begin{align}\label{3.10}
\left(J_{T,\beta,\varepsilon} \sigma\right)(x, Z)=f_{T }(x,
Z) i_Q ({\tau\sigma}(x,  Z )).
 \end{align}

 From now on, in case of no confusion, we will omit
 $i_Q$.


 \begin{lemma}\label{t11.1} (i) For any $\sigma\in \Gamma ((S(\mF\oplus\mF^\perp_1)  )|_{s(M)})$ and any $f\in C^\infty(\mM)$ with
${\rm Supp}(f)\subset \psi(U_\alpha(\mF_2^\perp))$, one has
\begin{align}\label{11.1}
\left( p_{T,\beta,\varepsilon}(f\,\tau\sigma )\right)(x,Z) = \left(
\int_{\mM_x} f_{T }(x,  Z') f(x,Z'){k(x,  Z')}
 dv_{\mM_x}(  Z')\right)\left( J_{T,\beta,\varepsilon} \sigma\right)(x,
Z);
\end{align}
(ii) For any $u\in\Gamma (W (\mF,\mF_1^\perp,\mF_2^\perp ) )$ with
${\rm Supp}(u)\subset \psi(U_\alpha(\mF_2^\perp))$, one has
\begin{align}\label{11.2}
  p_{T,\beta,\varepsilon}\left(f_{T } u\right)
  =J_{T,\beta,\varepsilon} \left((Qu)|_{s(M)}\right) +p_{T,\beta,\varepsilon}\left(O_{\beta,\varepsilon}(|Z|)\right).
\end{align}
 \end{lemma}
 \begin{proof} Take any $u\in\Gamma (W (\mF,\mF_1^\perp,\mF_2^\perp ) )$.
Then for any $(x,Z)\in U_\alpha(\mF_2^\perp)$, $(Qu)|_{\psi(x,Z)}$ determines a unique
element $u'\in(S(\mF\oplus\mF^\perp_1) )|_{x}$ such that
 $(\tau u')|_{\psi(x,Z)}=(Qu)|_{\psi(x,Z)}$. We denote this element by
 $\tau^{-1}((Qu)|_{(x,Z)})$.

 Then one verifies easily that (compare with \cite[(9.6) and
(9.13)]{BL91})
\begin{align}\label{11.2c}
 \left( p_{T,\beta,\varepsilon}  u\right)(x,Z)
  =f_{T }(x,Z)\left( \tau \int_{\mM_x}f_{T }(x,Z')k(x,Z')
 \, \tau^{-1}\left((Qu)|_{(x,Z')}\right)dv_{\mM_x}(Z')\right)(x,Z).
\end{align}

Formulas (\ref{11.1}) and (\ref{11.2}) follow  from (\ref{11.2c})
easily.
 \end{proof}

 \begin{lemma}\label{t11.2} For any $X\in
 \Gamma((\mF\oplus\mF_1^\perp)|_{s(M)})$, one has
\begin{align}\label{11.3}
  p_{T,\beta,\varepsilon} c_{\beta,\varepsilon}(\tau X)=c_{\beta,\varepsilon}(\tau X) p_{T,\beta,\varepsilon}.
\end{align}
 \end{lemma}
 \begin{proof} For any $\sigma\in\Gamma ((S(\mF \oplus\mF_1^\perp) )|_{s(M)})$ and
 $X\in\Gamma ((\mF\oplus \mF_1^\perp)|_{s(M)})$, we claim that
\begin{align}\label{11.3a}
    c_{\beta,\varepsilon}(\tau X)\tau \sigma= \tau\left (c_{\beta,\varepsilon}(X)\sigma\right).
\end{align}
Indeed, it is easy to verify that
\begin{multline}\label{11.3b}
  ^Q  {\nabla}^{\mF,\beta,
 \varepsilon}_Z\left( c_{\beta,\varepsilon}(\tau
 X)\tau \sigma\right)=Q\left(c_{\beta,\varepsilon}\left( {\nabla}^{T\mM,\beta,\varepsilon}_Z(\tau X)\right)\tau \sigma\right)+c_{\beta,\varepsilon}(\tau
 X) \, ^Q  {\nabla}^{\mF,\beta,
 \varepsilon}_Z(\tau\sigma)\\
=c_{\beta,\varepsilon}\left(\left(p+p_1^\perp\right)\nabla^{T\mM,\beta,\varepsilon}_Z(\tau
X)\right)\tau \sigma =0.
\end{multline}
From (\ref{11.3b}), one sees that $c_{\beta,\varepsilon}(\tau X)\tau \sigma$ is the
parallel transport of $(c_{\beta,\varepsilon}(\tau X)\tau \sigma)|_{s(M)}=c_{\beta,\varepsilon}(  X)
\sigma$, from which (\ref{11.3a}) follows.

 Now for any $\sigma\in\Gamma ((S(\mF\oplus\mF_1^\perp)  )|_{s(M)})$ and $u\in\Gamma (W
(\mF,\mF_1^\perp,\mF_2^\perp ) 
 )$ with ${\rm Supp}(u)\subset
\psi(U_\alpha(\mF_2^\perp))$,  one verifies via (\ref{11.3a}) that
\begin{multline}\label{11.4}
 \left\langle p_{T,\beta,\varepsilon} c_{\beta,\varepsilon}(\tau X)u,J_{T,\beta,\varepsilon}
 \sigma\right\rangle= \left\langle   c_{\beta,\varepsilon}(\tau X)u,J_{T,\beta,\varepsilon}
 \sigma\right\rangle= -\left\langle   u,c_{\beta,\varepsilon}(\tau X)J_{T,\beta,\varepsilon}
 \sigma\right\rangle
\\
= -\left\langle   u,J_{T,\beta,\varepsilon}(c_{\beta,\varepsilon}(  X)
 \sigma)\right\rangle
 =-\left\langle p_{T,\beta,\varepsilon}  u,J_{T,\beta,\varepsilon}(c_{\beta,\varepsilon}(  X)
 \sigma)\right\rangle=-\left\langle p_{T,\beta,\varepsilon}  u,c_{\beta,\varepsilon}(\tau  X)J_{T,\beta,\varepsilon}
 \sigma\right\rangle
\\
=\left\langle c_{\beta,\varepsilon}(\tau  X)p_{T,\beta,\varepsilon}  u,J_{T,\beta,\varepsilon}
 \sigma\right\rangle,
\end{multline}
from which (\ref{11.3}) follows.
 \end{proof}

 For any $X\in
 \Gamma((\mF\oplus\mF_1^\perp)|_{s(M)})$,
 by  (\ref{11.3}), one finds
\begin{align}\label{3.12}
\left(1-p_{T,\beta,\varepsilon}\right) c_{\beta,\varepsilon}(\tau X)=c_{\beta,\varepsilon}(\tau
X)\left(1-p_{T,\beta,\varepsilon}\right).
 \end{align}

Let $f_i'$, $1\leq i\leq q$ (resp. $f_j'$, $q+1\leq j\leq q+q_1$)
be an orthonormal basis of $(\mF,g^{\mF})$ (resp.
$(\mF_1^\perp,g^{\mF_1^\perp})$) on $U_\alpha(\mF_2^\perp)$, which does not depend
on $\beta$ and $\varepsilon$, and which satisfies $f_i'|_{s(M)}=f_i$ (resp.
$f_j'|_{s(M)}=f_j$).

Without loss of generality, we assume that $f_1',\,\cdots,\,f_q'$ are lifted from corresponding elements on $M$. That is, there is an orthonormal basis $\widehat{f}_1,\,\cdots,\,\widehat{f}_q$ of $(F, g^F) $ such that
\begin{align}\label{3.12a}
 f'_i=\pi^*\widehat{f}_i,\ \ \ 1\leq i\leq q.
 \end{align}

\begin{lemma}\label{t11.3} The following asymptotic formulas at $(x,
Z)$ (i.e., $\psi(x,Z)$) with $x\in s(M)$, $Z\in \mF_2^\perp|_x$, hold near $s(M)$: (i) if
$1\leq i\leq q$, then
\begin{align}\label{12.1}
 \tau f_i=f_i'+ \sum_{m=q+1}^{q+q_1}O\left( \varepsilon^{2}
 |Z|\right)f_m'+O\left(|Z|^2\right);
 \end{align}
 (ii) if $q+1\leq i\leq q+q_1$, then
 \begin{align}\label{12.2}
 \tau f_i=f_i'+\sum_{j=1}^q
 O\left(\frac{|Z|}{\beta^2}\right)f_j'+\sum_{m=q+1}^{q+q_1}O(
 |Z|)f_m'+O\left(|Z|^2\right).
 \end{align}
\end{lemma}
\begin{proof}
We write
\begin{align}\label{131aa}
 \tau f_i =f_i'+\sum_{k=1}^{q+q_1}\left\langle \tau
 f_i-f_i',f_k'\right\rangle f_k'.
\end{align}

Since
\begin{align}\label{131a}
\left(p+p_1^\perp\right)\nabla^{T\mM,\beta,\varepsilon}_Z(\tau f_i)=0,
\end{align}
  one has for $1\leq i,\, k\leq q$ that
\begin{multline}\label{131b}
\left\langle \tau f_i-f_i',f_k'\right\rangle_{(x,Z)}=Z\left(\left\langle
\tau f_i ,f_k'\right\rangle_{(x,Z)}\right)+O\left(|Z|^2\right)\\
=\left\langle \tau f_i
,\nabla_{Z}^{T\mM,\beta,\varepsilon}f_k'\right\rangle_{(x,Z)}+O\left(|Z|^2\right)
=\left\langle   f_i ,\nabla_Z^{T\mM,\beta,\varepsilon }f_k'\right\rangle_x
+O\left(|Z|^2\right),
\end{multline}
while for $1\leq i\leq q$, $q+1\leq k\leq q+q_1$, one has, by
(\ref{1.5}), (\ref{1.7}),
\begin{multline}\label{131c}
\left\langle \tau f_i-f_i',f_k'\right\rangle_{(x,  Z)}=
Z\left(\left\langle
\tau f_i ,f_k'\right\rangle_{(x,  Z)}\right)+O\left(|Z|^2\right)\\
=\beta^2\varepsilon^{2}\left\langle   f_i
,\nabla_Z^{T\mM,\beta,\varepsilon}f_k'\right\rangle_{x}+O\left(|Z|^2\right)
=O\left( \varepsilon^{2} |Z| \right) +O\left(|Z|^2\right).
\end{multline}

Now by (\ref{3.12a}), one has that for any  $e\in\Gamma(\mF_2^\perp)$ and $1\leq i\leq q$,
\begin{align}\label{3.12b}
  \left[e, f_i'\right]\in\Gamma\left(\mF_2^\perp\right),
\end{align}
 from which one verifies that for any $e\in\Gamma(\mF_2^\perp)$ and $1\leq i,\, k\leq q$,
\begin{align}\label{3.12c}
  \left\langle f_i',\nabla_e^{T\mM,\beta,\varepsilon}f_k'\right\rangle = \left\langle e,\nabla_{ f_i'}^{T\mM,\beta,\varepsilon}f_k'\right\rangle=0.
\end{align}

From  (\ref{131aa}),  (\ref{131b}), (\ref{131c}) and
(\ref{3.12c}), one gets (\ref{12.1}).

By proceeding as in (\ref{131b}), one sees that for $q+1\leq m\leq
q+q_1$, $1\leq k\leq q$,
\begin{multline}\label{11.15}
\left\langle \tau f_m-f_m',f_k'\right\rangle_{(x,Z)}=Z\left(\left\langle
\tau f_m ,f_k'\right\rangle_{(x,Z)}\right)+O\left(|Z|^2\right)\\
=\frac{1}{\beta^2\varepsilon^{2}}\left\langle   f_m
,\nabla_Z^{T\mM,\beta,\varepsilon}f_k'\right\rangle_{x}+O\left(|Z|^2\right)
=O\left(\frac{|Z|}{\beta^2}\right) +O\left(|Z|^2\right),
\end{multline}
while for $q+1\leq m,\,k\leq q+q_1$, one has
\begin{multline}\label{11.16}
\left\langle \tau f_m-f_m',f_k'\right\rangle_{(x,Z)}=Z\left(\left\langle
\tau f_m ,f_k'\right\rangle_{(x,Z)}\right)+O\left(|Z|^2\right)\\
= \left\langle   f_m
,\nabla_Z^{T\mM,\beta,\varepsilon}f_k'\right\rangle_{x}+O\left(|Z|^2\right)
=O\left( {|Z|} \right) +O\left(|Z|^2\right).
\end{multline}

From (\ref{131aa}),  (\ref{11.15}) and (\ref{11.16}), one gets (\ref{12.2}).
\end{proof}

\begin{lemma}\label{t11.4}  There exists   
  $C_{\beta,\varepsilon}>0$   such that the
following estimate holds near $s(M)$ for $|Z|\leq 2\alpha/3$: for any
$\sigma \in\Gamma ( (S(\mF\oplus\mF_1^\perp)  )|_{s(M)} )$, one has
\begin{multline}\label{12.3}
   \sum_{i=1}^{q+q_1}\left|\, ^Q\nabla ^{\mF,
  ,\beta,\varepsilon}_{ \tau
f_i}(\tau \sigma)\right|_{\psi(x,Z)}^2+\sum_{j=1}^{q_2}\left|\,
^Q\nabla ^{\mF, \beta,\varepsilon}_{ \tau e_j}(\tau
\sigma)\right|_{\psi(x,Z)}^2\\
\leq C_{\beta,\varepsilon}\left( \sum_{i=1}^{q+q_1}\left|\,
^Q\nabla ^{\mF,  \beta,\varepsilon}_{
f_i}(\tau \sigma)\right|_x^2+ |\sigma|_x^2\right).
 \end{multline}
  \end{lemma}
\begin{proof} For any $X\in\Gamma (T\mM)|_{s(M)}$ and $\sigma,\ \sigma' \in\Gamma (
(S(\mF\oplus\mF_1^\perp) )|_{s(M)} )$, one verifies that,
\begin{multline}\label{12.3p}
 \left\langle\, ^Q\nabla ^{\mF,
 \beta,\varepsilon}_{\tau X}(\tau
\sigma),\tau\sigma'\right\rangle_{\beta,\varepsilon}=\tau
X\left\langle
\tau\sigma,\tau\sigma'\right\rangle_{\beta,\varepsilon}-\left\langle\tau\sigma,\,^Q\nabla
^{\mF,  \beta,\varepsilon}_{\tau X}(\tau
\sigma')\right\rangle_{\beta,\varepsilon}\\
=\tau X\left\langle
 \sigma, \sigma'\right\rangle_{\beta,\varepsilon}-\left\langle\tau\sigma,\,^Q\nabla
^{\mF,  \beta,\varepsilon}_{\tau X}(\tau
\sigma')\right\rangle_{\beta,\varepsilon}.
 \end{multline}

 From (\ref{12.3p}) and let $\sigma'$ run through the orthonormal
 basis of $(S(\mF\oplus\mF_1^\perp) )|_{s(M)}$, one obtains (\ref{12.3}) easily.
\end{proof}

 We now start to estimate the  terms $I_k$, $1\leq k\leq 6$. 

For any $1\leq i\leq q+q_1$, we denote by $\widetilde{\tau}f_i$ the unit vector field (with respect to $g^{T\mM}_{\beta,\varepsilon}$) corresponding to $\tau f_i$, that is,
\begin{align}\label{12.3f}
\widetilde{\tau}f_i=\frac{\tau f_i}{|\tau f_i|_{\beta,\varepsilon}}.
\end{align}
Then, one has $\widetilde{\tau}f_i=\beta^{-1}\tau f_i$ if $1\leq i\leq q$, while $\widetilde{\tau}f_i=\varepsilon\tau f_i$ if $q+1\leq i\leq q+q_1$.

 Let $1\leq i,\, j\leq q+q_1$ be such that $i\neq j$. By (\ref{3.12}) one deduces
 that
\begin{multline}\label{3.13}
 \left\langle
\left(1-p_{T,\beta,\varepsilon}\right)c_{\beta,\varepsilon}\left(\widetilde{\tau} f_i\right) {\nabla}_{\widetilde{\tau}
f_i}^{\mF, \beta,\varepsilon}J_{T,\beta,\varepsilon} \sigma,
 \left(1-p_{T,\beta,\varepsilon}\right)c_{\beta,\varepsilon}\left(\widetilde{\tau} f_j\right)\ {\nabla}_{\widetilde{\tau} f_j}^{\mF,
\beta,\varepsilon}J_{T,\beta,\varepsilon} \sigma\right\rangle\\
=\left\langle c_{\beta,\varepsilon}\left(\widetilde{\tau} f_i\right)\left(1-p_{T,\beta,\varepsilon}\right)
\widetilde{\tau} f_i(f_{T }) \tau\sigma, c_{\beta,\varepsilon}\left(\widetilde{\tau}
f_j\right)\left(1-p_{T,\beta,\varepsilon}\right) \widetilde{\tau}
f_j(f_{T })\tau \sigma\right\rangle
\\
+\left\langle c_{\beta,\varepsilon}\left(\widetilde{\tau}
f_i\right)\left(1-p_{T,\beta,\varepsilon}\right) \widetilde{\tau}
f_i(f_{T })\tau
\sigma,\left(1-p_{T,\beta,\varepsilon}\right)c_{\beta,\varepsilon}\left(\widetilde{\tau}
f_j\right)f_{T } {\nabla} ^{\mF,
\beta,\varepsilon}_{\widetilde{\tau}
f_j}(\tau \sigma)\right\rangle
\\
+\left\langle\left(1-p_{T,\beta,\varepsilon}\right)c_{\beta,\varepsilon}\left(\widetilde{\tau}
f_i\right)f_{T }{\nabla} ^{\mF,
\beta,\varepsilon}_{\widetilde{\tau}
f_i}(\tau \sigma),c_{\beta,\varepsilon}\left(\widetilde{\tau}
f_j\right)\left(1-p_{T,\beta,\varepsilon}\right) \widetilde{\tau}
f_j(f_{T })\tau \sigma\right\rangle\\
+\left\langle\left(1-p_{T,\beta,\varepsilon}\right)c_{\beta,\varepsilon}\left(\widetilde{\tau}
f_i\right)f_{T } {\nabla} ^{\mF,
\beta,\varepsilon}_{\widetilde{\tau}
f_i}(\tau \sigma), \left(1-p_{T,\beta,\varepsilon}\right)c_{\beta,\varepsilon}\left(\widetilde{\tau}
f_j\right)f_{T }\ {\nabla} ^{\mF,
\beta,\varepsilon}_{\widetilde{\tau}
f_j} (\tau\sigma)\right\rangle .
\end{multline}

By
(\ref{3.10}) and (\ref{11.1}), one has for any $1\leq i\leq q+q_1$,
\begin{align}\label{113}
\left(1-p_{T,\beta,\varepsilon}\right) \tau
f_i\left(f_{T }\right)\tau \sigma = \left(\tau
f_i\left(f_{T }\right)-
f_{T }\int_{\mM_x}f_{T } \tau
f_i\left(f_{T }\right) k\,dv_{\mM_x}\right)\tau\sigma.
 \end{align}

 For any $1\leq i\leq q+q_1$, set
\begin{align}\label{116}
\rho_{T,\beta,\varepsilon,i}  = \tau f_i\left(f_{T }\right) -
f_{T }\int_{\mM_x}f_{T } \tau
f_i\left(f_{T }\right)k \,dv_{\mM_x} .
 \end{align}
 By (\ref{3.9}), one has
\begin{align}\label{119}
\tau f_i\left(f_{T }\right) (x,  Z) =\left(
-\frac{ \tau f_i\left(k\right)\gamma}{2k^{3/2}\sqrt{\alpha _{T } }}  +
\frac{\tau f_i(\gamma)}{k^{1/2}\sqrt{\alpha _{T } } }
  -\frac{ T\tau
f_i\left(|Z|^2\right)\gamma
}{2k^{1/2}\sqrt{\alpha _{T } } }\right)
\exp\left(-\frac{T|Z|^2}{2 }\right) .
 \end{align}

 Let $Z=\sum_{i=1}^{q_2}z_ie_i\in\mF_2^\perp|_{s(M)}$.
 Let $a_{ik}^j\in C^\infty(s(M))$ be defined by
\begin{align}\label{120}
\tau f_i \left(z_j\right)=\left.\tau f_i
\left(z_j\right)\right|_{s(M)}+\sum_{k=1}^{q_2}a_{ij}^kz_k+O\left(|Z|^2\right).
 \end{align}
By (\ref{3.9}), (\ref{116})-(\ref{120}) and Lemma \ref{t11.3},
when $T>0$ is large enough, if $1\leq i\leq q$,
\begin{multline}\label{121}
\rho_{T,\beta,\varepsilon,i}(x,  Z) =
  -\frac{ T\tau
f_i\left(|Z|^2\right)  }{2  } f_{T }(x,
Z)+
\frac{\tau f_i(\gamma)}{k^{1/2}\sqrt{\alpha _{T } } }(1-\gamma)\exp\left(-\frac{T|Z|^2}{2 }\right)
\\
+\frac{1}{2}\left(\sum_{j=1}^{q_2} a_{ij}^j+ O(
|Z|)+O\left(|Z|^2\right)+O \left(\frac{1}{\sqrt{T}}\right)\right)
f_{T }(x,  Z) ,
 \end{multline}
 while for $q+1\leq i\leq q+q_1$, one has
\begin{multline}\label{121e}
\rho_{T,\beta,\varepsilon,i}(x,  Z) =
  -\frac{ T\tau
f_i\left(|Z|^2\right)  }{2  } f_{T }(x,
Z)+
\frac{\tau f_i(\gamma)}{k^{1/2}\sqrt{\alpha _{T } } }(1-\gamma)\exp\left(-\frac{T|Z|^2}{2 }\right)
\\
+\frac{1}{2}\left(\sum_{j=1}^{q_2} a_{ij}^j+ O\left(
\frac{|Z|}{\beta^2}\right)+O\left(|Z|^2\right)+O
\left(\frac{1}{\sqrt{T}}\right)\right)
f_{T }(x,  Z) .
 \end{multline}

We now start to estimate (\ref{3.13}).

For the first term in the right hand side of (\ref{3.13}), by
(\ref{113}) and (\ref{116}), for $i\neq j$,
\begin{multline}\label{114}
{\rm Re}\left(\left\langle c_{\beta,\varepsilon}\left(\widetilde{\tau} f_i\right)\left(1-p_{T,\beta,\varepsilon}\right)
\tau f_i\left(f_{T }\right)\tau \sigma,
c_{\beta,\varepsilon}\left(\widetilde{\tau} f_j\right)\left(1-p_{T,\beta,\varepsilon}\right) \tau f_j\left(f_{T }\right) \tau\sigma\right\rangle\right)
\\
 =
{\rm Re}\left(\left\langle c_{\beta,\varepsilon}\left(\widetilde{\tau} f_i\right)c_{\beta,\varepsilon}\left(\widetilde{\tau} f_j\right) \rho_{T,\beta,\varepsilon,i}\rho_{T,\beta,\varepsilon,j}\tau \sigma,
\tau \sigma\right\rangle\right)=0,
 \end{multline}
 as $c_{\beta,\varepsilon}( \widetilde{\tau} f_i) c_{\beta,\varepsilon} (\widetilde{\tau} f_j ) $ is skew-adjoint.

For   the second and the third terms in the right hand side of (\ref{3.13}),  by (\ref{3.12}),   one finds that for $i\neq j$,
\begin{multline}\label{117}
 \left\langle c_{\beta,\varepsilon}\left(\widetilde{\tau}
f_i\right)\left(1-p_{T,\beta,\varepsilon}\right) \widetilde{\tau}
f_i\left(f_{T }\right)
\tau\sigma,\left(1-p_{T,\beta,\varepsilon}\right)c_{\beta,\varepsilon}\left(\widetilde{\tau}
f_j\right)f_{T } {\nabla} ^{\mF, \beta,\varepsilon}_{\widetilde{\tau}
f_j}(\tau \sigma)\right\rangle   \\
=  \left\langle c_{\beta,\varepsilon}\left(\widetilde{\tau} f_i\right) c_{\beta,\varepsilon}\left(\widetilde{\tau}
f_j\right) \widetilde{\tau}
f_i\left(f_{T }\right) \tau\sigma,\left(1-p_{T,\beta,\varepsilon}\right)
 f_T {\nabla} ^{\mF, \beta,\varepsilon}_{\widetilde{\tau} f_j}(\tau
\sigma)\right\rangle \\
= \left\langle c_{\beta,\varepsilon}\left(\widetilde{\tau} f_i\right) c_{\beta,\varepsilon}\left(\widetilde{\tau}
f_j\right) \widetilde{\tau}
f_i\left(f_{T }\right) f_T\tau\sigma, \,^Q {\nabla} ^{\mF, \beta,\varepsilon}_{\widetilde{\tau} f_j}(\tau
\sigma)-\tau\left(\left.\,^Q {\nabla} ^{\mF, \beta,\varepsilon}_{\widetilde{\tau} f_j}(\tau
\sigma)\right|_{s(M)}\right)
\right\rangle\\
- \left\langle c_{\beta,\varepsilon}\left(\widetilde{\tau} f_i\right) c_{\beta,\varepsilon}\left(\widetilde{\tau}
f_j\right) f_T  p_{T,\beta,\varepsilon}\left( \widetilde{\tau}
f_i\left(f_{T }\right) \tau\sigma\right), \,^Q {\nabla} ^{\mF, \beta,\varepsilon}_{\widetilde{\tau} f_j}(\tau
\sigma)-\tau\left(\left.\,^Q {\nabla} ^{\mF, \beta,\varepsilon}_{\widetilde{\tau} f_j}(\tau
\sigma)\right|_{s(M)}\right)
\right\rangle.
  \end{multline}

Since this term is more delicate to deal with than the other
terms, we postpone it's analysis  to the next subsection.

 For the fourth  term in the right hand side of (\ref{3.13}), one
 first sees easily via   (\ref{11.2}) and (\ref{12.3})    that when $T>0$ is large enough, for any $x\in
 s(M)$,
\begin{multline}\label{126}
\int_{\mM_x}\left\langle\left(1-p_{T,\beta,\varepsilon}\right)c_{\beta,\varepsilon}\left(\widetilde{\tau}
f_i\right)f_{T } {\nabla} ^{\mF, \beta,\varepsilon}_{\widetilde{\tau}  f_i}(\tau \sigma),  \left(1-p_{T,\beta,\varepsilon}\right)c_{\beta,\varepsilon}\left(\widetilde{\tau}
f_j\right)f_{T } {\nabla} ^{\mF, \beta,\varepsilon}_{\widetilde{\tau}
f_j}(\tau \sigma)\right\rangle k\, d{v}_{\mM_x}\\
= \left\langle c_{\beta,\varepsilon}\left(\widetilde{\tau} f_i\right) (1-Q) {\nabla} ^{\mF,
\beta,\varepsilon}_{\widetilde{\tau}
f_i}(\tau \sigma),  c_{\beta,\varepsilon}\left(\widetilde{\tau} f_j\right) (1-Q) {\nabla }^{\mF,
\beta,\varepsilon}_{\widetilde{\tau}
f_j} (\tau\sigma)\right\rangle_x
\\
   +O\left(\frac{1}{\sqrt{T}}\right)
   |\sigma|^2_x   +O\left(\frac{1}{\sqrt{T}}\right)
 \sum_{j=1}^{q+q_1}  \left|
^Q {\nabla} ^{\mF, \beta ,\varepsilon}_{ {f_j} }(\tau\sigma)\right|^2_x  .
\end{multline}

 By definition (cf. (\ref{1.42}) and (\ref{1.42a})), one has on $s(M)$ that
\begin{multline}\label{127}
  (1-Q)\left( {\nabla }^{\mF,
\beta,\varepsilon}_{
f_i}\right)Q
 = \frac{\beta}{2} \sum_{  k=1}^{q }\sum_{j=1}^{ q_2}\left\langle
 \nabla_{f_i}^{T\mM,\beta,\varepsilon}e_j,f_k\right\rangle c_{\beta,\varepsilon}\left(e_j\right)c_{\beta,\varepsilon}\left(\beta^{-1}f_k\right)
 \\
 + \frac{\varepsilon^{-1}}{2} \sum_{  k=q+1}^{q+q_1 }\sum_{j=1}^{ q_2}
 \left\langle \nabla_{f_i}^{T\mM,\beta,\varepsilon}  e_j,  f_k\right\rangle c_{\beta,\varepsilon}\left(e_j\right)c_{\beta,\varepsilon}\left(\varepsilon f_k\right)
  .
 \end{multline}

 By (\ref{3.12c}), one has for $1\leq i,\ k\leq q$ that
\begin{align}\label{127a}
    \left\langle
 \nabla_{f_i}^{T\mM,\beta,\varepsilon}e_j,f_k\right\rangle=0.
 \end{align}

 Also, by (\ref{1.5}) and (\ref{1.7}), one finds that when $1\leq i\leq
 q$,
 $q+1\leq k\leq q+q_1$,
\begin{align}\label{127b}
  \varepsilon^{-1} \left\langle
 \nabla_{f_i}^{T\mM,\beta,\varepsilon}e_j,f_k\right\rangle= O\left(\varepsilon \right).
 \end{align}

 From (\ref{12.3f}) and (\ref{126})-(\ref{127b}), one gets that if $1\leq i,\, j\leq q$ with $i\neq j$,
 then
\begin{multline}\label{127c}
\int_{\mM_x}
\left\langle\left(1-p_{T,\beta,\varepsilon}\right)c_{\beta,\varepsilon}\left(\widetilde{\tau}
f_i\right)f_{T } {\nabla} ^{\mF,
\beta,\varepsilon}_{\widetilde{\tau}
f_i}(\tau \sigma),
\left(1-p_{T,\beta,\varepsilon}\right)c_{\beta,\varepsilon}\left(\widetilde{\tau}
f_j\right)f_{T } {\nabla} ^{\mF,
\beta,\varepsilon}_{\widetilde{\tau}
f_j}(\tau \sigma)\right\rangle k\, dv_{\mM_x}
\\
=\left( O\left(  {\frac{\varepsilon^2}{\beta^2}} \right) +O\left(\frac{1}{\sqrt{T}}\right)\right)
   |\sigma|^2_x   +O\left(\frac{1}{\sqrt{T}}\right)
 \sum_{j=1}^{q+q_1}  \left|
^Q {\nabla} ^{\mF, \beta,\varepsilon}_{ {f_j} }(\tau\sigma)\right|^2_x .
 \end{multline}

 If $q+1\leq i \leq q+q_1$, $1\leq k\leq q$,   then one has by (\ref{f5})  that 
 \begin{align}\label{127d}
  \beta \left\langle
 \nabla_{f_i}^{T\mM,\beta,\varepsilon}e_j,f_k\right\rangle=O\left(\frac{1}{\beta}\right),
 \end{align}
while if $q+1\leq i,\, k\leq q+q_1$, one has
 \begin{align}\label{127e}
  \varepsilon^{-1} \left\langle
 \nabla_{f_i}^{T\mM,\beta,\varepsilon}e_j,f_k\right\rangle= O\left(\varepsilon^{-1}\right).
 \end{align}
 Combining with (\ref{126})-(\ref{127b}), one gets that if $q+1\leq i\leq
 q+q_1$, $1\leq j\leq q$,
 then
\begin{multline}\label{127f}
\int_{\mM_x}\left\langle\left(1-p_{T,\beta,\varepsilon}\right)c_{\beta,\varepsilon}\left(\widetilde{\tau}
f_i\right)f_{T } {\nabla} ^{\mF, \beta,\varepsilon}_{\widetilde{\tau}  f_i}(\tau \sigma), \left(1-p_{T,\beta,\varepsilon}\right)c_{\beta,\varepsilon}\left(\widetilde{\tau}
f_j\right)f_{T } {\nabla} ^{\mF, \beta,\varepsilon}_{\widetilde{\tau}
f_j}(\tau \sigma)\right\rangle k\, d{v}_{\mM_x}\\
=\left(
O\left(\frac{\varepsilon(\beta+\varepsilon)}{\beta^2}\right)+O\left(\frac{1}{\sqrt{T}}\right)\right)
   |\sigma|^2_x   +O\left(\frac{1}{\sqrt{T}}\right)
 \sum_{j=1}^{q+q_1}  \left|
^Q {\nabla} ^{\mF, \beta,\varepsilon}_{ {f_j} }(\tau\sigma)\right|^2_x .
 \end{multline}
Also, when $q+1\leq i,\, j\leq
 q+q_1$  with $i\neq j$, one gets
\begin{multline}\label{127g}
\int_{\mM_x}\left\langle\left(1-p_{T,\beta,\varepsilon}\right)c_{\beta,\varepsilon}\left(\widetilde{\tau}
f_i\right)f_{T } {\nabla }^{\mF, \beta,\varepsilon}_{\widetilde{\tau}  f_i} (\tau\sigma),  \left(1-p_{T,\beta,\varepsilon}\right)c_{\beta,\varepsilon}\left(\widetilde{\tau}
f_j\right)f_{T } {\nabla} ^{\mF, \beta,\varepsilon}_{\widetilde{\tau}
f_j} (\tau\sigma)\right\rangle k\, d{v}_{\mM_x}\\
= \left(O\left( \frac{(\beta+\varepsilon)^2}{\beta^2}\right) +O\left(\frac{1}{\sqrt{T}}\right)\right)
   |\sigma|^2_x   +O\left(\frac{1}{\sqrt{T}}\right)
 \sum_{j=1}^{q+q_1}  \left|
^Q {\nabla} ^{\mF, \beta,\varepsilon}_{ {f_j} }(\tau\sigma)\right|^2_x .
 \end{multline}

 Now we consider the terms $I_5$ and $I_6$.   By (\ref{i5}) and (\ref{i6}), we need to consider the following term for $1\leq j\leq q+q_1$ and $1\leq k\leq q_2$:
\begin{multline}\label{128}
 \left\langle
\left(1-p_{T,\beta,\varepsilon}\right)c_{\beta,\varepsilon}\left(\widetilde{\tau}  f_i\right) {\nabla}_{\widetilde{\tau}
f_i}^{\mF, \beta,\varepsilon}J_{T,\beta,\varepsilon}  \sigma,
\left(1-p_{T,\beta,\varepsilon}\right)c_{\beta,\varepsilon}\left(\tau e_k\right) {\nabla}_{\tau
e_k}^{\mF, \beta,\varepsilon}J_{T,\beta,\varepsilon} \sigma\right\rangle\\
=\left\langle c_{\beta,\varepsilon}\left(\widetilde{\tau}  f_i\right)\left(1-p_{T,\beta,\varepsilon}\right)
\widetilde{\tau}  f_i\left(f_{T }\right)\tau \sigma,
 c_{\beta,\varepsilon}\left(\tau e_k\right) \tau
e_k\left(f_{T }\right)\tau \sigma\right\rangle
\\
+\left\langle\left(1-p_{T,\beta,\varepsilon}\right)c_{\beta,\varepsilon}\left(\widetilde{\tau}
f_i\right)f_{T }\nabla ^{\mF, \beta,\varepsilon}_{\widetilde{\tau}  f_i}{\tau \sigma},
 c_{\beta,\varepsilon}\left(\tau e_k\right)f_{T } {\nabla}
^{\mF, \beta,\varepsilon}_{\tau e_k}(\tau
\sigma)\right\rangle
\\
+\left\langle\left(1-p_{T,\beta,\varepsilon}\right)c_{\beta,\varepsilon}\left(\widetilde{\tau}
f_i\right)f_{T } {\nabla} ^{\mF, \beta,\varepsilon}_{\widetilde{\tau}  f_i}(\tau \sigma),
c_{\beta,\varepsilon}\left(\tau e_k\right) \tau e_k\left(f_{T }\right)
\tau\sigma\right\rangle
\\
+\left\langle c_{\beta,\varepsilon}\left(\widetilde{\tau}  f_i\right)\left(1-p_{T,\beta,\varepsilon}\right)
\widetilde{\tau} f_i\left(f_{T }\right)\tau \sigma, c_{\beta,\varepsilon}\left(\tau
e_k\right)f_{T } {\nabla} ^{\mF, \beta,\varepsilon}_{\tau e_k}(\tau
\sigma)\right\rangle.
\end{multline}

First, by (\ref{113}) and the obvious parity consideration,\footnote{By the``parity consideration" here we mean that if a term $A$  involves an odd   number of Clifford actions $c(U)$ with $U\in \mF_2^\perp$, then one has the obvious fact that $QAQ=0$, etc. The ``degree consideration" appears in the later text is based on the same reasoning.} we
have
\begin{align}\label{129}
 \left\langle
  c_{\beta,\varepsilon}\left(\widetilde{\tau}  f_i\right)\left(1-p_{T,\beta,\varepsilon}\right)
\widetilde{\tau}  f_i\left(f_{T }\right)\tau \sigma,
 c_{\beta,\varepsilon}\left(\tau e_k\right) \tau
e_k\left(f_{T }\right)\tau \sigma\right\rangle
 =0.
\end{align}

\begin{lemma}\label{t11.5} For any $U\in\Gamma(\mF_2^\perp|_{s(M)})$, the following identity holds on
$s(M)$,
\begin{align}\label{12.24}
\left.\left(^Q {\nabla}^{\mF, \beta,\varepsilon}_U(\tau\sigma)\right)\right|_{s(M)}=0.
\end{align}
\end{lemma}
\begin{proof} By construction, one has
\begin{align}\label{12.25}
 ^Q\nabla^{\mF,  \beta,\varepsilon}_Z(\tau\sigma)  =0.
\end{align}

 Taking the derivative with respect to $z_i$, one gets
\begin{align}\label{12.26}
\left.\left(^Q\nabla^{\mF, \beta,\varepsilon}_{e_i}(\tau\sigma)\right)\right|_{s(M)}=0.
\end{align}

Formula (\ref{12.24}) follows from (\ref{12.26}).
\end{proof}

For the second term in the right hand side of (\ref{128}),   one obtains by
(\ref{11.2}), (\ref{12.3}) and  Lemma \ref{t11.5} that  for any $x\in s(M)$, one has
\begin{multline}\label{130}
\int_{\mM_x} \left\langle\left(1-p_{T,\beta,\varepsilon}\right)c_{\beta,\varepsilon}\left(\widetilde{\tau}
f_i\right)f_{T } {\nabla }^{\mF, \beta,\varepsilon}_{\widetilde{\tau}  f_i} (\tau\sigma),
 c_{\beta,\varepsilon}\left(\tau e_k\right)f_{T } {\nabla}
^{\mF, \beta,\varepsilon}_{\tau e_k}
(\tau\sigma)\right\rangle_{(x,Z)} k\,dv_{\mM_x} \\
=
 \left\langle c_{\beta,\varepsilon}\left( \widetilde{\tau} f_i\right)(1-Q)  {\nabla}
^{\mF, \beta,\varepsilon}_{\widetilde{\tau}
f_i} (\tau\sigma),
 c_{\beta,\varepsilon}\left(  e_k\right) (1-Q)  {\nabla}
^{\mF, \beta,\varepsilon}_{
e_k}(\tau  \sigma )\right\rangle_x\\
+O\left( \frac{1}{\sqrt{T}}\right)
   |\sigma|^2_x   +O\left(\frac{1}{\sqrt{T}}\right)
 \sum_{j=1}^{q+q_1}  \left|
^Q {\nabla} ^{\mF, \beta,\varepsilon}_{ {f_j}
}(\tau\sigma)\right|^2_x.\end{multline}

By (\ref{1.6}) and (\ref{2.2}), one knows that for any $U,\,
V\in\Gamma(\mF_2^\perp)$ and $X\in\Gamma(\mF)$, one has
\begin{align}\label{12.27a}
 \left\langle \nabla^{T\mM,\beta,\varepsilon}_UV,X\right\rangle=0.
\end{align}

Similar to (\ref{127}), one has by (\ref{12.27a}) that, on $s(M)$,
\begin{multline}\label{12.30}
  (1-Q)\left( {\nabla} ^{\mF,
\beta,\varepsilon}_{
e_k}\right)Q
 = \frac{\beta}{2} \sum_{  s=1}^{q }\sum_{j=1}^{ q_2}\left\langle
 \nabla_{e_k}^{T\mM,\beta,\varepsilon}e_j,f_s\right\rangle c_{\beta,\varepsilon}\left(e_j\right)c_{\beta,\varepsilon}\left(\beta^{-1}f_s\right)
 \\
 + \frac{\varepsilon^{-1}}{2} \sum_{  s=q+1}^{q+q_1 }\sum_{j=1}^{ q_2}
 \left\langle \nabla_{e_k}^{T\mM,\beta,\varepsilon}  e_j,  f_s\right\rangle c_{\beta,\varepsilon}\left(e_j\right)c_{\beta,\varepsilon}\left(\varepsilon f_s\right)
    \\
 =\frac{\varepsilon^{-1}}{2} \sum_{  s=q+1}^{q+q_1 }\sum_{j=1}^{ q_2}
 \left\langle \nabla_{e_k}^{T\mM,\beta,\varepsilon}  e_j,  f_s\right\rangle c_{\beta,\varepsilon}\left(e_j\right)c_{\beta,\varepsilon}\left(\varepsilon f_s\right)
   .
 \end{multline}

From (\ref{127}),  (\ref{130}), (\ref{12.30}) and  the easy parity
consideration,  one gets that for $ 1\leq i\leq q+q_1$, $1\leq
k\leq q_2$,
\begin{multline}\label{130t}
\int_{\mM_x}
\left\langle\left(1-p_{T,\beta,\varepsilon}\right)c_{\beta,\varepsilon}\left(\widetilde{\tau}
f_i\right)f_{T } {\nabla} ^{\mF,
\beta,\varepsilon}_{\widetilde{\tau}
f_i} (\tau\sigma),
 c_{\beta,\varepsilon}\left(\tau e_k\right)f_{T } {\nabla}
^{\mF, \beta,\varepsilon}_{\tau e_k}
(\tau\sigma)\right\rangle_{(x,Z)} k\,dv_{\mM_x} \\
=O\left(   \frac{1}{\sqrt{T}}\right)
   |\sigma|^2_x   +O\left(\frac{1}{\sqrt{T}}\right)
 \sum_{j=1}^{q+q_1}  \left|
^Q\nabla ^{\mF,  \beta,\varepsilon}_{ {f_j} }(\tau\sigma)\right|^2_x.
\end{multline}

For the third term in the right hand side of (\ref{128}), if $1\leq i\leq q+q_1$,  one has by  an easy degree consideration,
\begin{multline}\label{131}
  \left\langle\left(1-p_{T,\beta,\varepsilon}\right)c_{\beta,\varepsilon}\left(\widetilde{\tau}  f_i \right)f_{T }{\nabla}
^{\mF, \beta,\varepsilon}_{\widetilde{\tau}  f_i }(\tau \sigma),
c_{\beta,\varepsilon}\left(\tau e_k\right) \tau e_k\left(f_{T }\right)
\tau\sigma\right\rangle\\
 =\left\langle c_{\beta,\varepsilon}\left(\widetilde{\tau}  f_i \right)f_{T } {\nabla}
^{\mF, \beta,\varepsilon}_{\widetilde{\tau}  f_i }{(\tau \sigma)},
c_{\beta,\varepsilon}\left(\tau e_k\right) \tau e_k\left(f_{T }\right)
\tau\sigma\right\rangle
\\
=\left\langle c_{\beta,\varepsilon}\left(\widetilde{\tau}  f_i \right)f_{T }(1-Q) {\nabla}
^{\mF, \beta,\varepsilon}_{\widetilde{\tau}  f_i }{(\tau \sigma)},
c_{\beta,\varepsilon}\left(\tau e_k\right) \tau e_k\left(f_{T }\right)
\tau\sigma\right\rangle .
\end{multline}

As in (\ref{127}), one has
\begin{multline}\label{11.7}
  (1-Q)\left( {\nabla} ^{\mF,
\beta,\varepsilon}_{\tau
f_i}\right)Q
 = \frac{1}{2\beta} \sum_{  k=1}^{q }\sum_{j=1}^{ q_2}\left\langle
 \nabla_{\tau f_i}^{T\mM,\beta,\varepsilon}(\tau e_j),\tau f_k\right\rangle_{\beta,\varepsilon}
 c_{\beta,\varepsilon}\left(\tau e_j\right)c_{\beta,\varepsilon}\left(\beta^{-1}\tau f_k\right)
 \\
 + \frac{\varepsilon }{2} \sum_{  k=q+1}^{q+q_1 }\sum_{j=1}^{ q_2}
 \left\langle \nabla_{\tau f_i}^{T\mM,\beta,\varepsilon} (\tau e_j), \tau f_k\right\rangle_{\beta,\varepsilon}
  c_{\beta,\varepsilon}\left(\tau e_j\right)c_{\beta,\varepsilon}\left(\varepsilon\tau
 f_k\right),
 \end{multline}
where the subscripts ``$\beta$'', ``$\varepsilon$'' are to emphasize that the pointwise
inner product is the one with respect to $g^{T\mM}_{\beta,\varepsilon}$.

From (\ref{11.7}), one finds
\begin{multline}\label{11.8}
\left\langle c_{\beta,\varepsilon}\left(\tau f_i\right)f_{T }(1-Q) {\nabla}
^{\mF, \beta,\varepsilon}_{\tau f_i}{(\tau \sigma)},
c_{\beta,\varepsilon}\left(\tau e_k\right) \tau e_k\left(f_{T }\right)
\tau\sigma\right\rangle
\\
=\frac{1}{2\beta} \sum_{  m=1}^{q }\sum_{j=1}^{ q_2}\left(
\int_{s(M)}\left\langle c_{\beta,\varepsilon}\left(  f_i\right)c_{\beta,\varepsilon}\left(
e_j\right)c_{\beta,\varepsilon}\left(\beta^{-1} f_m\right)\sigma,c_{\beta,\varepsilon}\left(  e_k\right)
\sigma\right\rangle dv_{s(M)}\right.
 \end{multline}
$$
\left.\cdot\int_{\mM_x}\left\langle { \nabla}_{\tau
f_i}^{T\mM,\beta,\varepsilon}(\tau e_j),\tau
f_m\right\rangle_{\beta,\varepsilon} f_{T }\, \tau
e_k\left(f_{T }\right)k\, dv_{\mM_x}(Z)\right)
$$
$$
+\frac{\varepsilon }{2} \sum_{  m=q+1}^{q+q_1
}\sum_{j=1}^{ q_2}\left(\int_{s(M)}\left\langle c_{\beta,\varepsilon}\left(
f_i\right)c_{\beta,\varepsilon}\left( e_j\right)c_{\beta,\varepsilon}\left(\varepsilon  f_m\right)\sigma,c_{\beta,\varepsilon}\left(
e_k\right) \sigma\right\rangle dv_{s(M)}\right.
$$
$$
\left.\cdot \int_{\mM_x}\left\langle  \nabla_{\tau
f_i}^{T\mM,\beta,\varepsilon}(\tau e_j),\tau
f_m\right\rangle_{\beta,\varepsilon} f_{T }\, \tau
e_k\left(f_{T }\right)k\, dv_{\mM_x}(Z) \right)
$$
$$
=-\frac{1}{2\beta} \sum_{  m=1}^{q } \int_{s(M)}\left\langle
c_{\beta,\varepsilon}\left(  f_i\right)c_{\beta,\varepsilon} \left(\beta^{-1} f_m\right)\sigma, \sigma\right\rangle
dv_{s(M)} 
$$
$$
\cdot\int_{\mM_x}\left\langle \nabla_{\tau
f_i}^{T\mM,\beta,\varepsilon}(\tau e_k),\tau
f_m\right\rangle_{\beta,\varepsilon} f_{T }\, \tau
e_k\left(f_{T }\right)k\, dv_{\mM_x}(Z)
$$
$$
-\frac{\varepsilon }{2} \sum_{  m=q+1}^{q+q_1
} \int_{s(M)}\left\langle c_{\beta,\varepsilon}\left( f_i\right) c_{\beta,\varepsilon}\left(\varepsilon
f_m\right)\sigma, \sigma\right\rangle
dv_{s(M)}
$$
$$
\cdot\int_{\mM_x}\left\langle \nabla_{\tau
f_i}^{T\mM,\beta,\varepsilon}(\tau e_k),\tau
f_m\right\rangle_{\beta,\varepsilon} f_{T }\, \tau
e_k\left(f_{T }\right)k\, dv_{\mM_x}(Z).
$$

Clearly,  when $i\neq m$, $c_{\beta,\varepsilon}(f_i)c_{\beta,\varepsilon}(f_m)$ is skew-adjoint, thus
\begin{align}\label{11.9a}
{\rm Re}\left( \left\langle c_{\beta,\varepsilon}\left( f_i\right) c_{\beta,\varepsilon}\left(
f_m\right)\sigma, \sigma\right\rangle\right) =0.
 \end{align}

By (\ref{3.9}), one has
\begin{align}\label{11.9}
\tau e_k\left(f_{T }\right) (x,  Z) =\left(
-\frac{ \tau e_k\left(k\right)\gamma}{2k^{3/2}\sqrt{\alpha_{T }} }  +
\frac{\tau e_k(\gamma)}{k^{1/2}\sqrt{\alpha_{T }} }
  -\frac{ T\tau
e_k\left(|Z|^2\right)\gamma
}{2k^{1/2}\sqrt{\alpha_{T }} }\right)
\exp\left(-\frac{T|Z|^2}{2 }\right) .
 \end{align}

 By (\ref{2.3}), one knows that $\tau e_k$ does not depend on $\beta$ and
$\varepsilon$.

From   Lemma \ref{t11.3} and (\ref{127a}), one gets that for $1\leq i,\, m\leq q$,
$1\leq j\leq q_2$,
\begin{multline}\label{11.11c}
 \left.\left\langle \nabla_{\tau f_i}^{T\mM,\beta,\varepsilon}(\tau e_j),\tau
f_m\right\rangle_{\beta,\varepsilon}\right|_{(x,  Z)} =
   \left\langle
\nabla_{f_i'+  \sum_{k=q+1}^{q+q_1}O(\varepsilon^{2}
|Z|)f_k'}^{T\mM,\beta,\varepsilon}(\tau e_j),  f_m' +\sum_{k=q+1}^{q+q_1}O\left(\varepsilon^{2}
|Z|\right)f_k'\right\rangle_{\beta,\varepsilon} \\ +O\left(|Z|^2\right)
=  O\left(\varepsilon^2{|Z|} \right)+O\left(|Z|^2\right).
 \end{multline}

From (\ref{11.9}) and (\ref{11.11c}), one gets
\begin{align}\label{11.12}
\frac{1}{\beta} \int_{\mM_x}\left\langle \nabla_{\tau
f_i}^{T\mM,\beta,\varepsilon}(\tau e_j),\tau
f_m\right\rangle_{\beta,\varepsilon} f_{T }\, \tau
e_k\left(f_{T }\right)k\,
dv_{\mM_x}(Z)=O\left( \frac{\varepsilon^2}{\beta}\right)+O\left(\frac{1}{\sqrt{T}}\right).
\end{align}

 From  (\ref{12.3f}), (\ref{131}), (\ref{11.8}), (\ref{11.9a}) and (\ref{11.12}), one finds that when $1\leq i\leq
 q$, $1\leq k\leq q_2$,
\begin{multline}\label{132}
 \left\langle\left(1-p_{T,\beta,\varepsilon}\right)c_{\beta,\varepsilon}\left(\widetilde{\tau}
f_i\right)f_{T } {\nabla} ^{\mF,
\beta, \varepsilon}_{\widetilde{\tau}
f_i}  (\tau \sigma),  c_{\beta,\varepsilon}\left(\tau e_k\right) \tau
e_k\left(f_{T }\right)
\tau\sigma\right\rangle \\
 =\left( O\left( \frac{\varepsilon^2}{\beta^2} \right)+O\left(\frac{1}{\sqrt{T}}\right)\right)
\int_{s(M)}|\sigma|^2dv_{s(M)}.
\end{multline}

Now for $q+1\leq i,\, m\leq q+q_1$ and $1\leq j\leq q_2$, one has
\begin{multline}\label{11.18b}
 \left.\left\langle \nabla_{\tau f_i}^{T\mM,\beta,\varepsilon}(\tau e_j),\tau
f_m\right\rangle_{\beta,\varepsilon}\right|_{(x,Z)}  =
   \left\langle
\nabla_{f_i'+\sum_{j=1}^q O\left(\frac{|Z|}{\beta^2}\right)
f_j'+\sum_{k=q+1}^{q+q_1}O( |Z|)f_k'}^{T\mM,\beta,\varepsilon}(\tau e_j),\right.\\
\left.f_m'+\sum_{j=1}^q O\left(\frac{|Z|}{\beta^2}\right)
f_j'+\sum_{k=q+1}^{q+q_1}O( |Z|)f_k' \right\rangle_{\beta,\varepsilon}  +O\left(|Z|^2\right)\\
=
O\left(\frac{1}{\varepsilon^{2}}\right)+O\left(\left(\frac{1}{\beta^2}+\frac{1}{\varepsilon^2} \right)|Z|\right)+O\left(|Z|^2\right).
 \end{multline}

By using (\ref{12.3f}),  (\ref{131}), (\ref{11.8})-(\ref{11.9}) and
(\ref{11.18b}), one finds that when $q+1\leq i\leq q+q_1$, $1\leq
k\leq q_2$,
 \begin{multline}\label{132a}
\left\langle\left(1-p_{T,\beta,\varepsilon}\right)c_{\beta,\varepsilon}\left(\widetilde{\tau}
f_i\right)f_{T } {\nabla }^{\mF,
\beta,\varepsilon}_{\widetilde{\tau}
f_i}{\tau \sigma}, c_{\beta,\varepsilon}\left(\tau e_k\right) \tau
e_k\left(f_{T }\right)
\tau\sigma\right\rangle 
\\
 =
 \left(O\left(1+\frac{\varepsilon^2}{\beta^2}\right)+O\left(\frac{1}{\sqrt{T}}\right)\right)
\int_{s(M)}|\sigma|^2dv_{s(M)}.
\end{multline}

For the fourth term in the right hand side of (\ref{128}), one verifies easily by (\ref{113}) and (\ref{116})  that
\begin{multline}\label{11.19a}
\left\langle c_{\beta,\varepsilon}\left( {\tau}  f_i\right)\left(1-p_{T,\beta,\varepsilon}\right)
 {\tau}  f_i\left(f_{T }\right)\tau \sigma, c_{\beta,\varepsilon}\left(\tau
e_k\right)f_{T } {\nabla} ^{\mF, \beta,\varepsilon}_{\tau e_k}(\tau
\sigma)\right\rangle\\
= \left\langle c_{\beta,\varepsilon}\left( {\tau}
f_i\right)\left(1-p_{T,\beta,\varepsilon}\right)  {\tau}
f_i\left(f_{T }\right)\tau \sigma, c_{\beta,\varepsilon}\left(\tau
e_k\right)f_{T }(1-Q) {\nabla} ^{\mF,
\beta,\varepsilon}_{\tau
e_k}(\tau \sigma)\right\rangle \\
=\left\langle c_{\beta,\varepsilon}\left( {\tau}  f_i\right)\rho_{T,\beta,\varepsilon,i}\tau
\sigma, c_{\beta,\varepsilon}\left(\tau e_k\right)f_{T }(1-Q) {\nabla} ^{\mF,
\beta,\varepsilon}_{\tau
e_k}(\tau \sigma)\right\rangle .
\end{multline}

As in (\ref{11.7}), one has
\begin{multline}\label{134a}
(1-Q)  { \nabla} ^{\mF, \beta,\varepsilon}_{\tau e_k} (\tau\sigma )=
\frac{1}{2\beta}\sum_{j=1}^{q_2}\sum_{m=1}^q\left\langle
\nabla^{T\mM,\beta,\varepsilon}_{\tau e_k}(\tau e_j),\tau
f_m\right\rangle_{\beta,\varepsilon}
c_{\beta,\varepsilon}(\tau e_j)c_{\beta,\varepsilon}\left(\beta^{-1}\tau f_m\right)\tau\sigma\\
+\frac{\varepsilon }{2}\sum_{j=1}^{q_2}\sum_{m=q+1}^{q+q_1}\left\langle
\nabla^{T\mM,\beta,\varepsilon}_{\tau e_k}(\tau e_j),\tau
f_m\right\rangle_{\beta,\varepsilon} c_{\beta,\varepsilon}(\tau e_j)c_{\beta,\varepsilon}(\varepsilon\tau f_m)\tau\sigma .
\end{multline}

By Lemma \ref{t11.3},  (\ref{2.2}) and (\ref{12.27a}), one
verifies that for $1\leq m\leq q$, one has
\begin{multline}\label{11.20s}
 \left.\left\langle \nabla_{\tau e_i}^{T\mM,\beta,\varepsilon}(\tau e_j),\tau
f_m\right\rangle_{\beta,\varepsilon}\right|_{(x,Z)}  =
   \left\langle
\nabla^{T\mM,\beta,\varepsilon}_{\tau e_i} \tau e_j, f_m'+ \sum_{k=q+1}^{q+q_1}O\left(\varepsilon^{2} |Z|\right)f_k' \right\rangle_{\beta,\varepsilon } +O\left(|Z|^2\right)\\
=
 O\left({\varepsilon^{2}} {|Z|}\right)+O\left(|Z|^2\right),
 \end{multline}
while for $q+1\leq m\leq q+q_1$, one has,
\begin{multline}\label{11.20}
 \left.\left\langle \nabla_{\tau e_i}^{T\mM,\beta,\varepsilon}(\tau e_j),\tau
f_m\right\rangle_{\beta,\varepsilon}\right|_{(x,Z)}  =
   \left\langle
\nabla^{T\mM,\beta,\varepsilon}_{\tau e_i} \tau e_j, f_m'+\sum_{j=1}^q
O\left(\frac{|Z|}{\beta^2}\right)
f_j'+\sum_{k=q+1}^{q+q_1}O( |Z|)f_k' \right\rangle_{\beta,\varepsilon }
\\
 +O\left(|Z|^2\right)
=
O\left(1\right)+O\left( {|Z|}{ }\right)+O\left(|Z|^2\right).
 \end{multline}

From Lemma \ref{t11.3},  (\ref{121}), (\ref{121e}) and (\ref{11.19a})-(\ref{11.20}),
one gets that for $1\leq i\leq q $ and $1\leq k\leq q_2$,   and also using the parity consideration,
\begin{multline}\label{11.19d}
 \frac{1}{\beta}\left\langle c_{\beta,\varepsilon}\left(\beta^{-1}\tau
f_i\right)\left(1-p_{T,\beta,\varepsilon}\right) \tau
f_i\left(f_{T }\right)\tau \sigma, c_{\beta,\varepsilon}\left(\tau
e_k\right)f_{T } {\nabla} ^{\mF,
\beta,\varepsilon}_{\tau
e_k}(\tau \sigma)\right\rangle 
 \\
=\left(O\left( \frac{\varepsilon^2}{\beta^2}\right)+O\left(\frac{1}{\sqrt{T}}\right)\right)\int_{s(M)}|\sigma|^2dv_{s(M)},
\end{multline}
while for $q+1\leq i\leq q+q_1 $ and $1\leq k\leq q_2$,   one has
\begin{multline}\label{11.19f}
 \varepsilon\left\langle c_{\beta,\varepsilon}\left(\varepsilon\tau
f_i\right)\left(1-p_{T,\beta,\varepsilon}\right) \tau
f_i\left(f_{T }\right)\tau \sigma, c_{\beta,\varepsilon}\left(\tau
e_k\right)f_{T } {\nabla} ^{\mF,
\beta,\varepsilon}_{\tau
e_k}(\tau \sigma)\right\rangle 
 \\
=\left(O\left(  \frac{\varepsilon ^2}{\beta^2} \right)+O\left( \frac{1}{\sqrt{T}}\right)\right)\int_{s(M)}|\sigma|^2dv_{s(M)}.
\end{multline}

Now we consider the term for $1\leq i,\, k\leq q_2$ with $i\neq
k$,
\begin{multline}\label{136}
 \left\langle
\left(1-p_{T,\beta,\varepsilon}\right)c_{\beta,\varepsilon}\left(\tau e_i\right) {\nabla}_{\tau
e_i}^{\mF,\beta,\varepsilon}J_{T,\beta,\varepsilon} \sigma,
 c_{\beta,\varepsilon}\left(\tau e_k\right) {\nabla}_{\tau
e_k}^{\mF, \beta,\varepsilon}J_{T,\beta,\varepsilon} \sigma\right\rangle\\
=\left\langle\left(1-p_{T,\beta,\varepsilon}\right) c_{\beta,\varepsilon}\left(\tau e_i\right)
\tau e_i\left(f_{T }\right)\tau \sigma,
 c_{\beta,\varepsilon}\left(\tau e_k\right) \tau
e_k\left(f_{T }\right) \tau\sigma\right\rangle
\\
+\left\langle\left(1-p_{T,\beta,\varepsilon}\right)c_{\beta,\varepsilon}\left(\tau
e_i\right)f_{T } {\nabla} ^{\mF, \beta,\varepsilon}_{\tau e_i} (\tau \sigma),
 c_{\beta,\varepsilon}\left(\tau e_k\right)f_{T } {\nabla}
^{\mF, \beta,\varepsilon}_{\tau e_k}(\tau
\sigma)\right\rangle
\\
+\left\langle\left(1-p_{T,\beta,\varepsilon}\right)c_{\beta,\varepsilon}\left(\tau
e_i\right)f_{T } {\nabla }^{\mF, \beta,\varepsilon}_{\tau e_i}(\tau \sigma),
c_{\beta,\varepsilon}\left(\tau e_k\right) \tau e_k\left(f_{T }\right)
\tau\sigma\right\rangle
\\
+\left\langle \left(1-p_{T,\beta,\varepsilon}\right)c_{\beta,\varepsilon}\left(\tau e_i\right)
\tau e_i\left(f_{T }\right)\tau \sigma,
c_{\beta,\varepsilon}\left(e_k\right)f_{T } {\nabla }^{\mF,
\beta,\varepsilon}_{\tau
e_k} (\tau\sigma)\right\rangle.
\end{multline}

For the first term in the right hand side of (\ref{136}), one has,
as $i\neq k$,
\begin{multline}\label{137}
\left\langle \left(1-p_{T,\beta,\varepsilon}\right)c_{\beta,\varepsilon}\left(\tau e_i\right)
\tau e_i\left(f_{T }\right)\tau \sigma,
 c_{\beta,\varepsilon}\left(\tau e_k\right) \tau
e_k\left(f_{T }\right)\tau \sigma\right\rangle\\
=-\left\langle  \tau e_k\left(f_{T }\right)
   \tau e_i\left(f_{T }\right)\tau \sigma,
 c_{\beta,\varepsilon}\left(\tau e_i\right)c_{\beta,\varepsilon}\left(\tau e_k\right) \tau \sigma\right\rangle=0.
\end{multline}

For the second term in the right hand side of (\ref{136}), one has
by   (\ref{11.2}) and Lemma \ref{t11.5} that for any $x\in s(M)$,
\begin{multline}\label{138}
\int_{\mM_x}\left\langle\left(1-p_{T,\beta,\varepsilon}\right)c_{\beta,\varepsilon}\left(\tau
e_i\right)f_{T }{\nabla} ^{\mF, \beta,\varepsilon}_{\tau e_i} (\tau \sigma),
 c_{\beta,\varepsilon}\left(\tau e_k\right)f_{T } {\nabla}
^{\mF,\beta,\varepsilon}_{\tau e_k}(\tau \sigma)\right\rangle_{(x,Z)}k\,dv_{\mM_x}\\
= \int_{\mM_x}f_{T }^2
\left\langle\left(1-Q\right)c_{\beta,\varepsilon}\left(\tau
e_i\right)(1-Q) {\nabla }^{\mF,
\beta,\varepsilon}_{\tau
e_i} (\tau \sigma),
 c_{\beta,\varepsilon}\left(\tau e_k\right) (1-Q) {\nabla}
^{\mF, \beta,\varepsilon}_{\tau e_k}(\tau
\sigma)\right\rangle_{(x,Z)}k\,dv_{\mM_x}\end{multline}
$$
+O\left(\frac{1}{\sqrt{T}}\right) |\sigma|_x^2
+O\left(\frac{1}{\sqrt{T}}\right) \sum_{i=1}^{q+q_1}\left|\,
^Q\nabla ^{\mF,  \beta,\varepsilon}_{
f_i}(\tau \sigma)\right|_x^2 $$
$$
=
 \left\langle\left(1-Q\right)c_{\beta,\varepsilon}\left(e_i\right)(1-Q) { \nabla} ^{\mF,
\beta,\varepsilon}_{ e_i}
(\tau\sigma),
 c_{\beta,\varepsilon}\left(e_k\right) (1-Q) {\nabla}
^{\mF, \beta,\varepsilon}_{ e_k}
(\tau\sigma)\right\rangle_x$$
$$
+O\left(\frac{1}{\sqrt{T}}\right) |\sigma|_x^2
+O\left(\frac{1}{\sqrt{T}}\right) \sum_{i=1}^{q+q_1}\left|\,
^Q\nabla ^{\mF,  \beta,\varepsilon}_{
f_i}(\tau \sigma)\right|_x^2.$$

Now, one has by (\ref{12.30}) that for any $1\leq i\leq q_2$, at
$x\in s(M)$,
\begin{multline}\label{139}
 (1-Q)c_{\beta,\varepsilon}(e_i)(1-Q) { \nabla} ^{\mF,
\beta,\varepsilon}_{ e_i}
Q\\
=\frac{\varepsilon^{-1}}{2}\sum_{j=1,\, j\neq
i}^{q_2}\sum_{m=q+1}^{q+q_1}\left\langle
\nabla^{T\mM,\beta,\varepsilon}_{e_i}e_j,f_m\right\rangle
c_{\beta,\varepsilon}(e_i)c_{\beta,\varepsilon}(e_j)c_{\beta,\varepsilon}(\varepsilon f_m) .
\end{multline}

For $q+1\leq m\leq q+q_1$, one has, by (\ref{2.2}),
\begin{align}\label{141}
\left\langle \nabla^{T\mM,\beta,\varepsilon}_{e_i}e_j,f_m\right\rangle=
O\left(\varepsilon^{2}\right  ).
\end{align}

From  (\ref{138})-(\ref{141}), one gets that for $x\in s(M)$,
\begin{multline}\label{142}
 \int_{\mM_x}\left\langle\left(1-p_{T,\beta,\varepsilon}\right)c_{\beta,\varepsilon}\left(\tau
e_i\right)f_{T } {\nabla} ^{\mF, \beta,\varepsilon}_{\tau e_i}(\tau \sigma),
 c_{\beta,\varepsilon}\left(\tau e_k\right)f_{T } {\nabla}
^{\mF, \beta,\varepsilon}_{\tau e_k}(\tau \sigma)\right\rangle_{(x,Z)}k\,dv_{\mM_x}\\
= \left(O\left( \varepsilon ^2\right)+O\left({
\frac{1}{\sqrt{T}}}\right)\right)|\sigma|_x^2+O\left(\frac{1}{\sqrt{T}}\right)
\sum_{i=1}^{q+q_1}\left|\, ^Q\nabla ^{\mF,
 \beta,\varepsilon}_{
f_i}(\tau \sigma)\right|_x^2.
\end{multline}

For the third term in the right hand side of (\ref{136}),  since
$i\neq k$,  by (\ref{134a}) and a simple parity consideration, one
has that
\begin{multline}\label{143}
\left\langle\left(1-p_{T,\beta,\varepsilon}\right)c_{\beta,\varepsilon}\left(\tau
e_i\right)f_{T } {\nabla }^{\mF, \beta,\varepsilon}_{\tau e_i}(\tau \sigma),
c_{\beta,\varepsilon}\left(\tau e_k\right) \tau e_k\left(f_{T }\right)
(\tau\sigma)\right\rangle\\
= \left\langle  c_{\beta,\varepsilon}\left(\tau e_i\right)f_{T } {\nabla}
^{\mF, \beta,\varepsilon}_{\tau e_i} (\tau\sigma),
c_{\beta,\varepsilon}\left(\tau e_k\right) \tau e_k\left(f_{T }\right)
\tau\sigma\right\rangle=0.
 \end{multline}

Similarly, for the fourth term in the right hand side of (\ref{136}), one has
\begin{align}\label{144}
\left\langle\left(1-p_{T,\beta,\varepsilon}\right)c_{\beta,\varepsilon}\left(\tau e_i\right)
\tau e_i\left(f_{T }\right) \tau\sigma, c_{\beta,\varepsilon}\left(\tau
e_k\right)f_{T } {\nabla} ^{\mF,
\beta,\varepsilon}_{\tau
e_k}(\tau \sigma) \right\rangle =0.
 \end{align}

By (\ref{i3}),  (\ref{136}), (\ref{137}) and (\ref{142})-(\ref{144}), one gets
\begin{multline}\label{145}
I_3
= \left(O\left(\varepsilon^2\right)+O\left(
\frac{1}{\sqrt{T}}\right)\right)\int_{s(M)}|\sigma|^2dv_{s(M)}
\\
+O\left(\frac{1}{\sqrt{T}}\right)\int_{s(M)}
\sum_{i=1}^{q+q_1}\left|\, ^Q\nabla ^{\mF,
 \beta,\varepsilon}_{
f_i}(\tau \sigma)\right|^2dv_{s(M)}.
\end{multline}

Similarly, by (\ref{i5}), (\ref{128}), (\ref{129}), (\ref{130t}), (\ref{132}) and (\ref{11.19d}), one gets  
\begin{multline}\label{1455}
I_5
= \left(O\left(\frac{\varepsilon^2}{\beta^2}\right)+O\left(
\frac{1}{\sqrt{T}}\right)\right)\int_{s(M)}|\sigma|^2dv_{s(M)}
\\
+O\left(\frac{1}{\sqrt{T}}\right)\int_{s(M)}
\sum_{i=1}^{q+q_1}\left|\, ^Q\nabla ^{\mF,
 \beta,\varepsilon}_{
f_i}(\tau \sigma)\right|^2dv_{s(M)},
\end{multline}
while by  (\ref{i6}), (\ref{128}), (\ref{129}),  (\ref{130t}), (\ref{132a}) and (\ref{11.19f}), one gets 
\begin{multline}\label{1456}
I_6
= \left(O\left(1+\frac{\varepsilon^2}{\beta^2}\right)+O\left(
\frac{1}{\sqrt{T}}\right)\right)\int_{s(M)}|\sigma|^2dv_{s(M)}
\\
+O\left(\frac{1}{\sqrt{T}}\right)\int_{s(M)}
\sum_{i=1}^{q+q_1}\left|\, ^Q\nabla ^{\mF,
 \beta,\varepsilon}_{
f_i}(\tau \sigma)\right|^2dv_{s(M)}.
\end{multline}

\subsection{Estimates of  the terms $I_k$, $1\leq k\leq 6$, Part II}\label{s2.7}

 In this subsection, we  deal with the term left in (\ref{117}). First of all, by Lemma \ref{t11.4}  it is easy to see that the last term in (\ref{117}) verifies the following estimate,
\begin{multline}\label{117f}
\left\langle c_{\beta,\varepsilon}\left(\widetilde{\tau} f_i\right) c_{\beta,\varepsilon}\left(\widetilde{\tau}
f_j\right) f_T  p_{T,\beta,\varepsilon}\left( \widetilde{\tau}
f_i\left(f_{T }\right) \tau\sigma\right), \,^Q {\nabla} ^{\mF, \beta,\varepsilon}_{\widetilde{\tau}  f_j}(\tau
\sigma)-\tau\left(\left.\,^Q {\nabla} ^{\mF, \beta,\varepsilon}_{\widetilde{\tau}  f_j}(\tau
\sigma)\right|_{s(M)}\right)
\right\rangle \\
= O\left(
\frac{1}{\sqrt{T}}\right)\int_{s(M)}|\sigma|^2dv_{s(M)}
+O\left(\frac{1}{\sqrt{T}}\right)\int_{s(M)}
\sum_{i=1}^{q+q_1}\left|\, ^Q\nabla ^{\mF,
 \beta,\varepsilon}_{
f_i}(\tau \sigma)\right|^2dv_{s(M)}.
  \end{multline}
Thus  we need to deal with the following  term:
\begin{multline}\label{117c}
 \left\langle c_{\beta,\varepsilon}\left(\widetilde{\tau}  f_i\right) c_{\beta,\varepsilon}\left(\widetilde{\tau}
f_j\right) \widetilde{\tau}
f_i\left(f_{T }\right) f_T\tau\sigma, \,^Q {\nabla} ^{\mF, \beta,\varepsilon}_{\widetilde{\tau}  f_j}(\tau
\sigma)-\tau\left(\left.\,^Q {\nabla} ^{\mF, \beta,\varepsilon}_{\widetilde{\tau}  f_j}(\tau
\sigma)\right|_{s(M)}\right)
\right\rangle
\\
=\int_\mM\widetilde{\tau}
f_i\left(f_{T }\right) f_T \left\langle c_{\beta,\varepsilon}\left(\widetilde{\tau}  f_i\right) c_{\beta,\varepsilon}\left(\widetilde{\tau}
f_j\right) \tau\sigma, \,^Q {\nabla} ^{\mF, \beta,\varepsilon}_{\widetilde{\tau}  f_j}(\tau
\sigma)-\tau\left(\left.\,^Q {\nabla} ^{\mF, \beta,\varepsilon}_{\widetilde{\tau}  f_j}(\tau
\sigma)\right|_{s(M)}\right)
\right\rangle dv_\mM.
  \end{multline}

In view of (\ref{119}), we need to examine the first order terms (in $Z$) of the inner product  term in the right hand side of 
(\ref{117c}). 

   By  (\ref{12.6mb}) and (\ref{11.3a}), one has   the  following  pointwise formula on $\mM$,
\begin{multline}\label{13.2}
  Z  \left\langle c_{\beta,\varepsilon}\left(\widetilde{\tau}  f_i\right) c_{\beta,\varepsilon}\left(\widetilde{\tau}
f_j\right)\tau\sigma,
 \,  ^Q  {\nabla} ^{\mF, \beta,\varepsilon}_{\widetilde{\tau}  f_j}(\tau \sigma)
\right\rangle\\
 =\left\langle c_{\beta,\varepsilon}\left(\widetilde{\tau}  f_i\right) c_{\beta,\varepsilon}\left(\widetilde{\tau}
f_j\right)\tau\sigma,\, ^Q\nabla^{\mF,  \beta,\varepsilon}_Z
 \,  ^Q \nabla ^{\mF,  \beta,\varepsilon}_{\widetilde{\tau}  f_j}(\tau \sigma)
\right\rangle
\\
 =\left\langle c_{\beta,\varepsilon}\left(\widetilde{\tau}  f_i\right) c_{\beta,\varepsilon}\left(\widetilde{\tau}
f_j\right)\tau\sigma,\left(\, ^QR^{\mF,  \beta,\varepsilon}(Z,\widetilde{\tau}  f_j)+ \,  ^Q \nabla ^{\mF,
 \beta,\varepsilon}_{[Z,\widetilde{\tau}  f_j]}\right)
  \tau \sigma
\right\rangle,
  \end{multline}
  where $\, ^QR^{\mF,  \beta,\varepsilon}$ is the curvature of $\,  ^Q
 {\nabla} ^{\mF,  \beta,\varepsilon}$.

 From  Lemma \ref{t11.4},  (\ref{1311})  and (\ref{13.2}), one has, at $(x,Z)\simeq\psi(x,Z)\in\mM$,
\begin{multline}\label{13.1}
  \left\langle
c_{\beta,\varepsilon}\left(\widetilde{\tau}  f_i\right) c_{\beta,\varepsilon}\left(\widetilde{\tau}  f_j\right)\tau\sigma,
 \,  ^Q {\nabla} ^{\mF, \beta,\varepsilon}_{\widetilde{\tau}  f_j}(\tau \sigma)-\tau\left( \left.\,^Q {\nabla} ^{\mF, \beta,\varepsilon}_{\widetilde{\tau}  f_j}(\tau
\sigma)\right|_{s(M)}\right)
\right\rangle    \\
=\left\langle c_{\beta,\varepsilon}\left(\widetilde{\tau}  f_i\right) c_{\beta,\varepsilon}\left(\widetilde{\tau}
f_j\right)\tau\sigma,\left(\, ^QR^{\mF,  \beta,\varepsilon}(Z,\widetilde{\tau}  f_j)+ \,  ^Q \nabla ^{\mF,
 \beta,\varepsilon}_{[Z,\widetilde{\tau}  f_j]}\right)
  \tau \sigma
\right\rangle\\
 +O\left(|Z|^2\right)\left(
\left|\sigma \right|^2_x+
 \sum_{i=1}^{q+q_1}
\left|^Q\nabla^{\mF, \beta,\varepsilon}_{f_i}(\tau\sigma)\right|^2_x\right).
  \end{multline}

Clearly,
\begin{align}\label{146}
 ^Q R^{\mF, \beta,\varepsilon}=QR^{\mF, \beta,\varepsilon}Q
-Q\nabla^{\mF, \beta,\varepsilon}(1-Q)\nabla^{\mF, \beta,\varepsilon}Q.
 \end{align}

Recall that $f_1',\cdots,f_{q+q_1}'$ is an orthonormal basis of
$\mF\oplus \mF_1^\perp$ with respect to $g^\mF\oplus
g^{\mF_1^\perp}$ not depending on $\beta$ and $\varepsilon$, such that
$f_1',\cdots, f'_q$ is an orthonormal basis of $\mF$ verifying (\ref{3.12a}).

 By definition (cf.  (\ref{1.42a})), one has
\begin{multline}\label{13.3}
\left(QR^{\mF, \beta,\varepsilon}Q\right) (Z,\tau f_j )
=\frac{1}{4\beta^2}\sum_{s,\, t=1}^q \left\langle
R^{T\mM,\beta,\varepsilon}(Z,\tau f_j )\tau f_s,\tau
f_t\right\rangle_{\beta,\varepsilon} c_{\beta,\varepsilon}\left(\beta^{-1}\tau f_s\right)c_{\beta,\varepsilon}\left(\beta^{-1}\tau f_t\right)
\\
+\frac{\varepsilon^{2}}{4}\sum_{s,\, t=q+1}^{q+q_1}
\left\langle R^{T\mM,\beta,\varepsilon}(Z,\tau f_j )\tau f_s,\tau
f_t\right\rangle_{\beta,\varepsilon} c_{\beta,\varepsilon}(\varepsilon\tau f_s)c_{\beta,\varepsilon}(\varepsilon\tau f_t)
 \\
+\frac{\varepsilon }{2\beta}\sum_{s =1}^q\sum_{t
=q+1}^{q+q_1} \left\langle R^{T\mM,\beta,\varepsilon}(Z,\tau f_j )\tau
f_s,\tau f_t\right\rangle_{\beta,\varepsilon} c_{\beta,\varepsilon}\left(\beta^{-1}\tau f_s\right)c_{\beta,\varepsilon}(\varepsilon\tau f_t)
 .
\end{multline}

 If   $1\leq j,\,s,\,t\leq q$, one verifies, by (\ref{1.11a}), (\ref{12.1}),  (\ref{3.12b}) and (\ref{3.12c}) that\footnote{In the  following computations of terms involving curvatures, when the inner product is not indicated with subscripts $\beta$, $\varepsilon$, we view it is associated with $\beta=\varepsilon=1$.}
\begin{multline}\label{147}
\frac{1}{\beta^2}\left \langle R^{T\mM, \beta,\varepsilon}(Z,\tau f_j)\tau f_s,\tau
f_t\right\rangle_{\beta,\varepsilon} = \left \langle R^{T\mM,\beta,
\varepsilon}(f_s',f_t')Z,f_j'\right\rangle+O\left(|Z|^2\right)\\
= \left\langle
\nabla^{T\mM,\beta,\varepsilon}_{f_s'}\nabla^{T\mM,\beta,\varepsilon}_{f_t'}Z,f_j'\right\rangle
-\left\langle
\nabla^{T\mM,\beta,\varepsilon}_{f_t'}\nabla^{T\mM,\beta,\varepsilon}_{f_s'}Z,f_j'\right\rangle
-\left\langle
\nabla^{T\mM,\beta,\varepsilon}_{[f_s',f_t']}Z,f_j'\right\rangle+O\left(|Z|^2\right)\\
=-\left\langle
p\nabla^{T\mM,\beta,\varepsilon}_{f_t'}Z,\nabla^{T\mM,\beta,\varepsilon}_{f_s'}f_j'\right\rangle-\frac{1}{\beta^2\varepsilon^{2}}
\left\langle
p_1^\perp\nabla^{T\mM,\beta,\varepsilon}_{f_t'}Z,\nabla^{T\mM,\beta,\varepsilon}_{f_s'}f_j'\right\rangle
- \frac{1}{\beta^2}\left\langle
p_2^\perp\nabla^{T\mM,\beta,\varepsilon}_{f_t'}Z,\nabla^{T\mM,\beta,\varepsilon}_{f_s'}f_j'\right\rangle\\
+\left\langle
p\nabla^{T\mM,\beta,\varepsilon}_{f_s'}Z,\nabla^{T\mM,\beta,\varepsilon}_{f_t'}f_j'\right\rangle+\frac{1}{\beta^2\varepsilon^{2}}
\left\langle
p_1^\perp\nabla^{T\mM,\beta,\varepsilon}_{f_s'}Z,\nabla^{T\mM,\beta,\varepsilon}_{f_t'}f_j'\right\rangle
+\frac{1}{\beta^2}\left\langle
p_2^\perp\nabla^{T\mM,\beta,\varepsilon}_{f_s'}Z,\nabla^{T\mM,\beta,\varepsilon}_{f_t'}f_j'\right\rangle\\
+f_s'\left(\left\langle\nabla^{T\mM,\beta,\varepsilon}_{f_t'}Z,f_j'\right\rangle\right)
-f_t'\left(\left\langle\nabla^{T\mM,\beta,\varepsilon}_{f_s'}Z,f_j'\right\rangle\right)-\left\langle
\nabla^{T\mM ,\beta,\varepsilon}_{[f_s',f_t']}Z,f_j'\right\rangle
+O\left(|Z|^2\right)\\
=O\left(\varepsilon^2|Z|\right)+O\left(|Z|^2\right).
\end{multline}

 If   $1\leq j\leq q$ and $q+1\leq s,\,t\leq q+q_1$, one has, in view of (\ref{f4}),  
\begin{multline}\label{148}
\varepsilon^{2}\left \langle R^{T\mM, \beta,\varepsilon}(Z,\tau
f_j)\tau f_s,\tau f_t\right\rangle_{\beta,\varepsilon}
=\beta^2\varepsilon^{2}\left \langle R^{T\mM,\beta,
\varepsilon}(f_s',f_t')Z,f_j'\right\rangle+O\left(|Z|^2\right)\\
=\beta^2\varepsilon^{2}\left\langle
\nabla^{T\mM,\beta,\varepsilon}_{f_s'}\nabla^{T\mM,\beta,\varepsilon}_{f_t'}Z,f_j'\right\rangle
-\beta^2\varepsilon^{2}\left\langle
\nabla^{T\mM,\beta,\varepsilon}_{f_t'}\nabla^{T\mM,\beta,\varepsilon}_{f_s'}Z,f_j'\right\rangle
-\beta^2\varepsilon^{2}\left\langle
\nabla^{T\mM,\beta,\varepsilon}_{[f_s',f_t']}Z,f_j'\right\rangle+O\left(|Z|^2\right)
\\
=-\beta^2\varepsilon^{2}\left\langle
p\nabla^{T\mM,\beta,\varepsilon}_{f_t'}Z,\nabla^{T\mM,\beta,\varepsilon}_{f_s'}f_j'\right\rangle-
\left\langle
p_1^\perp\nabla^{T\mM,\beta,\varepsilon}_{f_t'}Z,\nabla^{T\mM,\beta,\varepsilon}_{f_s'}f_j'\right\rangle
- \varepsilon^{2}\left\langle
p_2^\perp\nabla^{T\mM,\beta,\varepsilon}_{f_t'}Z,\nabla^{T\mM,\beta,\varepsilon}_{f_s'}f_j'\right\rangle\\
+\beta^2\varepsilon^{2}\left\langle
p\nabla^{T\mM,\beta,\varepsilon}_{f_s'}Z,\nabla^{T\mM,\beta,\varepsilon}_{f_t'}f_j'\right\rangle+
\left\langle
p_1^\perp\nabla^{T\mM,\beta,\varepsilon}_{f_s'}Z,\nabla^{T\mM,\beta,\varepsilon}_{f_t'}f_j'\right\rangle
+ \varepsilon^{2}\left\langle
p_2^\perp\nabla^{T\mM,\beta,\varepsilon}_{f_s'}Z,\nabla^{T\mM,\beta,\varepsilon}_{f_t'}f_j'\right\rangle\\
+\beta^2\varepsilon^{2}f_s'\left(\left\langle\nabla^{T\mM,\beta,\varepsilon}_{f_t'}Z,f_j'\right\rangle\right)
-\beta^2\varepsilon^{2}f_t'\left(\left\langle
\nabla^{T\mM,\beta,\varepsilon}_{f_s'}Z,f_j'\right\rangle\right)-\beta^2\varepsilon^{2}\left\langle
\nabla^{T\mM,\beta,\varepsilon
}_{[f_s',f_t']}Z,f_j'\right\rangle+O\left(|Z|^2\right)\\
 =O\left(\varepsilon^2|Z|\right)+O\left(|Z|^2\right).
\end{multline}

 If  $1\leq j,\, t\leq q$ and $q+1\leq  s\leq q+q_1$, by Lemma \ref{t11.3} one has
\begin{multline}\label{148a}
\frac{\varepsilon}{\beta} \left \langle R^{T\mM,\beta,
\varepsilon}(Z,\tau f_j) \tau f_s,\tau
f_t\right\rangle_{\beta,\varepsilon}
=\beta\varepsilon \left \langle R^{T\mM,\beta,
\varepsilon}(f_s',f_t')Z,f_j'\right\rangle+O\left(|Z|^2\right)\\
=\beta\varepsilon\left\langle
\nabla^{T\mM,\beta,\varepsilon}_{f_s'}\nabla^{T\mM,\beta,\varepsilon}_{f_t'}Z,f_j'\right\rangle
-\beta\varepsilon\left\langle
\nabla^{T\mM,\beta,\varepsilon}_{f_t'}\nabla^{T\mM,\beta,\varepsilon}_{f_s'}Z,f_j\right\rangle
-\beta\varepsilon\left\langle
\nabla^{T\mM,\beta,\varepsilon}_{[f_s',f_t']}Z,f_j'\right\rangle+O\left(|Z|^2\right)
\\
=-\beta\varepsilon\left\langle
p\nabla^{T\mM,\beta,\varepsilon}_{f_t'}Z,\nabla^{T\mM,\beta,\varepsilon}_{f_s'}f_j'\right\rangle-\frac{1}{\beta\varepsilon }
\left\langle
p_1^\perp\nabla^{T\mM,\beta,\varepsilon}_{f_t'}Z,\nabla^{T\mM,\beta,\varepsilon}_{f_s'}f_j'\right\rangle
-\frac{\varepsilon}{\beta}\left\langle
p_2^\perp\nabla^{T\mM,\beta,\varepsilon}_{f_t'}Z,\nabla^{T\mM,\beta,\varepsilon}_{f_s'}f_j'\right\rangle
\\
+\beta\varepsilon\left\langle
p\nabla^{T\mM,\beta,\varepsilon}_{f_s'}Z,\nabla^{T\mM,\beta,\varepsilon}_{f_t'}f_j'\right\rangle+
\frac{1}{\beta\varepsilon }\left\langle
p_1^\perp\nabla^{T\mM,\beta,\varepsilon}_{f_s'}Z,\nabla^{T\mM,\beta,\varepsilon}_{f_t'}f_j'\right\rangle
+\frac{\varepsilon}{\beta} \left\langle
p_2^\perp\nabla^{T\mM,\beta,\varepsilon}_{f_s'}Z,\nabla^{T\mM,\beta,\varepsilon}_{f_t'}f_j'\right\rangle
\\
+\beta\varepsilon f_s'\left(\left\langle\nabla^{T\mM,\beta,\varepsilon}_{f_t'}Z,f_j'\right\rangle\right)
- \beta\varepsilon f_t'\left(\left\langle\nabla^{T\mM,\beta,\varepsilon}_{f_s'}Z,f_j'\right\rangle\right)- \beta\varepsilon\left\langle
\nabla^{T\mM,\beta,\varepsilon }_{[f_s',f_t']}Z,f_j'\right\rangle
+O\left(|Z|^2\right)
\\
=O\left(\frac{\varepsilon|Z|}{\beta}\right)+O\left(|Z|^2\right).
\end{multline}

 If  $q+1\leq j\leq q+q_1$ and $1\leq s,\,t\leq q$, one has
\begin{multline}\label{149}
\frac{1}{\beta^2}\left \langle R^{T\mM, \beta,\varepsilon}(Z,\tau f_j)\tau f_s,\tau
f_t\right\rangle_{\beta,\varepsilon} =\frac{1}{\beta^2\varepsilon^{2}}
\left \langle R^{T\mM,\beta,
\varepsilon}(f_s',f_t')Z,f_j'\right\rangle+O\left(|Z|^2\right)\\
=\frac{1}{\beta^2\varepsilon^{2}}\left\langle
\nabla^{T\mM,\beta,\varepsilon}_{f_s'}\nabla^{T\mM,\beta,\varepsilon}_{f_t'}Z,f_j'\right\rangle
-\frac{1}{\beta^2\varepsilon^{2}}\left\langle
\nabla^{T\mM,\beta,\varepsilon}_{f_t'}\nabla^{T\mM,\beta,\varepsilon}_{f_s'}Z,f_j'\right\rangle
-\frac{1}{\beta^2\varepsilon^{2}}\left\langle
\nabla^{T\mM,\beta,\varepsilon}_{[f_s',f_t']}Z,f_j'\right\rangle+O\left(|Z|^2\right)\\
=-\left\langle
p\nabla^{T\mM,\beta,\varepsilon}_{f_t'}Z,\nabla^{T\mM,\beta,\varepsilon}_{f_s'}f_j'\right\rangle-\frac{1}{\beta^2\varepsilon^{2}}
\left\langle
p_1^\perp\nabla^{T\mM,\beta,\varepsilon}_{f_t'}Z,\nabla^{T\mM,\beta,\varepsilon}_{f_s'}f_j'\right\rangle
- \frac{1}{\beta^2}\left\langle
p_2^\perp\nabla^{T\mM,\beta,\varepsilon}_{f_t'}Z,\nabla^{T\mM,\beta,\varepsilon}_{f_s'}f_j'\right\rangle\\
+\left\langle
p\nabla^{T\mM,\beta,\varepsilon}_{f_s'}Z,\nabla^{T\mM,\beta,\varepsilon}_{f_t'}f_j'\right\rangle+\frac{1}{\beta^2\varepsilon^{2}}
\left\langle
p_1^\perp\nabla^{T\mM,\beta,\varepsilon}_{f_s'}Z,\nabla^{T\mM,\beta,\varepsilon}_{f_t'}f_j'\right\rangle
+ \frac{1}{\beta^2}\left\langle
p_2^\perp\nabla^{T\mM,\beta,\varepsilon}_{f_s'}Z,\nabla^{T\mM,\beta,\varepsilon}_{f_t'}f_j'\right\rangle\\
+\frac{1}{\beta^2\varepsilon^{2}}f_s'\left(\left\langle\nabla^{T\mM,\beta,\varepsilon}_{f_t'}Z,f_j'\right\rangle\right)
-\frac{1}{\beta^2\varepsilon^{2}}f_t'\left(\left\langle\nabla^{T\mM,\beta,\varepsilon}_{f_s'}Z,f_j'\right\rangle\right)-\frac{1}{\beta^2\varepsilon^{2}}\left\langle
\nabla^{T\mM,\beta,\varepsilon }_{[f_s',f_t']}Z,f_j'\right\rangle
+O\left(|Z|^2\right)\\
=O\left(\frac{|Z|}{\beta^2}\right)+O\left(|Z|^2\right).
\end{multline}

 If  $q+1\leq j,\,s,\,t\leq q+q_1$, one has
\begin{multline}\label{150}
\varepsilon^{2}\left \langle R^{T\mM,\beta, \varepsilon}(Z,\tau
f_j)\tau f_s,\tau f_t\right\rangle_{\beta,\varepsilon} = \left \langle
R^{T\mM,\beta,
\varepsilon}(f_s',f_t')Z,f_j'\right\rangle+O\left(|Z|^2\right)\\
= \left\langle
\nabla^{T\mM,\beta, \varepsilon}_{f_s'}\nabla^{T\mM,\beta, \varepsilon}_{f_t'}Z,f_j'\right\rangle
- \left\langle
\nabla^{T\mM,\beta, \varepsilon}_{f_t'}\nabla^{T\mM,\beta, \varepsilon}_{f_s'}Z,f_j'\right\rangle
- \left\langle
\nabla^{T\mM,\beta, \varepsilon}_{[f_s',f_t']}Z,f_j'\right\rangle+O\left(|Z|^2\right) \\
=-\beta^2\varepsilon^{2}\left\langle
p\nabla^{T\mM,\beta, \varepsilon}_{f_t'}Z,\nabla^{T\mM,\beta, \varepsilon}_{f_s'}f_j'\right\rangle-
\left\langle
p_1^\perp\nabla^{T\mM,\beta, \varepsilon}_{f_t'}Z,\nabla^{T\mM,\beta, \varepsilon}_{f_s'}f_j'\right\rangle
- \varepsilon^{2}\left\langle
p_2^\perp\nabla^{T\mM,\beta, \varepsilon}_{f_t'}Z,\nabla^{T\mM,\beta, \varepsilon}_{f_s'}f_j'\right\rangle\\
+\beta^2\varepsilon^{2}\left\langle
p\nabla^{T\mM,\beta, \varepsilon}_{f_s'}Z,\nabla^{T\mM,\beta,\varepsilon}_{f_t'}f_j'\right\rangle+
\left\langle
p_1^\perp\nabla^{T\mM,\beta, \varepsilon}_{f_s'}Z,\nabla^{T\mM,\beta, \varepsilon}_{f_t'}f_j'\right\rangle
+ \varepsilon^{2}\left\langle
p_2^\perp\nabla^{T\mM,\beta, \varepsilon}_{f_s'}Z,\nabla^{T\mM,\beta, \varepsilon}_{f_t'}f_j'\right\rangle\\
+
f_s'\left(\left\langle\nabla^{T\mM,\beta, \varepsilon}_{f_t'}Z,f_j'\right\rangle\right)
-
f_t'\left(\left\langle\nabla^{T\mM,\beta, \varepsilon}_{f_s'}Z,f_j'\right\rangle\right)-
\left\langle \nabla^{T\mM
,\beta, \varepsilon}_{[f_s',f_t']}Z,f_j'\right\rangle
+O\left(|Z|^2\right)\\
=O\left( {|Z|}\right)+O\left(|Z|^2\right).
\end{multline}

If   $q+1\leq j, \,t\leq q+q_1$ and $1\leq s\leq q$, one has
\begin{multline}\label{150c}
-\frac{\varepsilon}{\beta}\left \langle R^{T\mM,\beta,
\varepsilon}(Z,\tau f_j)\tau f_s,\tau f_t\right\rangle_{\beta,\varepsilon}
= {\beta\varepsilon}  \left \langle R^{T\mM,\beta,
\varepsilon}(Z,  f_j')  f_t',  f_s'\right\rangle+O\left(|Z|^2\right)\\
= {\beta\varepsilon} \left\langle
\nabla^{T\mM,\beta,\varepsilon}_{Z}\nabla^{T\mM,\beta,\varepsilon}_{f_j'}f_t',f_s'\right\rangle
- {\beta\varepsilon} \left\langle
\nabla^{T\mM,\beta,\varepsilon}_{f_j'}\nabla^{T\mM,\beta,\varepsilon}_{Z}f_t',f_s'\right\rangle
-  {\beta\varepsilon} \left\langle
\nabla^{T\mM,\beta,\varepsilon}_{[Z,f_j']}f_t',f_s'\right\rangle+O\left(|Z|^2\right)\end{multline}
$$
=-\beta\varepsilon \left\langle
p\nabla^{T\mM,\beta,\varepsilon}_{f_j'}f_t',\nabla^{T\mM,\beta,\varepsilon}_{Z}f_s'\right\rangle-
\frac{1}{\beta\varepsilon} \left\langle
p_1^\perp\nabla^{T\mM,\beta,\varepsilon}_{f_j'}f_t',\nabla^{T\mM,\beta,\varepsilon}_{Z}f_s'\right\rangle
-\frac{\varepsilon}{\beta}  \left\langle
p_2^\perp\nabla^{T\mM,\beta,\varepsilon}_{f_j'}f_t',\nabla^{T\mM,\beta,\varepsilon}_{Z}f_s'\right\rangle
$$
$$
+\beta\varepsilon \left\langle
p\nabla^{T\mM,\beta,\varepsilon}_{Z}f_t',\nabla^{T\mM,\beta,\varepsilon}_{f_j'}f_s'\right\rangle+
\frac{1}{\beta\varepsilon} \left\langle
p_1^\perp\nabla^{T\mM,\beta,\varepsilon}_{Z}f_t',\nabla^{T\mM,\beta,\varepsilon}_{f_j'}f_s'\right\rangle
+ \frac{\varepsilon}{\beta} \left\langle
p_2^\perp\nabla^{T\mM,\beta,\varepsilon}_{Z}f_t',\nabla^{T\mM,\beta,\varepsilon}_{f_j'}f_s'\right\rangle$$
$$
+ {\beta\varepsilon}
Z\left(\left\langle\nabla^{T\mM,\beta,\varepsilon}_{f_j'}f_t',f_s'\right\rangle\right)
- {\beta\varepsilon}
f_j'\left(\left\langle\nabla^{T\mM,\beta,\varepsilon}_{Z}f_t',f_s'\right\rangle\right)- {\beta\varepsilon}
\left\langle \nabla^{T\mM
,\beta,\varepsilon}_{[Z,f_j']}f_t',f_s'\right\rangle+O\left(|Z|^2\right)$$
$$
=O\left(\frac{\varepsilon|Z|}
{\beta}\right)+O\left(|Z|^2\right).$$

Now  from (\ref{134a})-(\ref{11.20}), one
verifies easily that
\begin{align}\label{151}
  (1-Q)\nabla^{\mF, \beta,\varepsilon}_{Z}Q=O\left(
{\varepsilon} |Z|\right)+O\left(|Z|^2\right).
 \end{align}
 Similarly, one has
\begin{align}\label{152}
  Q\nabla^{\mF, \beta,\varepsilon}_{Z}(1-Q)=O\left(
{\varepsilon} |Z|\right)+O\left(|Z|^2\right).
 \end{align}

 On the other hand, by (\ref{127})-(\ref{127b}), one finds that
 for $1\leq j\leq q$,
\begin{align}\label{153}
  (1-Q)\nabla^{\mF, \beta,\varepsilon}_{\tau f_j}Q=O\left(
\varepsilon \right)+O_{\beta,\varepsilon}(|Z|).
 \end{align}
 Similarly,
\begin{align}\label{154}
 Q \nabla^{\mF, \beta,\varepsilon}_{\tau f_j}(1-Q)=O\left(
 \varepsilon \right)+O_{\beta,\varepsilon}(|Z|).
 \end{align}

 While for $q+1\leq j\leq q+q_1$, by (\ref{127}), (\ref{127d}) and
 (\ref{127e}), one has
\begin{align}\label{155}
  (1-Q)\nabla^{\mF, \beta,\varepsilon}_{\tau f_j}Q=O\left(\beta^{-1}+
\varepsilon^{-1}\right)+O_{\beta,\varepsilon}(|Z|).
 \end{align}
 Similarly,
\begin{align}\label{156}
 Q \nabla^{\mF, \beta,\varepsilon}_{\tau f_j}(1-Q)=O\left(\beta^{-1}+
\varepsilon^{-1}\right)+O_{\beta,\varepsilon}(|Z|).
 \end{align}

 From   (\ref{146})-(\ref{156}), one
 gets that  if $1\leq i, \  j\leq q+q_1 $,   then the
 following identity holds at $(x,Z)$ near $s(M)$,
\begin{align}\label{157}
 \left\langle c_{\beta,\varepsilon}\left( \widetilde{\tau} f_i\right)c_{\beta,\varepsilon}\left(\widetilde{\tau} f_j\right) \, ^Q
R^{\mF,\beta,\varepsilon}\left(Z,\widetilde{\tau} f_j\right)\tau
\sigma,\tau\sigma\right\rangle
=\left(O\left(  \frac{\varepsilon}{\beta^2} |Z|\right)+O\left(|Z|^2\right)\right)|\sigma|^2.
 \end{align}

Now we examine the   term
$$\left\langle c_{\beta,\varepsilon}\left(\widetilde{\tau} f_i\right) c_{\beta,\varepsilon}\left(\widetilde{\tau}
f_j\right)\tau\sigma,  \,  ^Q \nabla ^{\mF,
 \beta,\varepsilon}_{[Z,\widetilde{\tau} f_j]}
 ( \tau \sigma)
\right\rangle$$ in (\ref{13.1}).

By (\ref{12.6mb}) and
(\ref{131a}), one has
\begin{align}\label{13.5}
\left(p+p_1^\perp\right)[Z,\tau
f_j]=-\left(p+p_1^\perp\right)\nabla^{T\mM,\beta,\varepsilon}_{\tau
f_j}Z
=-\sum_{k=1}^{q_2}z_k\left(p+p_1^\perp\right)\nabla^{T\mM,\beta,\varepsilon}_{\tau
f_j}(\tau e_k) .
 \end{align}

For any $1\leq k\leq q_2$, $1\leq j\leq q$, by (\ref{3.12c}) one verifies easily
that
\begin{multline}\label{159}
 \left(p+p_1^\perp\right)\nabla^{T\mM,\beta,\varepsilon}_{\tau f_j}({\tau e_k})=\sum_{s=1}^q\left\langle
 \nabla^{T\mM,\beta,\varepsilon}_{\tau f_j}(\tau e_k),f_s'\right\rangle
 f_s'+ \sum_{s=q+1}^{q+q_1}\left\langle
 \nabla^{T\mM,\beta,\varepsilon}_{\tau f_j}(\tau e_k),f_s'\right\rangle
 f_s'\\
 =\sum_{s=1}^q   O_{\beta,\varepsilon}\left(
 |Z|\right)
 f_s'+\sum_{s=q+1}^{q+q_1}\left(O\left(\varepsilon^{2} \right)+O_{\beta,\varepsilon}\left(
 |Z|\right)\right)
 f_s'.
 \end{multline}

 By (\ref{13.5}) and (\ref{159}),
   for  $1\leq j\leq q$, one has,
\begin{multline}\label{159a}
\frac{1}{\beta}  \, ^Q\nabla^{\mF, \beta,\varepsilon}_{\left(p+p_1^\perp\right)[Z,\tau
f_j]}(\tau \sigma) =\sum_{i=1}^q O\left( |Z|^2\right)\,
^Q\nabla^{\mF, \beta,\varepsilon}_{  f_i'}(\tau\sigma) \\
+\sum_{i=q+1}^{q+q_1}O\left(\frac{\varepsilon^{2}|Z|}{\beta}+|Z|^2\right)
\, ^Q\nabla^{\mF, \beta,\varepsilon}_{  f_i'}(\tau \sigma) .
 \end{multline}

 Similarly, for $1\leq k\leq q_2$, $q+1\leq j\leq q+q_1$, one has
\begin{align}\label{160}
 p \nabla^{T\mM,\beta,\varepsilon}_{\tau f_j}(\tau e_k)=\sum_{s=1}^q\left\langle
 \nabla^{T\mM,\beta,\varepsilon}_{\tau f_j}(\tau e_k),  f_s'\right\rangle
 f_s'
 =\sum_{s=1}^q
O\left(
\beta^{-2}\right)f_s '+\sum_{s=1}^q
O_{\beta,\varepsilon}\left(
|Z|\right)f_s '.
 \end{align}
 Thus, for   $q+1\leq j\leq q+q_1$, one has,
\begin{align}\label{160a}
\varepsilon   \,
^Q\nabla^{\mF, \beta,\varepsilon}_{ p [Z,\tau
f_j]}(\tau \sigma)
=\sum_{i=1}^q\left(O\left(\frac{\varepsilon|Z|}{\beta^{2}}\right)+O\left(|Z|^2\right)\right)
\, ^Q\nabla^{\mF, \beta,\varepsilon}_{f_i'}(\tau\sigma)    .
 \end{align}

For  $1\leq k\leq q_2$, $q+1\leq j\leq q+q_1$, one has by (\ref{12.2})
\begin{align}\label{160b}
 p_1^\perp \nabla^{T\mM,\beta,\varepsilon}_{\tau f_j}(\tau e_k)=\sum_{s=q+1}^{q+q_1}\left\langle
 \nabla^{T\mM,\beta,\varepsilon}_{  f_j'}(\tau e_k),  f_s'\right\rangle
 f_s' +O_{\beta,\varepsilon}\left(
|Z|\right)
.
 \end{align}
Thus for 
  $q+1\leq j\leq q+q_1$, one has by (\ref{13.5}) and (\ref{160b}), 
\begin{align}\label{160c}
 \varepsilon   \,
^Q\nabla^{\mF, \beta,\varepsilon}_{ p_1^\perp [Z,\tau
f_j]}(\tau \sigma)
=-\varepsilon\sum_{k=1}^{q_2}\sum_{s=q+1}^{q+q_1} \left(z_k\left\langle
 \nabla^{T\mM,\beta,\varepsilon}_{  f_j'}(\tau e_k),  f_s'\right\rangle+O\left(|Z|^2\right)\right)
\, ^Q\nabla^{\mF, \beta,\varepsilon}_{f_s'}(\tau\sigma)  
.
 \end{align}

 Now  for any $1\leq j\leq q +q_1$, one has
\begin{multline}\label{13.6}
  p_2^\perp [Z,\tau
f_j]=p_2^\perp \nabla^{T\mM}_Z(\tau f_j)- \nabla^{ \mF_2^\perp
}_{\tau f_j}Z\\
 =\sum_{k=1}^{q_2}\left\langle \nabla^{T\mM}_Z(\tau
f_j),\tau e_k\right\rangle\tau e_k -\sum_{k=1}^{q_2}\tau
f_j(z_k)\tau e_k-\sum_{k=1}^{q_2}z_k \nabla^{\mF_2^\perp }_{\tau
f_j}(\tau e_k) .
 \end{multline}

 From (\ref{13.6}) and Lemmas \ref{t11.4}, \ref{t11.5}, one finds
\begin{multline}\label{13.7}
  \,
^Q\nabla^{\mF, \beta,\varepsilon}_{ p_2^\perp [Z,\tau f_j]}(\tau
\sigma) = -\sum_{k=1}^{q_2}\tau
f_j(z_k)\,^Q\nabla^{\mF, \beta,\varepsilon}_{\tau
e_k}(\tau\sigma)\\
+O\left(|Z|^2\right)\left(|\sigma|_x+\sum_{k=1}^{q+q_1}\left|
\,^Q\nabla^{\mF, \beta,\varepsilon}_{
f_k}(\tau\sigma)\right|_x\right).
 \end{multline}

For another section $\sigma'$ on $s(M)$,   one has  
\begin{multline}\label{13.8}
Z\left\langle \,^Q\nabla^{\mF, \beta,\varepsilon}_{\tau e_k}(\tau\sigma),\tau
\sigma'\right\rangle=
\left\langle\,^Q\nabla^{\mF, \beta,\varepsilon}_Z
\,^Q\nabla^{\mF, \beta,\varepsilon}_{\tau e_k}(\tau\sigma),\tau
\sigma'\right\rangle \\
=\left\langle \,^QR^{\mF, \beta,\varepsilon}(Z,{\tau e_k})\tau\sigma,\tau
\sigma'\right\rangle +\left\langle
\,^Q\nabla^{\mF, \beta,\varepsilon}_{[Z,\tau e_k]}(\tau\sigma),\tau
\sigma'\right\rangle.
 \end{multline}

As in (\ref{13.6}), one verifies
\begin{align}\label{13.9}
    [Z,\tau
  e_k]=- \nabla^{ \mF_2^\perp }_{\tau e_k}Z
=-\sum_{j=1}^{q_2}\tau e_k(z_j)\tau e_j-\sum_{j=1}^{q_2}z_j
\nabla^{\mF_2^\perp }_{\tau e_k}(\tau e_j) .
 \end{align}

 Clearly,
\begin{align}\label{13.10}
 \tau e_k(z_j)=\delta_{kj}+O(|Z|)   .
 \end{align}

 By Lemma \ref{t11.5} and (\ref{13.8})-(\ref{13.10}), one deduces
 that
\begin{multline}\label{13.11}
 \left\langle \,^Q\nabla^{\mF, \beta,\varepsilon}_{\tau e_k}(\tau\sigma),\tau
\sigma'\right\rangle =\frac{1}{2}\left\langle
\,^QR^{\mF,  \beta,\varepsilon}(Z,{\tau e_k})\tau\sigma,\tau
\sigma'\right\rangle +O\left(|Z|^2\right)\\
  =\frac{1}{2}\sum_{m=1}^{q_2}z_m\left\langle
\,^QR^{\mF,  \beta,\varepsilon}(\tau e_m,{\tau
e_k})\tau\sigma,\tau \sigma'\right\rangle +O\left(|Z|^2\right)  .
 \end{multline}

From (\ref{13.7}) and (\ref{13.11}), one gets
\begin{multline}\label{13.11a}
\left\langle c_{\beta,\varepsilon}(\widetilde{\tau} f_i)c_{\beta,\varepsilon}(\widetilde{\tau} f_j)\tau\sigma, \, ^Q\nabla^{\mF, \beta,\varepsilon}_{p_2^\perp\left[Z,\widetilde{\tau}f_j\right]}(\tau\sigma)\right\rangle _{(x,Z)}
\\
=- \frac{1}{2}\left\langle c_{\beta,\varepsilon}(\widetilde{\tau} f_i)c_{\beta,\varepsilon}(\widetilde{\tau} f_j)\tau\sigma, \,^QR^{\mF,  \beta,\varepsilon}\left(Z,\nabla^{\mF_2^\perp}_{\widetilde{\tau} f_j}Z\right)\tau\sigma\right\rangle _{(x,Z)}+O (|Z|^2) .
 \end{multline}

From (\ref{12.3f}),  (\ref{146}), (\ref{151}), (\ref{152}) and (\ref{13.11a}), one gets that for $1\leq i,\ j\leq q+q_1$, 
\begin{multline}\label{13.11b}
\left\langle c_{\beta,\varepsilon}(\widetilde{\tau} f_i)c_{\beta,\varepsilon}(\widetilde{\tau} f_j)\tau\sigma, \, ^Q\nabla^{\mF, \beta,\varepsilon}_{p_2^\perp\left[Z,\widetilde{\tau}f_j\right]}(\tau\sigma)\right\rangle _{(x,Z)}
\\
=- \frac{1}{2}\left\langle c_{\beta,\varepsilon}(\widetilde{\tau} f_i)c_{\beta,\varepsilon}(\widetilde{\tau} f_j)\tau\sigma,  R^{\mF, \beta,\varepsilon}\left(Z,\nabla^{\mF_2^\perp}_{\widetilde{\tau} f_j}Z\right)\tau\sigma\right\rangle _{(x,Z)}+O \left(\frac{\varepsilon^2|Z|}{|\tau f_j|_{\beta,\varepsilon}}\right) +O\left (|Z|^2\right) .
 \end{multline}

As in (\ref{13.3}), we have
\begin{multline}\label{13.12}
\left(QR^{\mF,  \beta,\varepsilon}Q\right) (\tau e_m,\tau e_k )
\\
=
\frac{1}{4\beta^2}\sum_{s,\, t=1}^q \left\langle
R^{T\mM,\beta,\varepsilon}(\tau e_m,\tau e_k ) \tau f_s,
\tau f_t\right\rangle_{\beta,\varepsilon} c_{\beta,\varepsilon}\left( \beta^{-1}\tau f_s\right)c_{\beta,\varepsilon}\left(\beta^{-1} \tau f_t\right)
 \\
+\frac{\varepsilon ^2}{4}\sum_{s,\, t=q+1}^{q+q_1}
\left\langle R^{T\mM,\beta,\varepsilon}(\tau e_m,\tau e_k ) \tau f_s,
\tau f_t\right\rangle_{\beta,\varepsilon} c_{\beta,\varepsilon}(\varepsilon\tau f_s)c_{\beta,\varepsilon}(\varepsilon \tau f_t)
\\
+\frac{\varepsilon  }{2\beta}\sum_{s =1}^q\sum_{t
=q+1}^{q+q_1} \left\langle R^{T\mM,\beta,\varepsilon}(\tau e_m,\tau e_k
) \tau f_s,\tau f_t\right\rangle_{\beta,\varepsilon} c_{\beta,\varepsilon}\left(\beta^{-1}\tau  f_s\right)c_{\beta,\varepsilon}(\varepsilon\tau
f_t)
 .
 \end{multline}

 If   $1\leq  s,\,t\leq q$, one has, in view of  (\ref{3.12c}) and (\ref{12.27a}), that
\begin{multline}\label{13.13}
\frac{1}{\beta^2}\left \langle R^{T\mM, \beta,\varepsilon}(\tau e_m,\tau e_k)\tau
f_s,\tau f_t\right\rangle_{\beta,\varepsilon} =\left \langle R^{T\mM,\beta,
\varepsilon}(\tau e_m,\tau e_k)f_s',  f_t'\right\rangle+O_{\beta,\varepsilon}\left(|Z|\right)\\
=\left\langle \nabla^{T\mM,\beta,\varepsilon}_{\tau
e_m}\nabla^{T\mM,\beta,\varepsilon}_{\tau e_k}f_s',  f_t'\right\rangle
-\left\langle \nabla^{T\mM,\beta,\varepsilon}_{\tau
e_k}\nabla^{T\mM,\beta,\varepsilon}_{\tau e_m}f_s',  f_t'\right\rangle
-\left\langle
\nabla^{T\mM,\beta,\varepsilon}_{[\tau e_m,\tau e_k]}f_s',  f_t'\right\rangle+O_{\beta,\varepsilon}\left(|Z| \right)\\
=-\left\langle p\nabla^{T\mM,\beta,\varepsilon}_{\tau
e_k}f_s',\nabla^{T\mM,\beta,\varepsilon}_{\tau
e_m}f_t'\right\rangle-\frac{1}{\beta^2\varepsilon^{2}}
\left\langle p_1^\perp\nabla^{T\mM,\beta,\varepsilon}_{\tau
e_k}f_s',\nabla^{T\mM,\beta,\varepsilon}_{\tau e_m}f_t'\right\rangle -
\frac{1}{\beta^2}\left\langle
p_2^\perp\nabla^{T\mM,\beta,\varepsilon}_{\tau e_k}f_s',\nabla^{T\mM,\beta,\varepsilon}_{\tau e_m}f_t'\right\rangle\\
+\left\langle p\nabla^{T\mM,\beta,\varepsilon}_{\tau
e_m}f_s',\nabla^{T\mM,\beta,\varepsilon}_{\tau
e_k}f_t'\right\rangle+\frac{1}{\beta^2\varepsilon^{2}}
\left\langle p_1^\perp\nabla^{T\mM,\beta,\varepsilon}_{\tau
e_m}f_s',\nabla^{T\mM,\beta,\varepsilon}_{\tau e_k}f_t'\right\rangle
+\frac{1}{\beta^2}\left\langle
p_2^\perp\nabla^{T\mM,\beta,\varepsilon}_{\tau e_m}f_s',\nabla^{T\mM,\beta,\varepsilon}_{\tau e_k}f_t'\right\rangle\\
+\tau e_m\left(\left\langle\nabla^{T\mM,\beta,\varepsilon}_{\tau
e_k}f_s',  f_t'\right\rangle\right) -\tau
e_k\left(\left\langle\nabla^{T\mM,\beta,\varepsilon}_{\tau e_m}f_s',
f_t'\right\rangle\right)-\left\langle \nabla^{T\mM,\beta,\varepsilon }_{[\tau
e_m,\tau e_k]}f_s',  f_t'\right\rangle
+O_{\beta,\varepsilon}\left(|Z| \right)\\
=O\left(\frac{\varepsilon^2}{\beta^{2}}\right)
+O_{\beta,\varepsilon}\left(|Z| \right)  .
\end{multline}

 If  $1\leq s\leq q$,  $q+1\leq   t\leq q+q_1$, one has
\begin{multline}\label{13.15}
\frac{\varepsilon}{\beta}\left \langle R^{T\mM,\beta,
\varepsilon}(\tau e_m,\tau e_k)\tau f_s,\tau
f_t\right\rangle_{\beta,\varepsilon} =\frac{1}{\beta\varepsilon }\left \langle R^{T\mM,\beta,
\varepsilon}(\tau e_m,\tau e_k)f_s',  f_t'\right\rangle+O_{\beta,\varepsilon}\left(|Z|\right)
\\
= \frac{1}{\beta\varepsilon }\left\langle
\nabla^{T\mM,\beta,\varepsilon}_{\tau
e_m}\nabla^{T\mM,\beta,\varepsilon}_{\tau e_k}f_s',  f_t'\right\rangle
-\frac{1}{\beta\varepsilon }\left\langle
\nabla^{T\mM,\beta,\varepsilon}_{\tau
e_k}\nabla^{T\mM,\beta,\varepsilon}_{\tau e_m}f_s',  f_t'\right\rangle
-\frac{1}{\beta\varepsilon }\left\langle
\nabla^{T\mM,\beta,\varepsilon}_{[\tau e_m,\tau e_k]}f_s',  f_t'\right\rangle+O_{\beta,\varepsilon}\left(|Z| \right)
\end{multline}
$$
=- {\beta\varepsilon }\left\langle
p\nabla^{T\mM,\beta,\varepsilon}_{\tau
e_k}f_s',\nabla^{T\mM,\beta,\varepsilon}_{\tau e_m}f_t'\right\rangle-
\frac{1}{\beta\varepsilon }\left\langle
p_1^\perp\nabla^{T\mM,\beta,\varepsilon}_{\tau
e_k}f_s',\nabla^{T\mM,\beta,\varepsilon}_{\tau e_m}f_t'\right\rangle -
\frac{\varepsilon}{\beta  }\left\langle
p_2^\perp\nabla^{T\mM,\beta,\varepsilon}_{\tau e_k}f_s',\nabla^{T\mM,\beta,\varepsilon}_{\tau e_m}f_t'\right\rangle
$$
$$
+ {\beta\varepsilon }\left\langle
p\nabla^{T\mM,\beta,\varepsilon}_{\tau
e_m}f_s',\nabla^{T\mM,\beta,\varepsilon}_{\tau
e_k}f_t'\right\rangle+\frac{1}{\beta\varepsilon }
\left\langle p_1^\perp\nabla^{T\mM,\beta,\varepsilon}_{\tau
e_m}f_s',\nabla^{T\mM,\beta,\varepsilon}_{\tau e_k}f_t'\right\rangle +
\frac{\varepsilon}{\beta  }\left\langle
p_2^\perp\nabla^{T\mM,\beta,\varepsilon}_{\tau e_m}f_s',\nabla^{T\mM,\beta,\varepsilon}_{\tau e_k}f_t'\right\rangle
$$
$$
+\frac{1}{\beta\varepsilon }\tau
e_m\left(\left\langle\nabla^{T\mM,\beta,\varepsilon}_{\tau e_k}f_s',
f_t'\right\rangle\right) -\frac{1}{\beta\varepsilon }\tau
e_k\left(\left\langle\nabla^{T\mM,\beta,\varepsilon}_{\tau e_m}f_s',
f_t'\right\rangle\right)-\frac{1}{\beta\varepsilon }\left\langle \nabla^{T\mM,\beta,\varepsilon }_{[\tau
e_m,\tau e_k]}f_s',  f_t'\right\rangle
+O_{\beta,\varepsilon}\left(|Z| \right)
$$
$$
=O\left(\frac{\varepsilon}{\beta}\right)+O_{\beta,\varepsilon}\left(|Z|
\right) .
$$

If   $q+1\leq  s,\,t\leq q+q_1$, one has, in view of (\ref{2.2}) and (\ref{12.2}),
\begin{multline}\label{13.14a}
\varepsilon^{2}\left \langle R^{T\mM,\beta, \varepsilon}(\tau
e_m,\tau e_k)\tau f_s,\tau f_t\right\rangle_{\beta,\varepsilon} =\varepsilon^{2}\left
\langle R^{T\mM,\beta,
\varepsilon}(f_s',  f_t')\tau e_m,\tau e_k\right\rangle+O_{\beta,\varepsilon}\left(|Z|\right)\\
=\varepsilon^{2}\left\langle \nabla^{T\mM,\beta,\varepsilon}_{f_s'}\nabla^{T\mM,\beta,\varepsilon}_{f_t'}\tau e_m,  \tau e_k\right\rangle
-\varepsilon^{2}\left\langle \nabla^{T\mM,\beta,\varepsilon}_{f_t'}\nabla^{T\mM,\beta,\varepsilon}_{f_s'}\tau e_m,  \tau e_k\right\rangle
\\
-\varepsilon^{2}\left\langle
\nabla^{T\mM,\beta,\varepsilon}_{[f_s',  f_t']}\tau e_m,\tau
e_k\right\rangle+O_{\beta,\varepsilon}\left(|Z| \right)
\\
 =-
{\varepsilon^{2}}{\beta^2}\left\langle
p\nabla^{T\mM,\beta,\varepsilon}_{f_t'}\tau
e_m,\nabla^{T\mM,\beta,\varepsilon}_{f_s'}\tau e_k\right\rangle-
\left\langle p_1^\perp\nabla^{T\mM,\beta,\varepsilon}_{f_t'}\tau
e_m,\nabla^{T\mM,\beta,\varepsilon}_{f_s'}\tau e_k\right\rangle
\\
-
 {\varepsilon }^{2}\left\langle
p_2^\perp\nabla^{T\mM,\beta,\varepsilon}_{f_t'}\tau
e_m,\nabla^{T\mM,\beta,\varepsilon}_{f_s'}\tau e_k\right\rangle
\\
+ {\varepsilon^{2}}{\beta^2}\left\langle
p\nabla^{T\mM,\beta,\varepsilon}_{f_s'}\tau
e_m,\nabla^{T\mM,\beta,\varepsilon}_{f_t'}\tau e_k\right\rangle +
\left\langle p_1^\perp\nabla^{T\mM,\beta,\varepsilon}_{f_s'}\tau
e_m,\nabla^{T\mM,\beta,\varepsilon}_{f_t'}\tau e_k\right\rangle
\\
+
 {\varepsilon }^{2}\left\langle
p_2^\perp\nabla^{T\mM,\beta,\varepsilon}_{f_s'}\tau
e_m,\nabla^{T\mM,\beta,\varepsilon}_{f_t'}\tau e_k\right\rangle
\\
+\varepsilon^{2}f_s'\left(\left\langle\nabla^{T\mM,\beta,\varepsilon}_{f_t'}\tau
e_m, \tau e_k\right\rangle\right)
-\varepsilon^{2}f_t'\left(\left\langle\nabla^{T\mM,\beta,\varepsilon}_{f_s'}\tau
e_m, \tau e_k\right\rangle\right)-\varepsilon^{2}\left\langle
\nabla^{T\mM ,\beta,\varepsilon}_{[f_s',f_t']}\tau e_m, \tau
e_k\right\rangle +O_{\beta,\varepsilon}\left(|Z| \right)\\
 =
\left\langle p_1^\perp\nabla^{ T\mM,\beta, \varepsilon}_{f_t'}\tau
e_k, \nabla^{ T\mM,\beta, \varepsilon}_{f_s'}\tau e_m
\right\rangle -\left\langle \nabla^{ T\mM,\beta,
\varepsilon}_{f_s'}\tau e_k,p_1^\perp \nabla^{ T\mM,\beta,
\varepsilon}_{f_t'}\tau e_m\right\rangle +O\left(
\frac{\varepsilon^2}{\beta^2}\right)+O_{\beta,\varepsilon}\left(|Z|
\right) .
\end{multline}

\comment{

From (\ref{13.14a}), one gets that for $q+1\leq  s,\,t\leq q+q_1$, one has
\begin{multline}\label{13.14f}
 \left \langle R^{\mF_1^\perp,\beta, \varepsilon}(\tau
e_m,\tau e_k)  f_s', f_t'\right\rangle = \left
\langle R^{T\mM,\beta,
\varepsilon}(\tau
e_m,\tau e_k)  f_s', f_t'\right\rangle \\
+{\beta^2} {\varepsilon^{2}}\left\langle
p\nabla^{T\mM,\beta,\varepsilon}_{ \tau e_m}f_s',\nabla^{T\mM,\beta,\varepsilon}_{\tau e_k}f_t'\right\rangle
+
 {\varepsilon }^{2}\left\langle
p_2^\perp\nabla^{T\mM,\beta,\varepsilon}_{ \tau e_m}f_s',\nabla^{T\mM,\beta,\varepsilon}_{ \tau e_k}f_t'\right\rangle\\
- {\beta^2}{\varepsilon^{2}}\left\langle
p\nabla^{T\mM,\beta,\varepsilon}_{ \tau e_k}f_s',\nabla^{T\mM,\beta,\varepsilon}_{ \tau e_m}f_t'\right\rangle-
 {\varepsilon }^{2}\left\langle
p_2^\perp\nabla^{T\mM,\beta,\varepsilon}_{ \tau e_k}f_s',\nabla^{T\mM,\beta,\varepsilon}_{ \tau e_m}f_t'\right\rangle\\
=  \left\langle p_1^\perp\nabla^{ T\mM,\beta, \varepsilon}_{f_t'}\tau e_k, \nabla^{ T\mM,\beta, \varepsilon}_{f_s'}\tau e_m \right\rangle -\left\langle \nabla^{ T\mM,\beta, \varepsilon}_{f_s'}\tau e_k,p_1^\perp \nabla^{ T\mM,\beta, \varepsilon}_{f_t'}\tau e_m\right\rangle +O\left( \frac{\varepsilon^2}{\beta^2}\right)+O_{\beta,\varepsilon}\left(|Z|
\right) .
\end{multline}

}

From (\ref{12.6m}), (\ref{12.3f}),  (\ref{13.1}), (\ref{157}), (\ref{159a}), (\ref{13.11b})-(\ref{13.14a}) and the obvious equality $\int_{-\infty}^{+\infty}z^2e^{-z^2}dz=\frac{1}{2}\int_{-\infty}^{+\infty}e^{-z^2}dz$, one gets that for $1\leq i,\,j\leq q$ with $i\neq j$, 
\begin{multline}\label{14.11a}
\int_\mM\widetilde{\tau}
f_i\left(f_{T }\right) f_T \left\langle c_{\beta,\varepsilon}\left(\widetilde{\tau}  f_i\right) c_{\beta,\varepsilon}\left(\widetilde{\tau}
f_j\right) \tau\sigma, \,^Q {\nabla} ^{\mF, \beta,\varepsilon}_{\widetilde{\tau}  f_j}(\tau
\sigma)-\tau\left(\left.\,^Q {\nabla} ^{\mF, \beta,\varepsilon}_{\widetilde{\tau}  f_j}(\tau
\sigma)\right|_{s(M)}\right)
\right\rangle dv_\mM
\\
=\frac{1}{8\beta^2} \sum_{s,\,t=q+1}^{q+q_1}\int_{s(M)}\left\langle c_{\beta,\varepsilon}\left(\beta^{-1}f_i\right)c_{\beta,\varepsilon}\left(\beta^{-1}f_j\right)\sigma,\left(
 \left\langle p_1^\perp\nabla^{ T\mM,\beta, \varepsilon}_{f_t}\left(\nabla^{\mF_2^\perp}_{f_j}Z\right) , \nabla^{ T\mM,\beta, \varepsilon}_{f_s}\left(\nabla^{\mF_2^\perp}_{f_i}Z\right)  \right\rangle \right.\right.
\\
\left.\left.-\left\langle \nabla^{ T\mM,\beta, \varepsilon}_{f_s}\left(\nabla^{\mF^\perp_2}_{f_j}Z\right) ,p_1^\perp \nabla^{ T\mM,\beta, \varepsilon}_{f_t}\left(\nabla^{\mF^\perp_2}_{f_i}Z\right) \right\rangle \right)c_{\beta,\varepsilon}(\varepsilon f_s)c_{\beta,\varepsilon}(\varepsilon f_t)\sigma
 \right\rangle dv_{s(M)}
\\
 + \left( 
O\left(\frac{\varepsilon}{\beta^4}\right)+O\left(
\frac{1}{\sqrt{T}}\right)\right)\int_{s(M)}|\sigma|^2dv_{s(M)}
+ 
\sum_{k=q+1}^{q+q_1}O\left( \frac{\varepsilon^2}{\beta^2}\right)\int_{s(M)} |\sigma|\cdot\left|
\,^Q\nabla^{\mF, \beta,\varepsilon}_{ f_k}(\tau\sigma)\right| dv_{s(M)}
\\
+O\left(\frac{1}{\sqrt{T}}\right)\int_{s(M)}
\sum_{k=1}^{q+q_1}\left|\, ^Q\nabla ^{\mF,
 \beta,\varepsilon}_{
f_k}(\tau \sigma)\right|^2dv_{s(M)}.
\end{multline}

Set
\begin{multline}\label{14.11b}
{\mathcal W}
= \sum_{i,\,j=1}^q\sum_{s,\,t=q+1}^{q+q_1}\int_{s(M)}\left\langle 
 \left(
 \left\langle p_1^\perp\nabla^{ T\mM,\beta, \varepsilon}_{f_t}\left(\nabla^{\mF_2^\perp}_{f_j}Z\right) , \nabla^{ T\mM,\beta, \varepsilon}_{f_s}\left(\nabla^{\mF_2^\perp}_{f_i}Z\right)  \right\rangle \right. \right.
\\
 \left.\left. -\left\langle \nabla^{ T\mM,\beta, \varepsilon}_{f_s}\left(\nabla^{\mF^\perp_2}_{f_j}Z\right) ,p_1^\perp \nabla^{ T\mM,\beta, \varepsilon}_{f_t}\left(\nabla^{\mF^\perp_2}_{f_i}Z\right) \right\rangle \right)\sigma,c_{\beta,\varepsilon}\left(\beta^{-1}f_i\right)c_{\beta,\varepsilon}\left(\beta^{-1}f_j\right)c_{\beta,\varepsilon}(\varepsilon f_s)c_{\beta,\varepsilon}(\varepsilon f_t)\sigma
 \right\rangle dv_{s(M)}.
\end{multline}

From (\ref{i1}), (\ref{3.13}),  (\ref{114}), (\ref{117}),  (\ref{127c}), (\ref{117f}), (\ref{117c}),  (\ref{14.11a}) and (\ref{14.11b}),
one finds, 
\begin{multline}\label{1451}
 I_1 
=-\frac{{\mathcal W}}{8\beta^2}
 + \left( 
O\left(\frac{\varepsilon}{\beta^4}\right)+O\left(
\frac{1}{\sqrt{T}}\right)\right)\int_{s(M)}|\sigma|^2dv_{s(M)}
\\
+ 
O\left( {\varepsilon^3}\right)\int_{s(M)} \sum_{k=q+1}^{q+q_1} \left|
\,^Q\nabla^{\mF, \beta,\varepsilon}_{ f_k}(\tau\sigma)\right|^2 dv_{s(M)}
+O\left(\frac{1}{\sqrt{T}}\right)\int_{s(M)}
\sum_{i=1}^{q+q_1}\left|\, ^Q\nabla ^{\mF,\beta,
 \varepsilon}_{
f_i}(\tau \sigma)\right|^2dv_{s(M)}.
 \end{multline}

\begin{lemma}\label{t3i} There exists $C_2>0$ such that the following formula holds on $s(M)$:
\begin{align}\label{145x}
\sum_{i=1}^q\sum_{t=q+1}^{q+q_1}\left|p_1^\perp \nabla^{ T\mM,\beta, \varepsilon}_{f_t}\left(\nabla^{\mF_2^\perp}_{f_i}Z\right)\right|^2\leq C_2.
\end{align}
\end{lemma}
\begin{proof}
From (\ref{f6}) and (\ref{12.6m}), one gets (\ref{145x}). 
\end{proof}

From (\ref{12.6m}), (\ref{12.3f}),  (\ref{13.1}), (\ref{157}), (\ref{160a}), (\ref{160c}),  (\ref{13.11b})-(\ref{13.14a}) and (\ref{145x}), one gets that for $q+1\leq i,\,j\leq q+q_1$ with $i\neq j$, 
\begin{multline}\label{14.11c}
\int_\mM\widetilde{\tau}
f_i\left(f_{T }\right) f_T \left\langle c_{\beta,\varepsilon}\left(\widetilde{\tau}  f_i\right) c_{\beta,\varepsilon}\left(\widetilde{\tau}
f_j\right) \tau\sigma, \,^Q {\nabla} ^{\mF, \beta,\varepsilon}_{\widetilde{\tau}  f_j}(\tau
\sigma)-\tau\left(\left.\,^Q {\nabla} ^{\mF, \beta,\varepsilon}_{\widetilde{\tau}  f_j}(\tau
\sigma)\right|_{s(M)}\right)
\right\rangle dv_\mM
\\
=  \left(
O\left(\frac{ \varepsilon}{\beta^2}\right)+O\left(
\frac{1}{\sqrt{T}}\right)\right)\int_{s(M)}|\sigma|^2dv_{s(M)}
+ 
\sum_{k=1}^{q}O\left( \frac{\varepsilon^2}{\beta^2}\right)\int_{s(M)} |\sigma|\cdot\left|
\,^Q\nabla^{\mF, \beta,\varepsilon}_{ f_k}(\tau\sigma)\right| dv_{s(M)}
\\
+ 
\sum_{k=q+1}^{q+q_1}O\left( {\varepsilon^2} \right)\int_{s(M)} |\sigma|\cdot\left|
\,^Q\nabla^{\mF, \beta,\varepsilon}_{ f_k}(\tau\sigma)\right| dv_{s(M)}
+O\left(\frac{1}{\sqrt{T}}\right)\int_{s(M)}
\sum_{k=1}^{q+q_1}\left|\, ^Q\nabla ^{\mF,
 \beta,\varepsilon}_{
f_k}(\tau \sigma)\right|^2dv_{s(M)}.
\end{multline}

From (\ref{i2}), (\ref{3.13}),  (\ref{114}), (\ref{117}),  (\ref{127g}), (\ref{117f}), (\ref{117c})   and  (\ref{14.11c}), one gets
\begin{multline}\label{1452}
I_2
=   \left(
O\left(\frac{\beta+\varepsilon}{\beta^2}\right)+O\left(
\frac{1}{\sqrt{T}}\right)\right)\int_{s(M)}|\sigma|^2dv_{s(M)}
\\
+
\left(O\left( \frac{\varepsilon^2}{\beta^2}\right)+O\left(\frac{1}{\sqrt{T}}\right)\right)  \int_{s(M)} \sum_{i=1}^{q} \left|
\,^Q\nabla^{\mF, \beta,\varepsilon}_{ f_i}(\tau\sigma)\right|^2 dv_{s(M)}
\\
+ 
\left(O\left( {\varepsilon^3} \right) + O\left(\frac{1}{\sqrt{T}}\right)\right)\int_{s(M)} \sum_{k=q+1}^{q+q_1}\left|
\,^Q\nabla^{\mF, \beta,\varepsilon}_{ f_k}(\tau\sigma)\right|^2 dv_{s(M)}
 .
\end{multline}

From (\ref{12.6m}), (\ref{12.3f}),  (\ref{13.1}), (\ref{157}), (\ref{159a}), (\ref{160a}), (\ref{160c}),  (\ref{13.11b})-(\ref{13.14a}), (\ref{145x}) and the   equality $\int_{-\infty}^{+\infty}z^2e^{-z^2}dz=\frac{1}{2}\int_{-\infty}^{+\infty}e^{-z^2}dz$, one gets that for $1\leq i\leq q$ and $q+1\leq j\leq q+q_1$, 
\begin{multline}\label{14.11d}
\int_\mM\widetilde{\tau}
f_i\left(f_{T }\right) f_T \left\langle c_{\beta,\varepsilon}\left(\widetilde{\tau}  f_i\right) c_{\beta,\varepsilon}\left(\widetilde{\tau}
f_j\right) \tau\sigma, \,^Q {\nabla} ^{\mF, \beta,\varepsilon}_{\widetilde{\tau}  f_j}(\tau
\sigma)-\tau\left(\left.\,^Q {\nabla} ^{\mF, \beta,\varepsilon}_{\widetilde{\tau}  f_j}(\tau
\sigma)\right|_{s(M)}\right)
\right\rangle dv_\mM
\\
= \int_{s(M)} \left(
O\left(\frac {\varepsilon}{\beta^3} \right)+O\left(
\frac{1}{\sqrt{T}}\right)\right) |\sigma|^2dv_{s(M)}
+ 
\sum_{k=1}^{q}O\left( \frac{\varepsilon}{\beta^3}\right)\int_{s(M)} |\sigma|\cdot\left|
\,^Q\nabla^{\mF, \beta,\varepsilon}_{ f_k}(\tau\sigma)\right| dv_{s(M)}
\\
+\frac{\varepsilon
}{2\beta} \int_{s(M)} \left\langle
c_{\beta,\varepsilon}\left(\beta^{-1}
f_i\right)c_{\beta,\varepsilon}\left(\varepsilon f_j\right)
\sigma,\,^Q{\nabla} ^{\mF,  \beta,\varepsilon}_{p_1^\perp
\nabla^{ T\mM,\beta,
\varepsilon}_{f_j}\left(\nabla^{\mF_2^\perp}_{f_i}Z\right)}(\tau
\sigma)\right\rangle  dv_{s(M)}
\\
+O\left(\frac{1}{\sqrt{T}}\right)\int_{s(M)}
\sum_{k=1}^{q+q_1}\left|\, ^Q\nabla ^{\mF,
 \beta,\varepsilon}_{
f_k}(\tau \sigma)\right|^2dv_{s(M)},
\end{multline}
while for $q+1\leq i\leq q+q_1$ and $1\leq j\leq q$, one has
\begin{multline}\label{14.11e}
\int_\mM\widetilde{\tau}
f_i\left(f_{T }\right) f_T \left\langle c_{\beta,\varepsilon}\left(\widetilde{\tau}  f_i\right) c_{\beta,\varepsilon}\left(\widetilde{\tau}
f_j\right) \tau\sigma, \,^Q {\nabla} ^{\mF, \beta,\varepsilon}_{\widetilde{\tau}  f_j}(\tau
\sigma)-\tau\left(\left.\,^Q {\nabla} ^{\mF, \beta,\varepsilon}_{\widetilde{\tau}  f_j}(\tau
\sigma)\right|_{s(M)}\right)
\right\rangle dv_\mM
\\
= \int_{s(M)} \left(
O\left(\frac {\varepsilon}{\beta^3} \right)+O\left(
\frac{1}{\sqrt{T}}\right)\right) |\sigma|^2dv_{s(M)}
+ 
\sum_{k=q+1}^{q+q_1}O\left( \frac{\varepsilon^3}{\beta}\right)\int_{s(M)} |\sigma|\cdot\left|
\,^Q\nabla^{\mF, \beta,\varepsilon}_{ f_k}(\tau\sigma)\right| dv_{s(M)}
\\
+O\left(\frac{1}{\sqrt{T}}\right)\int_{s(M)}
\sum_{k=1}^{q+q_1}\left|\, ^Q\nabla ^{\mF,
 \beta,\varepsilon}_{
f_k}(\tau \sigma)\right|^2dv_{s(M)}.
\end{multline}

From (\ref{i4}),  (\ref{3.13}),  (\ref{114}), (\ref{117}), (\ref{127f}),  (\ref{117f}),  (\ref{117c}),  (\ref{14.11d}) and  (\ref{14.11e}), one gets
\begin{multline}\label{1454a}
  I_4  
=
\frac{\varepsilon
}{\beta}\sum_{i=1}^q\sum_{t=q+1}^{q+q_1}\int_{s(M)}{\rm Re}\left(\left\langle
c_{\beta,\varepsilon}\left(\beta^{-1}
f_i\right)c_{\beta,\varepsilon}\left(\varepsilon f_t\right)
\sigma,\,^Q{\nabla} ^{\mF,  \beta,\varepsilon}_{p_1^\perp
\nabla^{ T\mM,\beta,
\varepsilon}_{f_t}\left(\nabla^{\mF_2^\perp}_{f_i}Z\right)}(\tau
\sigma)\right\rangle\right)  dv_{s(M)}
\\
+
O\left(\frac {\varepsilon}{\beta^4}\right)  \int_{s(M)}|\sigma|^2dv_{s(M)}
+ 
O\left( \frac{\varepsilon}{\beta^2}\right)\int_{s(M)} \sum_{i=1}^{q}\left|
\,^Q\nabla^{\mF, \beta,\varepsilon}_{ f_i}(\tau\sigma)\right|^2 dv_{s(M)}
\\
+ 
 O\left( {\varepsilon^3} \right) \int_{s(M)}\sum_{k=q+1}^{q+q_1}\left|
\,^Q\nabla^{\mF, \beta,\varepsilon}_{ f_k}(\tau\sigma)\right|^2 dv_{s(M)}
\\
+O\left(\frac{1}{\sqrt{T}}\right)\int_{s(M)} \left(
|\sigma|^2+\sum_{k=1}^{q+q_1}\left|
\,^Q\nabla^{\mF, \beta,\varepsilon}_{
f_k}(\tau\sigma)\right|^2\right) dv_{s(M)}.
\end{multline}


\subsection{Proof of Proposition \ref{t12.4}}\label{s2.8}

From (\ref{145})-(\ref{1456}), (\ref{1451}), (\ref{1452})  and  (\ref{1454a}), one has
\begin{multline}\label{145z}
\sum_{i=1}^6  I_i  
=-\frac{\mathcal W}{8\beta^2}+ O\left(\frac{1}{\beta}+\frac {\varepsilon}{\beta^4}\right)  \int_{s(M)}|\sigma|^2dv_{s(M)}+
O\left( \frac{\varepsilon}{\beta^2}\right)\int_{s(M)} \sum_{i=1}^{q}\left|
\,^Q\nabla^{\mF, \beta,\varepsilon}_{ f_i}(\tau\sigma)\right|^2 dv_{s(M)}
\\
+\frac{\varepsilon
}{\beta}\sum_{i=1}^q\sum_{t=q+1}^{q+q_1}\int_{s(M)}{\rm Re}\left(\left\langle
c_{\beta,\varepsilon}\left(\beta^{-1}
f_i\right)c_{\beta,\varepsilon}\left(\varepsilon f_t\right)
\sigma,\,^Q{\nabla} ^{\mF,  \beta,\varepsilon}_{p_1^\perp
\nabla^{ T\mM,\beta,
\varepsilon}_{f_t}\left(\nabla^{\mF_2^\perp}_{f_i}Z\right)}(\tau
\sigma)\right\rangle \right) dv_{s(M)}
\\
+ 
  O\left( {\varepsilon^3}  \right)\int_{s(M)}\sum_{k=q+1}^{q+q_1}\left|
\,^Q\nabla^{\mF, \beta,\varepsilon}_{ f_k}(\tau\sigma)\right|^2 dv_{s(M)}
\\
+O\left(\frac{1}{\sqrt{T}}\right)\int_{s(M)} \left(
|\sigma|^2+\sum_{k=1}^{q+q_1}\left|
\,^Q\nabla^{\mF, \beta,\varepsilon}_{
f_k}(\tau\sigma)\right|^2\right) dv_{s(M)}.
\end{multline}

From (\ref{m2}), (\ref{3.1}), (\ref{m55}), (\ref{3.8a}) and
(\ref{3.12}), one deduces that
\begin{multline}\label{167}
 \left\|p_{T,\beta,\varepsilon}D^{\mF, \beta,\varepsilon}
J_{T,\beta,\varepsilon}\sigma\right\|_0^2\geq
\left\langle\left(\frac{k^\mF}{4\beta^2}+O\left(\frac{1}{\beta}+\frac{\varepsilon^2}{\beta^2}\right)\right )J_{T,\beta,\varepsilon}\sigma
,J_{T,\beta,\varepsilon}\sigma\right\rangle
\\
+\frac{1}{\beta^2}\sum_{i=1}^q\left\|p_{T,\beta,\varepsilon}\nabla^{\mF, \beta,\varepsilon}_{\tau
f_i}J_{T,\beta,\varepsilon}\sigma\right\|^2_0
+\varepsilon^{2}\sum_{i=q+1}^{q+q_1}\left\|p_{T,\beta,\varepsilon}\nabla^{\mF, \beta,\varepsilon}_{\tau
f_i}J_{T,\beta,\varepsilon}\sigma\right\|^2_0
\\
+ \sum_{i=1}^{q_2}\left\|
\nabla^{\mF, \beta,\varepsilon}_{\tau
e_i}J_{T,\beta,\varepsilon}\sigma\right\|^2_0- \sum_{i=1}^{q_2}\left\|\left(1-p_{T,\beta,\varepsilon}\right)c_{\beta,\varepsilon}(\tau
e_i) \nabla^{\mF, \beta,\varepsilon}_{\tau
e_i}J_{T,\beta,\varepsilon}\sigma\right\|^2
 -\sum_{k=1}^6I_k.
\end{multline}

Clearly, for any $1\leq i\leq q_2$, one has
\begin{multline}\label{168}
   \left\|
\nabla^{\mF, \beta,\varepsilon}_{\tau
e_i}J_{T,\beta,\varepsilon}\sigma\right\|^2_0-
 \left\|\left(1-p_{T,\beta,\varepsilon}\right)c_{\beta,\varepsilon}(\tau e_i)
\nabla^{\mF, \beta,\varepsilon}_{\tau
e_i}J_{T,\beta,\varepsilon}\sigma\right\|^2_0 \\
\geq  \left\| \nabla^{\mF, \beta,\varepsilon}_{\tau
e_i}J_{T,\beta,\varepsilon}\sigma\right\|^2_0-
 \left\| c_{\beta,\varepsilon}(\tau e_i)
\nabla^{\mF, \beta,\varepsilon}_{\tau
e_i}J_{T,\beta,\varepsilon}\sigma\right\|^2_0=0.
\end{multline}

From (\ref{167}) and (\ref{168}), one gets
\begin{multline}\label{168a}
  \left\|p_{T,\beta,\varepsilon}D^{\mF, \beta,\varepsilon}
J_{T,\beta,\varepsilon}\sigma\right\|_0^2\geq
\left\langle\left(\frac{k^\mF}{4\beta^2}+O\left(\frac{1}{\beta}+\frac{\varepsilon^2}{\beta^2}\right)\right )J_{T,\beta,\varepsilon}\sigma
,J_{T,\beta,\varepsilon}\sigma\right\rangle
\\
+\frac{1}{\beta^2}\sum_{i=1}^q\left\|p_{T,\beta,\varepsilon}\nabla^{\mF, \beta,\varepsilon}_{\tau
f_i}J_{T,\beta,\varepsilon}\sigma\right\|^2_0
+\varepsilon^{2}\sum_{i=q+1}^{q+q_1}\left\|p_{T,\beta,\varepsilon}\nabla^{\mF, \beta,\varepsilon}_{\tau
f_i}J_{T,\beta,\varepsilon}\sigma\right\|^2_0
 -\sum_{k=1}^6I_k.
\end{multline}

Now since for any $U,\, V\in\Gamma(\mF_1^\perp)$, $W\in\Gamma(\mF_2^\perp)$, one has
\begin{align}\label{14.11f}
\left\langle\nabla^{T\mM,\beta,\varepsilon}_UW,V\right\rangle - \left\langle\nabla^{T\mM,\beta,\varepsilon}_VW,U\right\rangle=-\varepsilon^2\left\langle W,\nabla^{T\mM,\beta,\varepsilon}_UV-\nabla^{T\mM,\beta,\varepsilon}_VU\right\rangle  =-\varepsilon^2\left\langle W,[U,V]\right\rangle,
\end{align}
it is easy to verify that for $1\leq i\leq q$,  
\begin{multline}\label{168b}
\sum_{t=q+1}^{q+q_1} \int_{s(M)}{\rm Re}\left(\left\langle
c_{\beta,\varepsilon}\left(\beta^{-1}
f_i\right)c_{\beta,\varepsilon}\left(\varepsilon f_t\right)
\sigma,\,^Q{\nabla} ^{\mF,  \beta,\varepsilon}_{p_1^\perp
\nabla^{ T\mM,\beta,
\varepsilon}_{f_t}\left(\nabla^{\mF_2^\perp}_{f_i}Z\right)}(\tau
\sigma)\right\rangle \right) dv_{s(M)}
\\
=\sum_{t=q+1}^{q+q_1}\int_{s(M)}{\rm Re}\left(\left\langle
c_{\beta,\varepsilon}\left(\beta^{-1}
f_i\right)c_{\beta,\varepsilon}\left(\varepsilon p_1^\perp
\nabla^{ T\mM,\beta,
\varepsilon}_{f_t}\left(\nabla^{\mF_2^\perp}_{f_i}Z\right)\right)
\sigma,\,^Q{\nabla} ^{\mF,  \beta,\varepsilon}_{f_t}(\tau
\sigma)\right\rangle\right)  dv_{s(M)}
\\
+O\left(\varepsilon^2\right)
\sum_{t=q+1}^{q+q_1}
\int_{s(M)}\left|
\sigma\right|\cdot\left|^Q{\nabla} ^{\mF,  \beta,\varepsilon}_{f_t}(\tau
\sigma)\right|  dv_{s(M)}
\\
=\sum_{t=q+1}^{q+q_1}{\rm Re}\left(\left\langle J_{T,\beta,\varepsilon}
c_{\beta,\varepsilon}\left(\beta^{-1}
f_i\right)c_{\beta,\varepsilon}\left(\varepsilon p_1^\perp
\nabla^{ T\mM,\beta,
\varepsilon}_{f_t}\left(\nabla^{\mF_2^\perp}_{f_i}Z\right)\right)
\sigma,\,^Q{\nabla} ^{\mF,  \beta,\varepsilon}_{\tau f_t} J_{T,\beta,\varepsilon}
\sigma\right\rangle \right)
\\
+ O\left( \frac{\varepsilon^2}{\beta}\right)  \int_{s(M)}|\sigma|^2dv_{s(M)}
+ 
  O\left(\beta\, {\varepsilon^2}  \right)\int_{s(M)}\sum_{k=q+1}^{q+q_1}\left|
\,^Q\nabla^{\mF, \beta,\varepsilon}_{ f_k}(\tau\sigma)\right|^2 dv_{s(M)}
\\
+O\left(\frac{1}{\sqrt{T}}\right)\int_{s(M)} \left(
|\sigma|^2+\sum_{k=1}^{q+q_1}\left|
\,^Q\nabla^{\mF, \beta,\varepsilon}_{
f_k}(\tau\sigma)\right|^2\right) dv_{s(M)}.
\end{multline}

Also, by the obvious equality  $|a+b|^2=|a|^2+|b|^2+2\,{\rm Re}(\langle a,b\rangle)$, one has, for any $q+1\leq t\leq q+q_1$,
\begin{multline}\label{168c}
 \left\|\varepsilon \, p_{T,\beta,\varepsilon}\nabla^{\mF, \beta,\varepsilon}_{\tau
f_t}J_{T,\beta,\varepsilon}\sigma   -  \frac{1}{2\beta}\sum_{i=1}^q  J_{T,\beta,\varepsilon}
c_{\beta,\varepsilon}\left(\beta^{-1}
f_i\right)c_{\beta,\varepsilon}\left(\varepsilon p_1^\perp
\nabla^{ T\mM,\beta,
\varepsilon}_{f_t}\left(\nabla^{\mF_2^\perp}_{f_i}Z\right)\right)
\sigma      \right\|^2_0
\\
=\varepsilon^2  \left\|p_{T,\beta,\varepsilon}\nabla^{\mF, \beta,\varepsilon}_{\tau
f_t}J_{T,\beta,\varepsilon}\sigma  \right\|^2_0 +\frac{1}{4\beta^2} \left\|\sum_{i=1}^q c_{\beta,\varepsilon}\left(\beta^{-1}
f_i\right)c_{\beta,\varepsilon}\left(\varepsilon p_1^\perp
\nabla^{ T\mM,\beta,
\varepsilon}_{f_t}\left(\nabla^{\mF_2^\perp}_{f_i}Z\right)\right)
\sigma \right\|^2_0
\end{multline}
$$
-\frac{\varepsilon}{\beta}
\sum_{i=1}^q{\rm Re}\left(\left\langle J_{T,\beta,\varepsilon}
c_{\beta,\varepsilon}\left(\beta^{-1}
f_i\right)c_{\beta,\varepsilon}\left(\varepsilon p_1^\perp
\nabla^{ T\mM,\beta,
\varepsilon}_{f_t}\left(\nabla^{\mF_2^\perp}_{f_i}Z\right)\right)
\sigma,\,^Q{\nabla} ^{\mF,  \beta,\varepsilon}_{\tau f_t} J_{T,\beta,\varepsilon}
\sigma\right\rangle \right)
,
$$
with the following pointwise formula on $s(M)$, where again (\ref{14.11f}) is used,
\begin{multline}\label{168d}
\sum_{t=q+1}^{q+q_1}\left( \sum_{i=1}^q c_{\beta,\varepsilon}\left(\beta^{-1}
f_i\right)c_{\beta,\varepsilon}\left(\varepsilon p_1^\perp
\nabla^{ T\mM,\beta,
\varepsilon}_{f_t}\left(\nabla^{\mF_2^\perp}_{f_i}Z\right)\right)
 \right)^2=-\sum_{i=1}^q \sum_{t=q+1}^{q+q_1}\left|p_1^\perp
\nabla^{ T\mM,\beta,
\varepsilon}_{f_t}\left(\nabla^{\mF_2^\perp}_{f_i}Z\right) \right|^2
\\
-\frac{1}{2}\sum_{i,\,j=1}^q \sum_{t=q+1}^{q+q_1} c_{\beta,\varepsilon}\left(\beta^{-1}
f_i\right)c_{\beta,\varepsilon}\left(\beta^{-1}
f_j\right)\left( c_{\beta,\varepsilon}\left(\varepsilon p_1^\perp
\nabla^{ T\mM,\beta,
\varepsilon}_{f_t}\left(\nabla^{\mF_2^\perp}_{f_i}Z\right)\right)c_{\beta,\varepsilon}\left(\varepsilon p_1^\perp
\nabla^{ T\mM,\beta,
\varepsilon}_{f_t}\left(\nabla^{\mF_2^\perp}_{f_j}Z\right)\right)\right.
\end{multline}
$$
\left. -c_{\beta,\varepsilon}\left(\varepsilon p_1^\perp
\nabla^{ T\mM,\beta,
\varepsilon}_{f_t}\left(\nabla^{\mF_2^\perp}_{f_j}Z\right)\right)c_{\beta,\varepsilon}\left(\varepsilon p_1^\perp
\nabla^{ T\mM,\beta,
\varepsilon}_{f_t}\left(\nabla^{\mF_2^\perp}_{f_i}Z\right)\right)\right)
$$
$$
=-\sum_{i=1}^q \sum_{t=q+1}^{q+q_1}\left|p_1^\perp
\nabla^{ T\mM,\beta,
\varepsilon}_{f_t}\left(\nabla^{\mF_2^\perp}_{f_i}Z\right) \right|^2
- \frac{1}{2}\sum_{i,\,j=1}^q \sum_{s,\,t=q+1}^{q+q_1} c_{\beta,\varepsilon} \left( \beta^{-1}
f_i\right)c_{\beta,\varepsilon} \left( \beta^{-1}
f_j\right) c_{\beta,\varepsilon}(\varepsilon f_s)c_{\beta,\varepsilon}(\varepsilon f_t)
$$
$$
\left(\left\langle p_1^\perp
\nabla^{ T\mM,\beta,
\varepsilon}_{f_s}\left(\nabla^{\mF_2^\perp}_{f_i}Z\right),  
\nabla^{ T\mM,\beta,
\varepsilon}_{f_t}\left(\nabla^{\mF_2^\perp}_{f_j}Z\right)\right\rangle
-\left\langle p_1^\perp
\nabla^{ T\mM,\beta,
\varepsilon}_{f_s}\left(\nabla^{\mF_2^\perp}_{f_j}Z\right),  
\nabla^{ T\mM,\beta,
\varepsilon}_{f_t}\left(\nabla^{\mF_2^\perp}_{f_i}Z\right)\right\rangle\right)
$$
$$
+O\left(\varepsilon^2\right)
.
$$

From (\ref{14.11b}), (\ref{145z}),  (\ref{168a}) and (\ref{168b})-(\ref{168d}), we get 
\begin{multline}\label{168e}
  \left\|p_{T,\beta,\varepsilon}D^{\mF, \beta,\varepsilon}
J_{T,\beta,\varepsilon}\sigma\right\|_0^2\geq
\left\langle\left(\frac{k^\mF}{4\beta^2}+O\left(\frac{1}{\beta}+\frac{\varepsilon^2}{\beta^2}\right)\right )J_{T,\beta,\varepsilon}\sigma
,J_{T,\beta,\varepsilon}\sigma\right\rangle
\\
+\frac{1}{\beta^2}\sum_{i=1}^q\left\|p_{T,\beta,\varepsilon}\nabla^{\mF, \beta,\varepsilon}_{\tau
f_i}J_{T,\beta,\varepsilon}\sigma\right\|^2_0
\end{multline}
$$
+\sum_{t=q+1}^{q+q_1} \left\|\varepsilon\, p_{T,\beta,\varepsilon}\nabla^{\mF, \beta,\varepsilon}_{\tau
f_t}J_{T,\beta,\varepsilon}\sigma   -  \frac{1}{2\beta}\sum_{i=1}^q  J_{T,\beta,\varepsilon}
c_{\beta,\varepsilon}\left(\beta^{-1}
f_i\right)c_{\beta,\varepsilon}\left(\varepsilon p_1^\perp
\nabla^{ T\mM,\beta,
\varepsilon}_{f_t}\left(\nabla^{\mF_2^\perp}_{f_i}Z\right)\right)
\sigma      \right\|^2_0
$$
$$
-\frac{1}{4\beta^2}\sum_{i=1}^q \sum_{t=q+1}^{q+q_1}\int_{s(M)}\left|p_1^\perp
\nabla^{ T\mM,\beta,
\varepsilon}_{f_t}\left(\nabla^{\mF_2^\perp}_{f_i}Z\right) \right|^2\cdot |\sigma|^2dv_{s(M)}
$$
$$
+ O\left(\frac{1}{\beta}+\frac {\varepsilon}{\beta^4}\right)  \int_{s(M)}|\sigma|^2dv_{s(M)}+
O\left( \frac{\varepsilon}{\beta^2}\right)\int_{s(M)} \sum_{i=1}^{q}\left|
\,^Q\nabla^{\mF, \beta,\varepsilon}_{ f_i}(\tau\sigma)\right|^2 dv_{s(M)}
$$
$$
+ 
  O\left( {\varepsilon^3}  \right)\int_{s(M)}\sum_{k=q+1}^{q+q_1}\left|
\,^Q\nabla^{\mF, \beta,\varepsilon}_{ f_k}(\tau\sigma)\right|^2 dv_{s(M)}
$$
$$
+O\left(\frac{1}{\sqrt{T}}\right)\int_{s(M)} \left(
|\sigma|^2+\sum_{k=1}^{q+q_1}\left|
\,^Q\nabla^{\mF, \beta,\varepsilon}_{
f_k}(\tau\sigma)\right|^2\right) dv_{s(M)}.
$$

For $1\leq i\leq q+q_1$, by (\ref{m22}) and
(\ref{3.9})-(\ref{11.1}), one has,
\begin{multline}\label{170}
p_{T,\beta,\varepsilon}\nabla^{\mF, \beta,\varepsilon}_{\tau
f_i}J_{T,\beta,\varepsilon}\sigma=p_{T,\beta,\varepsilon}\left(\tau f_i
\left(f_{T }\right)\tau\sigma
+f_{T }\nabla^{\mF, \beta,\varepsilon}_{\tau f_i}(\tau\sigma)\right)
 \\
= \left(\int_{\mM_x} f_{T }\tau
f_i\left(f_{T }\right)k\,dv_{\mM_x}\right)J_{T,\beta,\varepsilon}\sigma+
 p_{T,\beta,\varepsilon} \left(f_{T }Q \nabla^{\mF, \beta,\varepsilon}_{\tau f_i}(\tau \sigma)\right) .
\end{multline}

From (\ref{11.2}) and Lemma \ref{t11.4}, one deduces that the
following formula holds for any $1\leq i\leq q+q_1$,
\begin{multline}\label{171}
\left\|p_{T,\beta,\varepsilon} \left(f_{T }Q
\nabla^{\mF, \beta,\varepsilon}_{\tau f_i}(\tau
\sigma)\right)\right\|_0^2=\int_{s(M)}\left|\,
^Q\nabla^{\mF, \beta,\varepsilon}_{  f_i}(\tau \sigma
)\right|^2d{v}_{s(M)}\\
+O\left(\frac{1}{\sqrt{T}}\right)\int_{s(M)}|\sigma|^2d{v}_{s(M)} +O\left(\frac{1}{\sqrt{T}}\right)
\sum_{j=1}^{q+q_1}\int_{s(M)}\left|\,
^Q\nabla^{\mF, \beta,\varepsilon}_{  f_j}(\tau \sigma )\right|^2d{v}_{s(M)}.
\end{multline}

If $1\leq i\leq q$, by (\ref{12.1}) and (\ref{119}), one gets
\begin{align}\label{172}
\int_{\mM_x} f_{T }\tau
f_i\left(f_{T }\right)k\,dv_{\mM_x}=
O\left(1\right)+O\left(\frac{1}{\sqrt{T}}\right).
\end{align}

If $q+1\leq i\leq q+q_1$, by (\ref{12.2}) and (\ref{119}),  one
gets
\begin{align}\label{173}
 \int_{\mM_x} f_{T }\tau
f_i\left(f_{T }\right)k\,dv_{\mM_x}=
O\left( {\frac{1 }{\beta^2}}\right)+O\left(\frac{1}{\sqrt{T}}\right).
\end{align}

Recall the following obvious inequality,
\begin{align}\label{173a}
 |a+b|^2\geq \frac{|a|^2}{2}-|b|^2.
\end{align}

 By (\ref{145x}) and
(\ref{170})-(\ref{173a}), one gets that for   $0<\delta\leq 1$
sufficiently small,
\begin{multline}\label{174}
\frac{1}{\beta^2} \sum_{i=1}^q\left\|p_{T,\beta,\varepsilon}\nabla^{\mF, \beta,\varepsilon}_{\tau
f_i}J_{T,\beta,\varepsilon}\sigma\right\|^2_0
\\
+ \sum_{i=q+1}^{q+q_1} \left\|\varepsilon \, p_{T,\beta,\varepsilon}\nabla^{\mF, \beta,\varepsilon}_{\tau
f_t}J_{T,\beta,\varepsilon}\sigma   -  \frac{1}{2\beta}\sum_{i=1}^q  J_{T,\beta,\varepsilon}
c_{\beta,\varepsilon}\left(\beta^{-1}
f_i\right)c_{\beta,\varepsilon}\left(\varepsilon p_1^\perp
\nabla^{ T\mM,\beta,
\varepsilon}_{f_t}\left(\nabla^{\mF_2^\perp}_{f_i}Z\right)\right)
\sigma      \right\|^2_0
\\
\geq
 \sum_{i=1}^q\frac{\varepsilon^{\delta}}{\beta^2}\left\|p_{T,\beta,\varepsilon}\nabla^{\mF, \beta,\varepsilon}_{\tau
f_i}J_{T,\beta,\varepsilon}\sigma\right\|^2_0 
\\
+ \varepsilon^{\delta
} \sum_{i=q+1}^{q+q_1} \left\|\varepsilon\, p_{T,\beta,\varepsilon}\nabla^{\mF, \beta,\varepsilon}_{\tau
f_t}J_{T,\beta,\varepsilon}\sigma   -  \frac{1}{2\beta}\sum_{i=1}^q  J_{T,\beta,\varepsilon}
c_{\beta,\varepsilon}\left(\beta^{-1}
f_i\right)c_{\beta,\varepsilon}\left(\varepsilon p_1^\perp
\nabla^{ T\mM,\beta,
\varepsilon}_{f_t}\left(\nabla^{\mF_2^\perp}_{f_i}Z\right)\right)
\sigma      \right\|^2_0 \\
\geq
\int_{s(M)}\left(O\left( \frac{\varepsilon^{\delta}}{\beta^4}\right) +O\left(\frac{1}{\sqrt{T}}\right)\right)|\sigma|^2dv_{s(M)}
+\frac{\varepsilon^{\delta}}{4\beta^2}\sum_{i=1}^q\int_{s(M)}\left|^Q\nabla^{\mF, \beta,\varepsilon}_{  f_i}(\tau
\sigma)\right|^2dv_{s(M)}\\ +\frac{\varepsilon^{2+\delta
}}{8}\sum_{i=q+1}^{q+q_1}\int_{s(M)}\left|^Q\nabla^{\mF, \beta,\varepsilon}_{  f_i}
(\tau\sigma)\right|^2dv_{s(M)}
+O\left(\frac{1}{\sqrt{T}}\right)\sum_{i=1}^{q+q_1}\int_{s(M)}\left|^Q\nabla^{\mF, \beta,\varepsilon}_{  f_i}
(\tau\sigma)\right|^2dv_{s(M)}.
\end{multline}

From (\ref{168e})  and (\ref{174}),  one deduces that
\begin{multline}\label{175}
 \left\|p_{T,\beta,\varepsilon}D^{\mF, \beta,\varepsilon}
J_{T,\beta,\varepsilon}\sigma\right\|_0^2
\geq
\int_{s(M)} \left(\frac{k^\mF}{4\beta^2}- 
\frac{1}{4\beta^2}\sum_{i=1}^q \sum_{t=q+1}^{q+q_1} \left|p_1^\perp
\nabla^{ T\mM,\beta,
\varepsilon}_{f_t}\left(\nabla^{\mF_2^\perp}_{f_i}Z\right) \right|^2\right) |\sigma|^2dv_{s(M)}
\\
+O\left(\frac{1}{\beta}+ \frac{\varepsilon^\delta }{\beta^4}\right)   \int_{s(M)}|\sigma|^2dv_{s(M)}+
 \left(\frac{\varepsilon^\delta }{4\beta^2}+O\left(\frac{\varepsilon  }{\beta^2}\right)\right) \sum_{k=1}^q\int_{s(M)}  \left|
\,^Q\nabla^{\mF, \beta,\varepsilon}_{ f_k}(\tau\sigma)\right|^2 dv_{s(M)}
\\
+\left(\frac{\varepsilon^{2+\delta} }{8 } +O\left( {\varepsilon^3  }{ }\right)\right)\sum_{k=q+1}^{q+q_1}\int_{s(M)}  \left|
\,^Q\nabla^{\mF, \beta,\varepsilon}_{ f_k}(\tau\sigma)\right|^2 dv_{s(M)}
\\
+O\left(\frac{1}{\sqrt{T}}\right)\int_{s(M)} \left(
|\sigma|^2+\sum_{k=1}^{q+q_1}\left|
\,^Q\nabla^{\mF, \beta,\varepsilon}_{
f_k}(\tau\sigma)\right|^2\right) dv_{s(M)}.
\end{multline}

From  (\ref{175}),  one gets (\ref{12.6}).

The proof of Proposition \ref{t12.4} is completed.

\subsection{Proof of Theorem \ref{t0.2}}\label{s2.9}

 Since the metric $g^\mF=\pi^*g^F$ is lifted from $g^F$, for any $x\in \mM$, one has
\begin{align}\label{175a}
k^\mF(x)=k^F(\pi(x)).
\end{align}

\begin{lemma}\label{t175} For any $X\in\Gamma(s_*F)$, $U,\,V\in \Gamma(\mF_1^\perp|_{s(M)})$, one has,
\begin{align}\label{175b}
\left\langle\nabla^{T\mM,\beta,\varepsilon}_UX,V\right\rangle_{g^{\mF_1^\perp}}=\frac{\left\langle\omega\left(\pi_*X\right)\pi_*U,\pi_*V \right\rangle_{g^{F^\perp}}}{2}+ O\left(\varepsilon^2 \right) .
\end{align}
\end{lemma}

\begin{proof} Without loss of generality, we assume that $s_*F^\perp=\mF^\perp_1|_{s(M)}$. Then $s_*F\subseteq (\mF\oplus\mF_2^\perp)|_{s(M)}$ is orthogonal to $s_*F^\perp$ with respect to $g^{T\mM}|_{s(M)}$.

We first fix a $\beta_0>0$ and compute by using (\ref{1.7})  that
\begin{align}\label{175d}
\left\langle\nabla^{T\mM,\beta,\varepsilon}_UX,V\right\rangle =  \left\langle\nabla^{T\mM,\beta_0,\varepsilon}_UX,V\right\rangle  +O\left( \varepsilon^2 \right)  .
\end{align}

Let $g^{TM}_{\beta_0,\varepsilon}=s^*g^{T\mM}_{\beta_0,\varepsilon}$ be the induced metric on $TM$. Then one has
$
g^{TM}_{\beta_0,\varepsilon} =g^F_{\beta_0}+\frac{g^{F^\perp}}{\varepsilon^2}  , $ with $g^F_{\beta_0}$ does not depend on $\varepsilon$. 
Let $\nabla^{TM,\beta_0,\varepsilon}$ denote the associated Levi-Civita connection. Then one has (cf. (\ref{a1.4}))
\begin{align}\label{175f}
  \left\langle\nabla^{T\mM,\beta_0,\varepsilon}_UX,V\right\rangle  =  \left\langle\nabla^{TM,\beta_0,\varepsilon}_{\pi_*U}\pi_*X,\pi_*V\right\rangle =  \frac{\left\langle\omega\left(\pi_*X\right)\pi_*U,\pi_*V \right\rangle }{2}+ O\left(\varepsilon^2 \right)  .
\end{align}

From (\ref{175d}) and (\ref{175f}), one gets (\ref{175b}). 
\end{proof}

Now let $\widehat f\in \Gamma(F)$, $U\in \Gamma(\mF_1^\perp|_{s(M)})$. Denote $f=(\pi^*\widehat f)|_{s(M)}\in\Gamma(\mF|_{s(M)})$. Then one has on $s(M)$ that
\begin{align}\label{175g}
f=\left(f-s_*\widehat f \right)+s_*\widehat f  ,
\end{align}
with $f-s_*\widehat f \in\Gamma(\mF_2^\perp|_{s(M)})$, as $\pi_*(f-s_*\widehat f )=\widehat f-\widehat f=0$.

Thus, as $Z\equiv 0$ on $s(M)$  (cf. (\ref{12.6mb})),   the following identity holds on $s(M)$,
\begin{align}\label{175h}
 \nabla^{\mF_2^\perp}_fZ=\nabla^{\mF_2^\perp}_{f-s_*\widehat f}Z= f-s_*\widehat f.
\end{align}

From (\ref{f4}), (\ref{175b})  and (\ref{175h}), one finds
\begin{multline}\label{175i}
\pi_*\left(\left. \left( p_1^\perp  \nabla^{T\mM,\beta,\varepsilon}_U\left( \nabla^{\mF_2^\perp}_fZ\right)\right) \right|_{s(M)}\right)
=\pi_*\left(   p_1^\perp  \nabla^{T\mM,\beta,\varepsilon}_U\left(  f -s_*\widehat f\right)\right)  
\\
=-\frac{1}{2} \omega\left(\widehat f\right)\pi_*U +O\left(\varepsilon^2\right) .
\end{multline}

Let $\widehat f_1,\  \,\cdots,\,\widehat f_q$ be an orthonormal basis of $(F,g^F)$; $h_1,\,\cdots,\,h_{q_1}$  an orthornormal basis of $(F^\perp,g^{F^\perp})$. 

By (\ref{175a}) and (\ref{175i}), (\ref{12.6}) in Proposition \ref{t12.4} now takes the form
\begin{multline}\label{175j}
 \left\|p_{T,\beta,\varepsilon}D^{\mF, \beta,\varepsilon}
J_{T,\beta,\varepsilon}\sigma\right\|_0^2
\geq
 \int_{s(M)} \left(\frac{k^F }{4\beta^2}-
\frac{1}{16\beta^2}\sum_{i=1}^q \sum_{s=1}^{q_1} \left| \omega\left(\widehat f_i\right)h_s \right|^2\right)|\sigma|^2dv_{s(M)}
\\
- C'\left(\frac{1}{\beta}+\frac{\varepsilon^\delta}{\beta^4} \right)  \int_{s(M)}|\sigma|^2dv_{s(M)}
+  \frac{\varepsilon^{\delta} }{8\beta^2} \sum_{k=1}^q\int_{s(M)}  \left|
\,^Q\nabla^{\mF, \beta,\varepsilon}_{ f_k}(\tau\sigma)\right|^2 dv_{s(M)}
\\
+ \frac{\varepsilon^{2+\delta} }{16 } \sum_{k=q+1}^{q+q_1}\int_{s(M)}  \left|
\,^Q\nabla^{\mF, \beta,\varepsilon}_{ f_k}(\tau\sigma)\right|^2 dv_{s(M)}
-\frac{C_{\beta,\varepsilon}}{\sqrt{T}} \int_{s(M)} \left(
|\sigma|^2+\sum_{k=1}^{q+q_1}\left|
\,^Q\nabla^{\mF, \beta,\varepsilon}_{
f_k}(\tau\sigma)\right|^2\right) dv_{s(M)}.
\end{multline} 

Let $\widehat D_{s(M)}^{ \beta,\varepsilon}:\Gamma((S(\mF\oplus \mF_1^\perp))|_{s(M)})  \rightarrow \Gamma((S(\mF\oplus \mF_1^\perp))|_{s(M)}) $ be the limit operator 
\begin{align}\label{175ja}
 \widehat D_{s(M)}^{ \beta,\varepsilon} =\lim_{T\rightarrow +\infty} J_{T,\beta,\varepsilon}^{-1}p_{T,\beta,\varepsilon}D^{\mF, \beta,\varepsilon}
J_{T,\beta,\varepsilon} .
\end{align}

The existence of the limit   is clear. Also, 
one verifies easily that $\widehat D_{s(M)}^{ \beta,\varepsilon}$ is a formally self-adjoint (with respect to the inner product in (\ref{k2}))
 Dirac type operator.\footnote{Since in general $(\mF\oplus\mF_1^\perp)|_{s(M)}\neq Ts(M)$ geometrically, here by a Dirac type operator we mean   that its symbol is homotopic, through invertible elements,  to that of a standard Dirac operator.} homotopic through a family of  Moreover,  for any $\sigma\in \Gamma((S(\mF\oplus \mF_1^\perp))|_{s(M)})$, one has by   (\ref{175j}) that 
\begin{multline}\label{175k}
 \left\| \widehat D_{s(M)}^{ \beta,\varepsilon} \sigma\right\|_0^2
\geq
 \int_{s(M)} \left(\frac{k^F }{4\beta^2}-
\frac{1}{16\beta^2}\sum_{i=1}^q \sum_{s=1}^{q_1} \left| \omega\left(\widehat f_i\right)h_s \right|^2\right)|\sigma|^2dv_{s(M)}
\\
- C'\left(\frac{1}{\beta}+\frac{\varepsilon^\delta}{\beta^4} \right)  \int_{s(M)}|\sigma|^2dv_{s(M)}
+  \frac{\varepsilon^{\delta} }{8\beta^2} \sum_{k=1}^q\int_{s(M)}  \left|
\,^Q\nabla^{\mF, \beta,\varepsilon}_{\pi^*\widehat f_k} \sigma\right|^2 dv_{s(M)}
\\
+ \frac{\varepsilon^{2+\delta} }{16 } \sum_{t=1}^{q_1}\int_{s(M)}  \left|
\,^Q\nabla^{\mF, \beta,\varepsilon}_{\pi^*h_t} \sigma\right|^2 dv_{s(M)}
 .
\end{multline} 

Theorem \ref{t0.2}  follows from (\ref{175k}) easily. 

\begin{rem}\label{t175aa}
 The above proof assumes that $F$ and $F^\perp\simeq TM/F$ are oriented, which is needed in the construction of the Connes fibration. When $F^\perp$ is not orientable,  one can pass to the double covering of $M$, with respect to $w_1(TM/F)$ (the first Stiefel-Whitney class of $TM/F$), to complete the proof. 
\end{rem}


For a more concrete form of $\widehat D_{s(M)}^{ \beta,\varepsilon}$, let $\widetilde D_{s(M)}^{\beta,\varepsilon}:\Gamma((S(\mF\oplus \mF_1^\perp))|_{s(M)})  \rightarrow \Gamma((S(\mF\oplus \mF_1^\perp))|_{s(M)}) $ be defined by that for any $\sigma\in  \Gamma((S(\mF\oplus \mF_1^\perp))|_{s(M)}) $, 
\begin{align}\label{175m}
 \widetilde D_{s(M)}^{\beta,\varepsilon} \sigma =\left.\left(\frac{1}{\beta}\sum_{i=1}^qc_{\beta,\varepsilon} \left(\beta^{-1}\pi^*\widehat f_i\right){ ^Q\nabla}^{\mF,\beta,\varepsilon}_{\pi^*\widehat f_i} (\tau\sigma)+\varepsilon\sum_{t=1}^{q_1}c_{\beta,\varepsilon}\left(\varepsilon\,\pi^*h_t\right) {^Q\nabla}^{\mF,\beta,\varepsilon}_{\pi^*  h_t} (\tau\sigma)\right)\right|_{s(M)} ,
\end{align}
which by (\ref{12.24}) could be written as 
\begin{align}\label{175ma}
 \widetilde D_{s(M)}^{\beta,\varepsilon} = \frac{1}{\beta}\sum_{i=1}^qc_{\beta,\varepsilon}\left(\beta^{-1}\pi^*\widehat f_i\right){ ^Q\nabla}^{\mF,\beta,\varepsilon}_{s_*\widehat f_i}  +\varepsilon\sum_{t=1}^{q_1}c_{\beta,\varepsilon}\left(\varepsilon\,\pi^*h_t\right) {^Q\nabla}^{\mF,\beta,\varepsilon}_{s_*  h_t}  ,
\end{align}
which is clearly of Dirac type.

By (\ref{m22}), (\ref{11.2c}), (\ref{11.3}), (\ref{12.30}) and (\ref{175ja}), one sees directly that there exists $Y _{\beta,\varepsilon}\in\Gamma((\mF\oplus\mF_1^\perp)|_{s(M)})$ such that
\begin{align}\label{181}
 \widehat D_{s(M)}^{ \beta,\varepsilon} =
  \widetilde D_{s(M)}^{\beta,\varepsilon} +c_{\beta,\varepsilon}\left(Y_{\beta,\varepsilon}\right)   .
\end{align}

Let $(\widetilde D_{s(M)}^{\beta,\varepsilon})^*$ be the formal adjoint of $\widetilde D_{s(M)}^{\beta,\varepsilon}$ with respect to the inner product in (\ref{k2}).   

From (\ref{181}), one gets

\begin{thm}\label{t14.11} The following identity holds,
\begin{align}\label{175n}
 \widehat D_{s(M)}^{ \beta,\varepsilon} 
=\frac{1}{2}\left( \widetilde D_{s(M)}^{\beta,\varepsilon}  + \left (\widetilde D_{s(M)}^{\beta,\varepsilon}\right)^*\right).
\end{align}
\end{thm}

\appendix

\section{Adiabatic limit and the scalar curvature on a foliation}\label{sA}

In this Appendix, we summarise the computation of the adiabatic limit of the scalar curvature on an arbitrary  foliation carried out in \cite{LZ01} and \cite{LW09}.

Let $(M,F)$ be a  foliated manifold.  We take the orthogonal splitting as in (\ref{14.0.1}).  Let $p^\perp:TM\rightarrow F^\perp$ be the corresponding orthogonal projection. 

For any $\varepsilon>0$, 
let $g_\varepsilon^{TM}$ be the Riemannian metric on $TM$ such that
\begin{align}\label{a1.1}
  g^{TM}_\varepsilon=g^F\oplus \frac{g^{F^\perp}}{\varepsilon^2}. 
\end{align}
Let $\nabla^{TM,\varepsilon}$ be the associated Levi-Civita connection.

\comment{

For any $X\in\Gamma(F)$, let $\omega(X)\in{\rm End}(F^\perp)$ be defined as in (\ref{14.0.1b}). 

For $X\in\Gamma(F)$, $U,\ V\in\Gamma(F^\perp)$, set
\begin{align}\label{a1.3}
\langle \omega(X)U,V\rangle=X\langle U,V\rangle -\langle [X,U],V\rangle-\langle U,[X,V]\rangle.
\end{align}
It is clear that  $\langle\omega(X)U,V\rangle$ does not depend on the rescaling of $g^{F^\perp}$, if one rescals $U, \ V$ accordingly like unit vectors. 

}

For any $X\in\Gamma(F)$, let $\omega(X)\in\Gamma({\rm End}(F^\perp))$ be defined as in (\ref{14.0.1b}). 
Then for any $U\in\Gamma(F^\perp)$ one has (cf. \cite[(1.7) and (1.13)]{LZ01} and \cite[(2.6)]{LW09})
\begin{align}\label{a1.4}
\frac{1}{2} \omega(X)U=\lim_{\varepsilon\rightarrow 0}p^\perp \nabla^{TM,\varepsilon}_UX.
\end{align}

Let $f_1,\  \,\cdots,\,f_q$ be an orthonormal basis of $(F,g^F)$; $h_1,\,\cdots,\,h_{q_1}$  an orthornormal basis of $(F^\perp,g^{F^\perp})$. In what follows, we assume $X,\ Y$ are    of $f_i$'s, while $U,\ V$ are of $h_s$'s.

Set
\begin{align}\label{a1.5}
 |\omega(X)U|^2= \sum_{s=1}^{q_1}   |\langle \omega(X)U,h_s\rangle|^2,\ \ \  \  |\omega(X)|^2= \sum_{s,\, t=1}^{q_1}   |\langle \omega(X)h_t,h_s\rangle|^2
= \sum_{s=1}^{q_1}   |  \omega(X)h_s|^2.
\end{align}

It is easy to verify that 
\begin{align}\label{a1.6}
\left \langle R^{TM,\varepsilon}(X,Y)X,Y\right\rangle =\left \langle R^{F}(X,Y)X,Y\right\rangle +O\left(\varepsilon^2\right). 
\end{align}

Also, by (\ref{a1.4}), one has
\begin{multline}\label{a1.7}
\varepsilon^2 \left\langle R^{TM,\varepsilon}(U,V)U,V\right\rangle = \varepsilon^2\left(\langle  \nabla^{TM,\varepsilon}_U \nabla^{TM,\varepsilon}_VU,V\rangle - \langle  \nabla^{TM,\varepsilon}_V \nabla^{TM,\varepsilon}_UU,V\rangle-\langle  \nabla^{TM,\varepsilon}_{[U,V]}  U,V\rangle\right)
\\
=-\varepsilon^4\left\langle \nabla^{TM,\varepsilon}_VU, p\nabla^{TM,\varepsilon}_UV\right\rangle + \varepsilon^4\left\langle \nabla^{TM,\varepsilon}_UU, p\nabla^{TM,\varepsilon}_VV\right\rangle +O\left(\varepsilon^2\right)
\\
=-\frac{1}{4} \sum_{i=1}^q|\langle\omega(f_i)U,V\rangle|^2+\frac{1}{4} \sum_{i=1}^q\langle\omega(f_i)U,U\rangle \langle\omega(f_i)V,V\rangle +O\left(\varepsilon^2\right),
\end{multline}
and
\begin{multline}\label{a1.8}
 \left \langle R^{TM,\varepsilon}(X,U)X,U\right\rangle =\left  \langle \nabla^{TM,\varepsilon}_X\left(p+p^\perp\right) \nabla^{TM,\varepsilon}_UX,U\right\rangle -\left  \langle \nabla^{TM,\varepsilon}_U\left(p+p^\perp\right) \nabla^{TM,\varepsilon}_XX,U\right\rangle
\end{multline}
$$
-\left \langle \nabla^{TM,\varepsilon}_{[X,U] }X,U\right\rangle
$$
$$
= \left \langle \nabla^{TM,\varepsilon}_X p^\perp\nabla^{TM,\varepsilon}_UX,U\right\rangle -\left \langle \nabla^{TM,\varepsilon}_Up \nabla^{TM }_XX,U\right\rangle
-\left \langle \nabla^{TM,\varepsilon}_{p^\perp[X,U] }X,U\right\rangle+O\left(\varepsilon^2\right)
$$
$$
= \left\langle \nabla^{TM,\varepsilon}_X p^\perp\nabla^{TM,\varepsilon}_UX,U\right\rangle 
-\frac{1}{2}\left\langle  \omega(X)U,p^\perp[X,U]\right\rangle-\frac{1}{2} \left\langle\omega\left (p \nabla^{TM }_XX\right)U,U\right\rangle+O\left(\varepsilon^2\right)
$$
$$
=\left\langle \nabla^{TM,\varepsilon}_X p^\perp\nabla^{TM,\varepsilon}_UX,U\right\rangle 
-\frac{1}{2}\left\langle  \omega(X)U,p^\perp\nabla^{TM,\varepsilon}_XU\right\rangle+\frac{1}{4} | \omega(X)U |^2\\
 -\frac{1}{2} \left\langle\omega\left (p \nabla^{TM }_XX\right)U,U\right\rangle+O\left(\varepsilon^2\right)
$$
$$
=X\left\langle p^\perp\nabla^{TM,\varepsilon}_UX,U\right\rangle-\left\langle p^\perp\nabla^{TM,\varepsilon}_UX,\nabla^{TM,\varepsilon}_XU\right\rangle 
-\frac{1}{2}\left\langle  \omega(X)U,p^\perp\nabla^{TM,\varepsilon}_XU\right\rangle
$$
$$
+\frac{1}{4} | \omega(X)U |^2 -\frac{1}{2} \left\langle\omega\left (p \nabla^{TM }_XX\right)U,U\right\rangle+O\left(\varepsilon^2\right)
$$
$$
=\frac{1}{2} X\left\langle \omega(X)U,U\right\rangle- \left\langle  \omega(X)U,p^\perp\nabla^{TM,\varepsilon}_XU\right\rangle
+\frac{1}{4} | \omega(X)U |^2 -\frac{1}{2} \left\langle\omega\left (p \nabla^{TM }_XX\right)U,U\right\rangle+O\left(\varepsilon^2\right)
$$
$$
=\frac{1}{2}X\left\langle \omega(X)U,U\right\rangle- \left\langle  \omega(X)U,p^\perp[X,U]\right\rangle
-\frac{1}{4} | \omega(X)U |^2 -\frac{1}{2} \left\langle\omega\left (p \nabla^{TM }_XX\right)U,U\right\rangle+O\left(\varepsilon^2\right).
$$

Recall that the leafwise scalar curvature $k^F$ (associated to $g^F$) has been defined in (\ref{1.23aa}). Let $k^{TM,\varepsilon}$ be the scalar curvature associated to $g^{TM}_\varepsilon$. The following formula gives the adiabatic limit ($\varepsilon\rightarrow 0$) behaviour of $k^{TM,\varepsilon}$. 

From   (\ref{a1.6})-(\ref{a1.8}), one finds
\begin{multline}\label{a1.9}
   k^{TM,\varepsilon}=-\sum_{i,\,j=1}^q\left \langle R^{TM,\varepsilon}(f_i,f_j)f_i,f_j\right\rangle -\varepsilon^2\sum_{s,\,t=1}^{q_1}\left \langle R^{TM,\varepsilon}(h_s,h_t)h_s,h_t\right\rangle
\\ - 2 \sum_{i=1}^q \sum_{s=1}^{q_1}\left\langle R^{TM,\varepsilon}(f_i,h_s)f_i,h_s\right\rangle
\\
=k^F+ \frac{1}{4} \sum_{i=1}^q\sum_{s,\,t=1}^{q_1} \left(\left|\left\langle\omega(f_i)h_s,h_t\right\rangle\right|^2-
\langle\omega(f_i)h_s,h_s\rangle \langle\omega(f_i)h_t,h_t\rangle\right)-\sum_{i=1}^q\sum_{s=1}^{q_1} f_i\left(\left\langle \omega(f_i)h_s,h_s\right\rangle\right)
\end{multline}
$$
+2\sum_{i=1}^q\sum_{s=1}^{q_1}  \left\langle  \omega(f_i)h_s,p^\perp[f_i,h_s]\right\rangle
+\frac{1}{2}\sum_{i=1}^q\sum_{s=1}^{q_1}  | \omega(f_i)h_s |^2
+\sum_{i=1}^q\sum_{s=1}^{q_1} \left\langle\omega\left (p \nabla^{TM }_{f_i}f_i\right)h_s,h_s\right\rangle+O\left(\varepsilon^2\right)
$$
$$
=k^F+\frac{3}{4}\sum_{i=1}^q |\omega(f_i)|^2 -\frac{1}{4}\sum_{i=1}^q\left(\sum_{s=1}^{q_1} \langle\omega(f_i)h_s,h_s\rangle \right)^2
-\sum_{i=1}^q\sum_{s=1}^{q_1} f_i\left(\left\langle \omega(f_i)h_s,h_s\right\rangle\right)
$$
$$
+2\sum_{i=1}^q\sum_{s=1}^{q_1}  \left\langle  \omega(f_i)h_s,p^\perp[f_i,h_s]\right\rangle
 +\sum_{i=1}^q\sum_{s=1}^{q_1} \left\langle\omega\left (p \nabla^{TM }_{f_i}f_i\right)h_s,h_s\right\rangle+O\left(\varepsilon^2\right).
$$
 
\begin{rem}
If $q_1=1$, that is, if $(M,F)$ is a codimension one foliation, one finds
\begin{align}\label{a1.10}
  p^\perp[f_i,h_1]=p^\perp \nabla^{TM,\varepsilon}_{f_i}h_1-p^\perp \nabla^{TM,\varepsilon}_{h_1}f_i=-p^\perp \nabla^{TM,\varepsilon}_{h_1}f_i. 
\end{align}
Then (\ref{a1.9}) becomes,
\begin{align}\label{a1.11}
   k^{TM,\varepsilon} =k^F -\frac{1}{2}\sum_{i=1}^q |\omega(f_i)|^2 
-\sum_{i=1}^q  f_i\left(\left\langle \omega(f_i)h_1,h_1\right\rangle\right)+\sum_{i=1}^q \left\langle\omega\left (p \nabla^{TM }_{f_i}f_i\right)h_1,h_1\right\rangle
 +O\left(\varepsilon^2\right).
\end{align}
In this case, if one assumes $M$ is spin and takes the   Dirac operator $D_\varepsilon$ associated to $g^{TM}_\varepsilon$, then by the Lichnerowicz formula \cite{L63} and (\ref{a1.11}), one deduces that
 \begin{multline}\label{a1.12}
 D_\varepsilon ^2=\sum_{i=1}^q\left(\nabla^{ \varepsilon}_{f_i}+\frac{1}{4}\left\langle \omega(f_i)h_1,h_1\right\rangle\right)^*\left(\nabla^{ \varepsilon}_{f_i}+\frac{1}{4}\left\langle \omega(f_i)h_1,h_1\right\rangle\right)+\varepsilon^2\left(\nabla^{ \varepsilon}_{h_1}\right) ^* \nabla^{ \varepsilon}_{h_1} 
\\
+\frac{1}{4}\left(k^F-\frac{3}{4}\sum_{i=1}^q|\omega(f_i)|^2\right)+O\left(\varepsilon^2\right)
,
\end{multline}
where $\nabla^\varepsilon$ is the canonical  connection on the corresponding spinor bundle, which implies that
 \begin{align}\label{a1.12a}
 D_\varepsilon ^2\geq \frac{1}{4}\left(k^F-\frac{3}{4}\sum_{i=1}^q|\omega(f_i)|^2\right)+O\left(\varepsilon^2\right)
.
\end{align}
\end{rem}

\comment{

\subsection{Proof of Theorem \ref{t0.8}}\label{s2.12}

From our positivity  result (\ref{175d}), which can well be used to
replace the Lichnerowicz positivity used in \cite[Proof of
Theorem 2.1]{GL80}, one can proceed as in  \cite{GL80} to get the
same conclusion of \cite[Theorem 2.1]{GL80} and \cite[Corollary
2.2]{GL80} under the condition of Theorem \ref{t0.8}. In particular, Theorem \ref{t0.8} holds.

 We leave  the details and other immediate generalizations  to interested readers.

\begin{rem}\label{t32}  By a theorem of Bourguignon (cf. \cite[Theorem 5.7 of Chapter  IV]{LaMi89}),
 if $F=T(T^n)$, one can improve  the corresponding result of Schoen-Yau \cite{SY79}
 and Gromov-Lawson \cite{GL80a} by stating that any Riemannian metric on $T^n$ with
 non-negative scalar curvature is Ricci flat. It is natural to ask whether there is
 an analogue of  Bourguignon's result for foliations. If one could prove such an analogue,
  then by combining with Theorem \ref{t0.8} one would solve one of the conjectures
  stated in \cite[Conjectures C13]{MS88}.
\end{rem}

}

\def\cprime{$'$} \def\cprime{$'$}
\providecommand{\bysame}{\leavevmode\hbox to3em{\hrulefill}\thinspace}
\providecommand{\MR}{\relax\ifhmode\unskip\space\fi MR }
\providecommand{\MRhref}[2]{%
  \href{http://www.ams.org/mathscinet-getitem?mr=#1}{#2}
}
\providecommand{\href}[2]{#2}

\end{document}